\documentclass[11pt,reqno]{amsart}
\usepackage{}
\usepackage{amssymb}
\usepackage{amsfonts}
\usepackage{txfonts}
\usepackage{mathrsfs}
\usepackage{latexsym}
\usepackage{graphicx}
\usepackage{amscd,amssymb,amsmath,amsbsy,amsthm,amsfonts}
\usepackage[all]{xy}
\usepackage[colorlinks,plainpages,urlcolor=blue]{hyperref} 
\usepackage{verbatim}
\newcommand {\ttrig}{{\mathfrak t}_{\trig}^\Phi}

\hypersetup{
    colorlinks=true,       			
    linkcolor=blue, 
    citecolor=blue,        			
    filecolor=blue,      				
    urlcolor=blue           			
}

\usepackage{tikz}
\usetikzlibrary{arrows,calc}
\tikzset{
>=stealth',
help lines/.style={dashed, thick},
axis/.style={<->},
important line/.style={thick},
connection/.style={thick, dotted},
}

\newtheorem{theorem}{Theorem}[section]
\newtheorem{lemma}[theorem]{Lemma}

\newtheorem{corollary}[theorem]{Corollary}
\newtheorem{prop}[theorem]{Proposition}

\theoremstyle{definition}
\newtheorem{definition}[theorem]{Definition}

\newtheorem{remark}[theorem]{Remark}

\newtheorem{problem}[theorem]{Problem}
\newtheorem{conj}[theorem]{Conjecture}

\newtheorem{thm}{Theorem}

\newcommand{\N}{\mathbb{N}}
\newcommand{\Z}{\mathbb{Z}}
\newcommand{\Q}{\mathbb{Q}}
\newcommand{\R}{\mathbb{R}}
\newcommand{\C}{\mathbb{C}}
\newcommand{\M}{\mathcal {M}}
\newcommand{\E}{\mathcal {E}}

\newcommand {\IE}{\mathbb E}
\newcommand {\IF}{\mathbb F}
\newcommand {\IH}{\mathbb H}

\newcommand{\inj}{\hookrightarrow}
\newcommand {\dd}{\mathfrak d}
\newcommand{\h}{\mathfrak h}
\newcommand{\n}{\mathfrak n}
\newcommand{\g}{\mathfrak g}

\newcommand{\gl}{\mathfrak{gl}}
\newcommand{\sO}{\mathcal{O}}
\newcommand {\aand}{\qquad\text{and}\qquad}

\DeclareMathOperator{\Int}{{int}}
\DeclareMathOperator{\ad}{ad}
\DeclareMathOperator{\Hom}{Hom}
\DeclareMathOperator{\End}{End}

\DeclareMathOperator{\GL}{GL}

\DeclareMathOperator{\SL}{SL}
\DeclareMathOperator{\tor}{tor}
\DeclareMathOperator{\Der}{Der}

\DeclareMathOperator{\reg}{reg}

\DeclareMathOperator{\Aff}{Aff}
\DeclareMathOperator{\diag}{diag}
\DeclareMathOperator{\Hol}{Hol}
\DeclareMathOperator{\Lie}{Lie}
\DeclareMathOperator{\U}{U}

\DeclareMathOperator{\Ell}{Ell}
\DeclareMathOperator{\tr}{tr}
\DeclareMathOperator{\trig}{trig}

\DeclareMathOperator{\KZB}{KZB}

\DeclareMathOperator{\Diff}{Diff}

\DeclareMathOperator{\rank}{rank}

\newcommand {\Omit}[1]{}
\newcommand {\tb}[1]{\textcolor{blue}{#1}}
\newcommand {\bcomment}[1]{\footnote{\tb{#1}}}
\newcommand {\rcomment}[1]{\footnote{\textcolor{red}{#1}}}

\newcommand {\wh}[1]{\widehat{#1}}

\newcommand {\elp}{{\scriptscriptstyle{\operatorname{ell}}}}
\newcommand {\Aell}{{\mathfrak t}_\elp^\Phi}
\newcommand {\Aelln}[1]{{\mathfrak t}_\elp^{\sfA_{#1-1}}}
\newcommand {\AAell}[1]{{\mathfrak t}_{1,#1}}

\renewcommand {\Im}{\operatorname{Im}}
\newcommand {\rreg}{_{\scriptscriptstyle{\reg}}}

\newenvironment{romenum}
{

\begin{enumerate}}{\end{enumerate}}

\newcommand {\CEE}{Calaque--Enriquez--Etingof }

\newcommand {\sfA}{\mathsf A}
\newcommand {\IR}{\mathbb R}

\newcommand {\IC}{\mathbb C}
\newcommand {\wrt}{with respect to }

\newcommand {\fd}{finite--dimensional }
\newcommand {\hreg}{\h_{\reg}}
\newcommand {\ie}{{\it i.e. }}

\newcommand {\Ug}{U\g}
\newcommand {\Uhg}{U_\hbar\g}
\newcommand {\Yhg}{Y_\hbar\g}
\renewcommand {\sl}[1]{\mathfrak{sl}_{#1}}
\newcommand {\gla}[1]{\mathfrak{gl}_{#1}}
\newcommand {\Treg}{T_{\reg}}
\newcommand {\DDg}{D_{\lambda}(\g)}
\newcommand {\IN}{{\mathbb N}}

\newcommand {\redtext}[1]{#1}

\newcommand {\bfs}{\mathbf{s}}

\newcommand {\ol}[1]{\overline{#1}}

\title[The elliptic Casimir connection of a simple Lie algebra ]%
{The elliptic Casimir connection of a simple Lie algebra}
\author[V.~Toledano Laredo]{Valerio~Toledano Laredo}
\address{Department of Mathematics,
Northeastern University,
360 Huntington Ave.,
Boston, MA 02115, USA}
\email{V.ToledanoLaredo@neu.edu}

\author[Y.~Yang]{Yaping~Yang}
\address{School of Mathematics and Statistics, The University of Melbourne
813 Swanston Street, Parkville VIC 3010}
\email{yaping.yang1@unimelb.edu.au}
\thanks{The first author was supported in part through the NSF grant DMS--1505305.}

\newcommand{\yaping}[1]{}

\begin{document}
\begin{abstract}
We construct a flat connection on the elliptic configuration space associated 
to any complex semisimple Lie algebra $\g$. This elliptic Casimir connection
has logarithmic singularities, 
and takes values in the deformed double current algebra of $\g$ 
defined by Guay \cite{G1,G2}. It degenerates to the trigonometric Casimir
connection of $\g$ constructed by the first author in \cite{TL-JAlg}. By analogy
with the rational and trigonometric cases, we conjecture that the monodromy
of the elliptic Casimir connection is described by the quantum Weyl group
operators of the quantum toroidal algebra of $\g$.
\end{abstract}
\maketitle
\tableofcontents

\newpage
\section{Introduction}
\Omit{Contents
1. Theta functions
2. General form of the connection
3. Proof of flatness
4. Degeneration to trigonometric connection
5. Extension to moduli space (of what?)
6. RCA incarnation
7. DDCA incarnation
}

\subsection{Motivation}

Around 1990, Drinfeld proved that the $R$ matrix of quantum groups describes
the monodromy of the Knizhnik--Zamolodchikov (KZ) equations \cite{D}. Subsequently,
Millson--Toledano Laredo \cite{MTL, TL-Duke}, and independently De Concini
(unpublished, 1995) and G. Felder \textit{et al.} \cite{FMTV} introduced another
flat connection $\nabla_C$, the \textit{Casimir connection} of a complex, semisimple
Lie algebra $\g$. The latter connection is distinct from, but dual to the KZ connection,
and is described as follows.

Let $\h$ be a Cartan subalgebra of $\g$, $\Phi\subset\h^*$ the corresponding root
system, and $\hreg\subset\h$ the complement of the root hyperplanes in $\h$. For
a \fd $\g$--module $V$, the Casimir connection $\nabla_C$ is the connection on the
holomorphically trivial vector bundle on $\hreg$ with fibre $V$ given by
\[
\nabla_C=d-\hbar \sum_{\alpha \in \Phi^+}\frac{d\alpha}{\alpha}\kappa_\alpha,
\]
where $\hbar\in\C$ is a deformation parameter, the summation is over a chosen
system $\Phi^+$ of positive roots, and $\kappa_\alpha\in U\sl{2}^{\alpha}
\subset \Ug$ is the truncated (\ie Cartan--less) Casimir operator of the $\mathfrak
{sl}_2^{\alpha}$--subalgebra of $\g$ corresponding to the root $\alpha$. This connection
is flat and equivariant with respect to the Weyl group $W$, and therefore gives rise
to a one--parameter family of monodromy representations of the generalized braid
group $B_\g=\pi_1(\h_{\reg}/W)$ on $V$.

In this setting, a Drinfeld--Kohno theorem was obtained by the first author \cite{TL-Duke,
TL-Adv, TL-IMRN, TL-qCqT}, according to which the monodromy of the Casimir
connection of $\g$ is described by the quantum Weyl group operators \cite{Lu} of the quantum
group $\Uhg$. This result was recently generalised to an arbitrary symmetrisable Kac--Moody
algebra by Appel--Toledano Laredo \cite{ATL1-1,ATL1-2,ATL2,ATL3}.

In related work, the first author constructed a trigonometric version of $\nabla_C$ \cite{TL-JAlg}.
Let $G$ be the simply--connected complex Lie group with Lie algebra $\g$, $H\subset G$ the
maximal torus with Lie algebra $\h$, and $H_{\reg}=H\setminus \bigcup_{\alpha \in \Phi}\{e^
{\alpha} =1\}$ the complement of the root hypertori in $H$. The trigonometric Casimir connection
$\nabla_{\trig, C}$ is a connection on the trivial vector bundle $H_{\reg}\times V$, where
the fiber $V$ is a \fd representation of the {\it Yangian} $\Yhg$, which is a deformation
of the current algebra $U(\g[s])$ of $\g$. Let $\{u_i\}$, and $\{u^i\}$ be
dual bases of $\h^*$ and $\h$ respectively. Then, $\nabla_{\trig, C}$ is given by
\[
\nabla_{\trig,C}=d-
\hbar
\sum_{\alpha\in\Phi_+}\frac{d\alpha}{e^{\alpha}-1}\kappa_\alpha-
du_i\,X(u^i)
\]
\Omit{
\begin{equation}\label{intro:trigC}
\nabla_{\trig, C}=d-\frac{\hbar}{2}\sum_{\alpha\in\Phi_+}\frac{e^\alpha+1}{e^\alpha-1}d\alpha\,
\kappa_\alpha
+2du_i\,J(u^i)
\end{equation} is flat and $W$-equivariant,
}
where $du_i$ is regarded as a translation invariant 1--form on $H$, $X: \h \to \Yhg
$ is a linear map such that $X(u)\equiv u\otimes s \mod\hbar$, and the summation
over $i$ is implicit.

The connection $\nabla_{\trig, C}$ is flat and $W$-equivariant. Its monodromy yields a one
parameter family of monodromy representations of the affine braid group $B_\g^{\Aff}=\pi_1(H
_{\reg}/W)$. 
By analogy with \cite{TL-Duke}, Toledano Laredo conjectured that the monodromy of $\nabla_
{\trig, C}$ is described by the quantum Weyl group operators of the quantum loop algebra $U_
\hbar(L\g)$ \cite{TL-JAlg}. This conjecture was formulated more precisely in subsequent work
of Gautam--Toledano Laredo, and reduced to the case of $\g=\sl{2}$ \cite{GTL-Selecta,GTL3}.
Moreover, for $\sl{2}$, it was proved for the tensor product of evaluation modules in \cite{GTL-sl2}.
Work in progress of Bezrukavnikov--Okounkov also proves this conjecture for representations
of $\Yhg$ arising from geometry \cite{BezOk}.

The main goal of the present paper is to construct an elliptic analogue of the Casimir
connection, and to give a conjectural description of its monodromy.

\subsection{The universal KZB connection of a root system}

The construction of the elliptic Casimir connection relies on the universal elliptic connection
associated to an arbitrary finite (reduced, crystallographic) root system $\Phi$ obtained in
\cite{TLY1}, which we review below.

\Omit{
In \cite{CEE}, \CEE constructed a universal Knizhnik--Zamolodchikov--Bernard (KZB) connection, which is a flat connection on $\mathfrak{M}_{1, n}$, the moduli space of genus 1 curves with $n$ marked points. This construction is generalized in \cite{TLY1} to an arbitrary root system $\Phi$ associated to a complex simple Lie algebra $\g$.
}


Let $Q\subset \h^*$ and $Q^\vee\subset \h$ be the root and coroot lattices of
$\Phi$ respectively.
Let $\tau$ be a point in the upper half plane $\mathfrak{H}=\{z \in \C \mid
\Im(z)> 0\}$, and set $\Lambda_\tau=\Z+\Z\tau\subset \C$. Consider the
elliptic curve $\mathcal{E}_\tau:=\C/{\Lambda_\tau}$ with modular parameter
$\tau$. Set $T=\h/(Q^\vee+\tau Q^\vee)$.
\Omit{
\footnote{In \cite{TLY1}, the universal KZB connection is defined on both the torus $\h/(Q^\vee+\tau Q^\vee)$ of simply connected type and 
$\h/(P^\vee+\tau P^\vee)$ of adjoint type, where $P^{\vee}$ is the coweight lattice that is dual to the root lattice $Q$. 
The choice $\h/(P^\vee+\tau P^\vee)$ is crucial for the formality result \cite[Theorem 5.2]{TLY1} in \cite{TLY1}. 
In the current paper, we focus on the universal KZB connection defined on $\h/(Q^\vee+\tau Q^\vee)$ .}
}
Any root $\alpha
\in\Phi$ induces a map $\chi_\alpha: T \to\E_\tau$, with kernel $T_\alpha$.
We refer to  $\Treg=T\setminus \bigcup_{\alpha\in \Phi}T_{\alpha}$, as
the \textit{elliptic configuration space} associated to $\Phi$.
The fundamental group $\pi_{1}(\Treg/W)$ is the elliptic braid group. 

Let $\theta(z| \tau )$ be the Jacobi theta function,  which is a holomorphic function $\C\times
\mathfrak{H} \to \C$, whose zero set is $\{z\mid\theta(z| \tau )=0\}=\Lambda_\tau$ and such
that its residue at $z=0$ is $1$ (see Section \S\ref{theta function}). Let $x$ be another complex
variable, and set
\[k(z, x|\tau):=\frac{\theta(z+x| \tau)}{\theta(z| \tau)\theta(x| \tau)}-\frac{1}{x}.\]
The function $k(z, x|\tau)$ has only simple poles at $z\in \Lambda_\tau$, and is regular
near $x=0$. It may therefore be regarded as an element of $1+x\Hol(\C-\Lambda_\tau)[\![x]\!]$. 

Let $A$ be an algebra endowed with the following data: a set of elements $\{t_{\alpha}\}
_{\alpha\in\Phi}$, such that  $t_{-\alpha}=t_{\alpha}$, and two linear maps $x: \h\to A$,
$y:\h\to A$. Consider the following $A$--valued meromorphic connection on $\h$.\footnote
{we are assuming further that $A$ is a topological algebra such that the infinite sums
$k(\alpha, \ad(\frac{x_{\alpha^\vee}}{2})|\tau)(t_\alpha)$ converge. This is the case for
example if $A$ is complete with respect to a descending filtration, and $x(u)$ is of positive
degree for any $u\in\h$.}

\begin{equation}\label{universal-KZB}
\nabla_{\KZB, \tau}=d-\sum_{\alpha \in \Phi^+} k(\alpha, \ad(\frac{x_{\alpha^\vee}}{2})|\tau)(t_\alpha)d\alpha+y(u^i)du_i
\end{equation}
where $x_{\alpha^\vee}=x(\alpha^\vee)$. 
When $\Phi$ is the root system of type $\sfA_n$, the connection above coincides
with the universal KZB connection introduced by Calaque--Enriquez--Etingof in \cite{CEE}.
In \cite{TLY1}, we proved the following. 
\begin{theorem} 
\label{thm:flatness of KZB}
The connection  $\nabla_{\KZB, \tau}$ is flat if and only if the following relations
hold in $A$
\begin{enumerate}
\item For any rank 2 root subsystem $\Psi$ of $\Phi$, and $\alpha\in\Psi$,
\[[t_\alpha, \sum_{\beta \in \Psi^+}t_\beta]=0. \]
\item For any $u, v\in \h$
\[[x(u), x(v)]=0=[y(u), y(v)].\]
\item For any $u, v\in \h$,
\[[y(u), x(v)]=\displaystyle\sum_{\gamma\in \Phi^+}\langle v, \gamma\rangle \langle u, \gamma\rangle t_\gamma.\]
\item For any $\alpha\in\Phi$ and $u\in\h$ such that $\alpha(u)=0$,
\[[t_\alpha, x(u)]=0=[t_\alpha, y(u)].\]
\end{enumerate}
If, moreover, the Weyl group $W$ of $\Phi$ acts on $A$, then $\nabla_{\KZB, \tau}$ is
$W$--equivariant if and only if
\begin{enumerate}
\item[(5)] For any $w\in W$, $\alpha\in\Phi$, and $u,v\in\h$,
\[w(t_\alpha)=t_{w\,\alpha},\qquad
w(x(u))=x(w\,u),
\aand w(y(v))=y(w\,v)\]
\end{enumerate}
\end{theorem}

In what follows, we denote by $\Aell$ the Lie algebra defined by the
relations (1)--(4) of Theorem \ref{thm:flatness of KZB}, endowed with
the action of $W$ given by relation (5). $\Aell$ has an $\IN$--bigrading
given by $\deg(x(u))=(1,0)$, $\deg(y(v))=(0,1)$ and $\deg(t_\alpha)=(1,1)$.

\subsection{The deformed double current algebra $\DDg$}

The elliptic Casimir connection constructed in this paper is obtained
by specialising the universal connection \eqref{universal-KZB}, that
is mapping the Lie algebra $\Aell$ to an appropriate associative
algebra in a $W$--equivariant way. The correct algebra turns out
to be the \textit{deformed double current algebra} $\DDg$ of $\g$
introduced by Guay \cite {G1,G2}.

$\DDg$ is an algebra over $\IC[\lambda]$, where $\lambda$
is a formal parameter, which deforms the universal central extension of
the double current algebra $\g[u,v]$. It was introduced in \cite{G1} for
$\g=\sl{n}$ with $n\geq 4$ and in \cite{G2} for an arbitrary simple Lie
algebra $\g$ of rank $\ge 3$, but $\g\neq\sl{3}$. It is obtained by degenerating
the defining relations of the affine Yangian of $\g$, see \cite[Thm. 12.1]
{G1} and \cite[Thm. 5.5]{G2}. This degeneration yields a presentation
of $\DDg$ which is similar to the Kac--Moody presentation
of an affine Lie algebra. A second, double loop, presentation of $D_
{\lambda}(\g)$ is obtained in \cite{G2,GY}. It involves two current
algebras in one variable which play a symmetric role. Deformed double
current algebras for $\sl{2}$ and $\sl{3}$ were not defined in \cite{G1}
due to subtleties arising in small rank. A definition for $\sl{2},\sl{3}$
will be proposed in \cite{GY2}. A definition of $\DDg$ in
type $B_2$ and $G_2$ is currently unavailable.

In the current paper, we assume that $\rank(\g)\ge 3$. The following
is the second presentation of the deformed double current algebra
$\DDg$, when $\g\neq\sl{3}$. For $\g=\sl{3}$, more relations need to
be imposed in Definition \ref{def:DDCAintro} so that $\DDg$ is a flat
deformation of the universal central extension of $\g[u, v]$. All other
statements in the current paper hold for $\g=\sl{3}$.

\begin{definition}\cite[Def. 2.2]{GY}
\label{def:DDCAintro}
The deformed double current algebra $\DDg$ is generated by elements
$X$, $K(X)$, $Q(X)$, $P(X)$, $X\in \g$, such that
\begin{enumerate}
\item $X,K(X)$ generate a subalgebra which is an image of $\g[v]$ under the map $X\otimes v \mapsto K(X)$
\item $X,Q(X)$ generate a subalgebra which is an image of $\g[u]$ under the map $X\otimes u \mapsto Q(X)$
\item $P(X)$ is linear in $X$, and for any $X, X'\in \g$, $[X,P(X')]=P[X, X']$
\end{enumerate}
and the following relations hold for all root vectors $X_{\beta_1},X_{\beta_2}\in\g$ with $\beta_1\neq -\beta_2$:
\[
[K(X_{\beta_1}),Q(X_{\beta_2})]  = P([X_{\beta_1},X_{\beta_2}]) - \lambda\frac{(\beta_1,\beta_2)}{4}  S(X_{\beta_1},X_{\beta_2}) + \frac{\lambda}{4} \sum_{\alpha\in\Phi} S([X_{\beta_1},X_{\alpha}],[X_{-\alpha},X_{\beta_2}]),
\]
where $S(a_1, a_2)=a_1a_2+a_2a_1\in\Ug$.
\end{definition}

Note that $\DDg$ is $\IN$--bigraded by
\begin{equation}\label{eq:bigraded}
\deg(X)=(0,0),\quad
\deg(K(X))=(1,0),\quad
\deg(Q(X))=(0,1)
\quad\text{and}\quad
\deg(P(X))=(1,1)=
\deg(\lambda)
\end{equation}
Further, a central element $Z \in \DDg$ is constructed in \cite[Prop. 4.1]{GY},
which is given by
\begin{equation}\label{eq:central element}
Z:=\frac{1}{(\beta, \beta)}\left([K(H_{\beta}), Q(H_{\beta})]-\frac{\lambda}{4}
\sum_{\alpha\in\Phi} S([H_{\beta},X_{\alpha}],[X_{-\alpha},H_{\beta}])\right)
\end{equation}
where $\beta$ is any fixed element in $\Phi$ (the formula above is independent
of the choice of $\beta$, see Thm. \ref{thm:GY center}). 

\begin{remark}
The deformed double current algebra recently showed up in the work of Costello.
In \cite{Costello}, he studies the AdS/CFT correspondence in the reformulation of
Koszul duality for algebras.  For $M2$ branes in an $\Omega$--background, Costello
shows that the algebra of supersymmetric operators on a stack of $k$ branes,
when $k\to \infty$, is a deformed double current algebra of type $\sfA$.
\end{remark}

\subsection{The elliptic Casimir connection}

Consider the $\IN$--grading on $\DDg$ induced by the $\IN$--bigrading \eqref
{eq:bigraded} via the homomorphism $(a,b)\to a$, and let $\wh{\DDg}$ be the
completion of $\DDg$ \wrt this grading. The elliptic Casimir connection is
constructed as follows.

\begin{thm}[Theorem \ref{Thm:elliptic Casimir}]
\label{ThmA}\hfill
\begin{enumerate}
\item
There is an $\IN$--bigraded algebra homomorphism from $\Aell \to D_\lambda
(\g)$ given by
\[
x(u)\mapsto Q(u),\qquad
y(v)\mapsto K(v)
\aand
t_\alpha \mapsto \frac{\lambda}{2}\kappa_\alpha+\frac{Z}{h^\vee}\]
where $\kappa_\alpha$ is the truncated Casimir operator corresponding to $\alpha$,
$h^\vee$ is the dual Coxeter number of $\g$, and $Z$ is the central element
\eqref{eq:central element}.
\item
The corresponding elliptic Casimir connection valued in $\wh{\DDg}$ is given by
\[
\nabla_{\Ell, C}=
d-\frac{\lambda}{2}\sum_{\alpha \in \Phi^+}
\left(
\frac{\theta\left(\alpha+\ad(\frac{Q(\alpha^\vee)}{2})\right)}{\theta(\alpha) \theta\left(\ad(\frac{Q(\alpha^\vee)}{2})\right)}-
\frac{1}{\ad\left((\frac{Q(\alpha^\vee)}{2})\right)}\right)
(\kappa_\alpha)d\alpha
-\sum_{\alpha \in \Phi^+}\frac{\theta'(\alpha|\tau)}{\theta(\alpha|\tau)}\frac{Z}{h^\vee}d\alpha
+\sum_{i=1}^{n}K(u^i)du_i
\]
\item $\nabla_{\Ell, C}$ is flat and $W$-equivariant.
\end{enumerate}
\end{thm}

\subsection{Extension to the modular direction}

We proved in \cite{TLY1} that the elliptic connection \eqref{universal-KZB} extends
to a flat connection in the modular direction, thus generalising a result of \CEE in type
$\sfA$ \cite{CEE}. The extended connection takes values in the semidirect product
$\dd\ltimes\Aell$, where $\dd$ is an infinite--dimensional Lie algebra introduced in
\cite{CEE}. We prove in this paper that the elliptic Casimir connection also extends
in the modular direction by extending the homomorphism $\Aell\to\DDg$ given by
Theorem \ref{ThmA} to $\dd\ltimes\Aell$.

Let $\dd$ be the Lie algebra with generators $\Delta_0, d, X$, and $\{\delta_{2m}\}
_{m\geq 1}$, and relations
\begin{align}
&[d, X]=2X, \,\ [d, \Delta_0]=-2\Delta_0, \,\ [X, \Delta_0]=d, \notag\\
&[\delta_{2m}, X]=0, \,\ [d, \delta_{2m}]=2m\delta_{2m}, \,\ (\ad\Delta_0)^{2m+1}(\delta_{2m})=0 \label{derivation2}.
\end{align}
The generators $X,\Delta_0,d$ span a copy of the Lie algebra $\sl{2}$ inside $\dd$,
and the latter decomposes as $\dd_+\rtimes\sl{2}$, where $\dd_+$ is the subalgebra
generated by $\{\delta_{2m}\}_{m\geq 1}$. It is shown in \cite[Prop. 7.1]{TLY1} that
$\dd$ acts on $\Aell$ by derivations (see also Prop. \ref{prop:d is deriv}).

\begin{thm}[Thm. \ref{thm: sl2}, Prop. \ref{prop:derivation action}, Thm. \ref{thm: derivation on D(g)}]
\label{ThmC}\hfill
\begin{enumerate}
\item There exist elements $\IE, \IF, \IH$, and $\{\IE_{2m}\}_{m\geq 1}$ in
$\DDg$ such that $\{\IE, \IF, \IH\}$ forms an  $\sl{2}$--triple, and which together
with $\{\IE_{2m}\}_{m\in \N}$ satisfy the defining relations \eqref{derivation2} of $\dd$.
\item The corresponding homomorphism $\dd\to\DDg$, together with the homomorphism
$\Aell\to\DDg$ given by Theorem \ref{ThmA} give rise to a homomorphism $\dd\ltimes\Aell\to\DDg$.
\end{enumerate}
\end{thm}
In Proposition \ref{prop:sl2 action}, we also prove that there is an action of $\SL_2(\C)$ on $\DDg$
via linear transformations of the two lattices $\g[u], \g[v]$. The action of the $\sl{2}$-triple $\{\IE, 
\IF, \IH\}$ in Theorem \ref{ThmC} is induced from this $\SL_2(\C)$-action. 

Let $\M_{1, n}$ be the moduli space of pointed elliptic curves associated to the root system $\Phi$. 
More explicitly, let $\mathfrak{H}\ni\tau$ be the upper half plane.
The semidirect product $(Q^\vee\oplus Q^\vee)\rtimes \SL_2(\Z)$ acts on $\h \times \mathfrak{H}$. 
For $(\bold{n}, \bold{m}) \in (Q^\vee\oplus  Q^\vee)$ and $(z, \tau)\in \h \times \mathfrak{H}$, the action is given by translation:
$(\bold{n}, \bold{m})*(z, \tau):=
(z+\bold{n}+\tau\bold{m}, \tau)$. For $\left(\begin{smallmatrix}
a & b \\
c & d
\end{smallmatrix}\right)\in \SL_2(\Z)$, the action is given by
$\left(\begin{smallmatrix}
a & b \\
c & d
\end{smallmatrix}\right)*(z, \tau):=(\frac{z}{c\tau+d}, \frac{a\tau+b}{c\tau+d})$. 
 Let $\alpha(-): \h\to \C$ be the map induced by the root $\alpha\in \Phi$. 
 We define $\widetilde{H}_{\alpha, \tau}\subset \h\times \mathfrak{H}$ to be
\[
\widetilde{H}_{\alpha, \tau}=\{(z, \tau)\in \h\times\mathfrak{H}\mid \alpha(z) \in \Lambda_\tau=\Z+\tau \Z\}.
\]
The elliptic moduli space $\M_{1, n}$ is defined to be the quotient of 
$\h \times \mathfrak{H}\setminus\bigcup_{\alpha\in \Phi^+, \tau\in \mathfrak{H}} \widetilde{H}_{\alpha, \tau}$
by the action of $(Q^\vee \oplus Q^\vee) \rtimes \SL_2(\Z)$.

Let $g(z, x|\tau):=k_x(z, x|\tau)$ be the derivative of function $k(z, x|\tau)$ with respect to variable $x$. 
Set $a_{2n}:=-\frac{(2n+1)B_{2n+2}(2i\pi)^{2n+2}}{(2n+2)!}$, where $B_n$ are the Bernoulli numbers and let $E_{2n+2}(\tau)$ be the Eisenstein series. 
Consider the following function on $\h\times \mathfrak{H}$
\begin{align*}
\Delta:=\Delta(\underline{\alpha}, \tau)=
&-\frac{1}{2\pi i}\IH-\frac{1}{2\pi i}\sum_{n\geq 1}a_{2n}E_{2n+2}(\tau)\IE_{2n}+
\frac{1}{2\pi i}\sum_{\beta \in \Phi^+}g(\beta, \ad\frac{ x_{\beta^\vee}}{2}|\tau)(\frac{\lambda}{2}\kappa_\beta-\frac{Z}{h^{\vee}}).
\end{align*}
This is a meromorphic function on $\h \times \mathfrak{H}$ valued in $\widehat{\DDg}$. It has only poles along the hyperplanes $\bigcup_{\alpha\in \Phi^+, \tau\in \mathfrak{H}} \widetilde{H}_{\alpha, \tau}$.

\begin{thm}
\label{thm:extension}
The following $\widehat{\DDg}$-valued elliptic Casimir connection on $\M_{1, n}$ is flat.
\begin{align*}
\nabla&=\nabla_{\Ell, C}-\Delta d\tau\\
&=\nabla_{\Ell, C}+\left( \frac{1}{2\pi i}\IH+\frac{1}{2\pi i}\sum_{n\geq 1}a_{2n}E_{2n+2}(\tau)\IE_{2n}-\frac{1}{2\pi i}\sum_{\beta \in \Phi^+}g(\beta, \ad\frac{ x_{\beta^\vee}}{2}|\tau)(\frac{\lambda}{2}\kappa_\beta-\frac{Z}{h^{\vee}})\right)d\tau,
\end{align*}
\end{thm}

\subsection{Monodromy conjecture} 

The monodromy of the elliptic Casimir connection yields an action of the elliptic braid group
on representations of $\DDg$. The following is an elliptic analogue of the description of the
monodromy of the rational (resp. trigonometric) connection of $\g$ in terms of the quantum
Weyl group operators of $\Uhg$ (resp. the quantum loop algebra $U_\hbar(L\g)$ \cite
{TL-Duke,TL-IMRN,TL-qCqT,TL-JAlg}.

\begin{conj}\label{conj:TLY}
The monodromy of the elliptic Casimir connection $\nabla_{\Ell, C}$ is described by the quantum Weyl group opg of the quantum toroidal algebra $U_{\hbar}(\g^{\tor})$.
\end{conj}
\Omit{
Note that there are not many finite dimensional representations of $\DDg$. One issue in addressing Conjecture \ref{conj:TLY} is to find a correct category of representations of $\DDg$ on which to take monodromy. These should have in particular finite dimensional weight spaces under $\h$ so that the monodromy is well-defined and be invariant under some action of $\SL(2, \Z)$ so that the connection has a chance to be extended in the modular direction.
}

\Omit{
In \cite{GTL3}, Gautam and Toledano Laredo constructed an explicit equivalence of categories between
finite dimensional representations of the Yangian $\Yhg$ and finite dimensional representations of the quantum loop algebra $U_\hbar(L\g)$, when $\g$ is a finite dimensional simple Lie algebra.
In the case when $\g$ is a Kac-Moody Lie algebra, the functor in \cite{GTL3} establishes an equivalence between subcategories of
the two categories $\sO_{\Int}(Y_\hbar(\widehat{\g}))$ of the affine Yangian and $\sO_{\Int}(U_\hbar(\g^{\tor}))$ of the quantum toroidal algebra.
\begin{problem}\label{prob:fun}
Construct a functor from a suitable category of representations of the affine Yangian $Y_\hbar(\widehat{\g})$ to a suitable category of the deformed double current algebra $\DDg$.
\end{problem}
This should give rise to the correct category of representations of the deformed double current algebra and the quantum toroidal algebra in Conjecture \ref{conj:TLY}.
An approach to Problem \ref{prob:fun} is to 
construct an injective homomorphism from $Y_\hbar(\widehat{\g})$ to a completion of $\DDg$, which becomes an isomorphism after appropriate completions. Such a homomorphism, in the affine case, from the quantum loop algebra to a completion of the Yangian is constructed in \cite{GTL-Selecta} by Gautam-Toledano Laredo. 
}

\subsection{Trigonometric degeneration}

As the parameter $\tau$ tends to $+i\infty$, the elliptic connection \eqref{universal-KZB}
degenerates to a trigonometric one. As we explain in \S \ref{ss:Y and D}, the corresponding
degeneration of the elliptic Casimir connection yields a homomorphism of the Yangian
$\Yhg$ to the deformed double current algebra $\DDg$.

\subsubsection{}
Recall first the general form of the trigonometric connections studied in \cite{TL-JAlg}.
Let $G$ be the simply--connected complex Lie group with Lie algebra $\g$, and $\h
/Q^\vee\cong H\subset G$ the maximal torus with Lie algebra $\h$. Let $A$ be an
algebra endowed with a set of elements $\{ t_\alpha\}_{\alpha\in \Phi}$ such
that $t_{-\alpha}=t_\alpha$, and a linear map $X: \h\to A$. Consider the $A$--valued
connection on $H_{\reg}$ given by either of the following equivalent forms
\begin{align}
\nabla_{\trig}
&= d-\sum_{\alpha\in\Phi_+}\frac{d\alpha}{e^{\alpha}-1}t_{\alpha}-X(u^i)\, du_i \notag\\
&= d-\frac{1}{2} \sum_{\alpha\in \Phi^+} \frac{e^{\alpha}+1}{e^{\alpha}-1}d\alpha t_{\alpha}  -\delta(u^i) du_i
\label{equ:trig universal}
\end{align}
where $\delta:\h\to A$ is the linear map given by $\delta(u)=X(u)-\frac{1}{2}\sum_
{\alpha\in\Phi_+}\alpha(u)\,t_\alpha$. By \cite[Thm. 2.5]{TL-JAlg}, $\nabla_{\trig}$
is flat if, and only if the following relations hold
\begin{enumerate}
  \item For any rank 2 root subsystem $\Psi\subset \Phi$, and $\alpha\in \Psi$,
$[t_\alpha, \sum_{\beta\in \Psi_+}t_\beta]=0.$
  \item For any $u, v\in \h$,
$[X(u), X(v)]=0.$
  \item For any $\alpha\in \Phi_+$, $w\in W$ such that $w^{-1}\alpha$ is a simple root and $u\in \h$, such that $\alpha(u)=0$,
\[
[t_\alpha, X_w(u)]=0,\]
      where $ X_w(u)=X(u)-\sum_{\beta\in \Phi_+\cap w \Phi_-}\beta(u)t_\beta$.
\end{enumerate}
Modulo (1), the relation (3) is equivalent to the following
\begin{enumerate}
\item[(3)'] For $\alpha\in\Phi_+$ and $v\in \h$ such that $\alpha(v)=0$, $[t_\alpha, \delta(v)]=0$.
\end{enumerate}
Denote by $\ttrig$ the Lie algebra defined by the above relations, and note that
$\ttrig$ is $\IN$--graded by $\deg(t_\alpha)=1=\deg(X(u))=\deg(\delta(u))$.

\subsubsection{}
Consider now the elliptic connection \eqref{universal-KZB}. As $\tau\rightarrow+i\infty$,
the functions $\theta(z|\tau)$ and $k(z,x|\tau)$ tend to
\[\frac{e^{\pi iz}-e^{-\pi iz}}{2\pi i}\aand
2\pi i\left(\frac{1}{e^{2\pi i z}-1}+\frac{e^{2\pi i x}}{e^{2\pi i x}-1}-\frac{1}{2\pi i x}\right)\]
respectively. This implies \cite[Sect. 4]{TLY1} that $\nabla_{\KZB, \tau}$ degenerates to
the trigonometric connection 
\[\nabla^{\deg}
=
d-\sum_{\alpha\in\Phi^+}\frac{2\pi i d\alpha}{e^{2\pi i \alpha}-1}\,t_{\alpha}
-\sum_{\alpha\in\Phi^+}
2\pi id\alpha
\left(\frac{e^{2\pi i \ad(\frac{x_{\alpha^\vee}}{2})}}{e^{2\pi i \ad(\frac{x_{\alpha^\vee}}{2})}-1 }  -\frac{1}{2\pi i \ad(\frac{x_{\alpha^\vee}}{2})}\right) t_\alpha  
+\sum_i y(u^i) du_i
\]
Consider the $\IN$--grading on $\Aell$ induced by the $\IN$--bigrading via the homomorphism
$(a,b)\to a$, so that $\deg(x(u))=1=\deg(t_\alpha)$ and $\deg(y(v))=0$, and let $\wh{\ttrig}$ be
the completion of $\Aell$ \wrt this grading. By universality of $\ttrig$, the above degeneration
gives rise to a map $\ttrig\to \widehat{\Aell}$ given by
\begin{equation}\label{eq:hol deg}
t_\alpha\mapsto t_\alpha
\aand
X(u)\mapsto -y(u) +
2\pi i\sum_{\alpha \in \Phi^+}(\alpha, u) 
\left(\frac{e^{2\pi i \ad(\frac{x_{\alpha^\vee}}{2})}}{e^{2\pi i \ad(\frac{x_{\alpha^\vee}}{2})}-1 }  -\frac{1}{2\pi i \ad(\frac{x_{\alpha^\vee}}{2})}\right)t_\alpha
\end{equation}

\subsection{A homomorphism $\Yhg\to\DDg$}\label{ss:Y and D}

\subsubsection{}
Recall that the Yangian $\Yhg$ deforms the current algebra $U(\g[s])$ \cite
{Dr1}. It is an associative algebra over $\C[\hbar]$ generated by elements $x,J(x)$, $x\in\g$
subject to the relations in Definition \ref{def:Yangian}, where $J(x) \equiv x\otimes s$ mod
$\hbar$. In particular, we have $[x, J(y)]=J([x, y])$, for any $x, y\in \g$. 
\Omit{\begin{enumerate}
  \item $\lambda x+ \mu y$ (in $\Yhg$) $=\lambda x+ \mu y$ ( in $\g$).
  \item $xy-yx=[x, y]$.
  \item $J(\lambda x+ \mu y)=\lambda J(x)+\mu J(y)$.
  \item $[x, J(y)]=J([x, y])$.
  \item $[J(x), J([y, z])]+[J(z), J([x, y])]+[J(y), J([z, x])]=\hbar^2([x, x_a], [[y, x_b], [z, x_c]])\{x^a, x^b, x^c\}$.
  \item $[[J(x), J(y)], [z, J(w)]]+[[J(z), J(w)], [x, J(y)]]=\hbar^2([x, x_a], [[y, x_b], [[z, w], x_c]])\{x^a, x^b, J(x^c)\}$. 
\end{enumerate}
for any $x, y, z, w \in \g$ and $\lambda, \mu\in \C$, where $\{x_a\}, \{x^a\}$ are dual bases of $\g$ with respect to $(, )$ and
\[
\{z_1, z_2, z_3\}=\frac{1}{24}\sum_{\sigma\in S_3}z_{\sigma(1)}z_{\sigma(2)}z_{\sigma(3)}.
\]
Note that the relation (6) of $\Yhg$ follows from the relations (1)-(5) if $\g\neq \sl{2}$.}

Drinfeld \cite{Dr2} gave another realisation of $\Yhg$, with generators $\{X_{i, r}^{\pm}, H_{i, r}\}_{i\in I, r\in \mathbb{N}}$ subject to the relations which are similar to the Kac-Moody presentation of an affine Lie algebra.  The Lie subalgebra generated by $\{X_{i, 0}, H_{i, 0}\}_{i\in I}$ is isomorphic to $\mathfrak{g}$, and $X_{i, 1} \equiv x\otimes s$ mod $\hbar$. See Theorem \ref{thm:Yangpres} for a minimal presentation of $\Yhg$ in terms of  $\{X_{i, r}^{\pm}, H_{i, r}\}_{i\in I, r=0, 1}$ given by Guay-Nakajima-Wendlandt. 
The  isomorphism between the two presentations of the Yangian is given by
\begin{align*}
&X_{i,1}^{\pm}=J(X_i^{\pm}) - \hbar \Big(\pm\frac{1}{4}\sum_{\alpha\in\Phi^+} S\big(
[X_i^{\pm},X_{\alpha}^{\pm}],X_{\alpha}^{\mp}\big)-\frac{1}{4}S(X_i^{\pm},H_i)\Big)\\
&H_{i, 1}^{\pm}=J(H_i^{\pm})-\hbar \Big( \frac{1}{4} \sum_{\alpha\in\Phi^+} (\alpha_i, \alpha)S(X_{\alpha}^+, X_{\alpha}^-)-\frac{H_i^2}{2}\Big)
\end{align*}


The trigonometric Casimir connection $\nabla_{\trig, C}$ of $\g$ defined in \cite{TL-JAlg}
is the $\Yhg$--valued connection obtained from the universal trigonometric connection
\eqref{equ:trig universal} via a homomorphism $\ttrig \to \Yhg$ given by
$t_{\alpha}\mapsto\hbar\kappa_\alpha$ and
\[\delta(u) \mapsto -2J(u)
\qquad
\text{or equivalently}
\qquad
X(\alpha_i)\mapsto -2\left(H_{i, 1}-\frac{\hbar}{2} H_i^2\right)
\]

\Omit{
The elliptic Casimir connection depends on the modular parameter $\tau$. 
As the imaginary part of $\tau$ tends to $\infty$, the theta function $\theta(z)$
degenerates to the trigonometric function $\sin(\pi z)/\pi$. We show that the elliptic Casimir connection degenerates to a trigonometric connection. 
\begin{thm}[Proposition \ref{intr:degeration}]
\label{intro:degeneration}
\leavevmode 
As $\Im\tau\to+\infty$, the elliptic Casimir connection $\nabla_{\Ell, C}$ degenerates to the following flat connection on $H_{\reg}$
\Omit{\rcomment{can this expression be simplified somewhat?}}
\yaping{Write two forms of this degeneration.}
\begin{align*}
&AddAdd\\
=&d-\sum_{\alpha \in \Phi^+} \left(
\frac{\lambda\pi i}{2}\frac{e^{2\pi i \alpha}+1 }{e^{2\pi i \alpha}-1}\kappa_\alpha
+
\frac{\lambda}{2} \Big(\pi i\frac{e^{2\pi i \ad(\frac{Q(\alpha^\vee)}{2})}+1 }{e^{2\pi i \ad(\frac{Q(\alpha^\vee)}{2})}-1 }
-\frac{1}{\ad(\frac{Q(\alpha^\vee)}{2})} \Big)\kappa_\alpha
+\Big(\frac{e^{2\pi i \alpha}+1 }{e^{2\pi i \alpha}-1}\Big)\pi i\frac{Z}{h^\vee}\right) d\alpha
+\sum_{i=1}^{n}K(u^i)du_i.
\end{align*}
\end{thm}

By the universality of $\nabla_{\trig}$ in \cite{TL-JAlg}, Theorem \ref{intro:degeneration} induces an algebra homomorphism from $A_{\trig}$ to $\widehat{\DDg}$ given by
$t_{\alpha}\mapsto \kappa_{\alpha}$, and for $h\in \h$, 
 \begin{align*}
\delta(h)	&\mapsto 
 -K(h)+
\sum_{\alpha \in \Phi^+}(\alpha, h)\left(
 \frac{\lambda}{2}\Big(\pi i\frac{e^{2\pi i \ad(\frac{Q(\alpha^\vee)}{2})}+1 }{e^{2\pi i \ad(\frac{Q(\alpha^\vee)}{2})}-1 }
-\frac{1}{\ad(\frac{Q(\alpha^\vee)}{2})}\Big) \kappa_\alpha+ \Big(\frac{e^{2\pi i \alpha}+1 }{e^{2\pi i \alpha}-1}\Big)\pi i\frac{Z}{h^\vee}\right).
\end{align*}
}

\subsubsection{}

Consider now the following diagram
\begin{equation}\label{eq:partial comm}
\xymatrix@R=.7cm@C=1.5cm{
\wh{\Aell}\ar[r]	&\wh{\DDg}\\
\ttrig\ar[u]\ar[r]	&\Yhg \ar@{-->}[u]
}\end{equation}
where the top horizontal arrow is the (completion of) the homomorphism corresponding to the 
elliptic Casimir connection (Thm. \ref{ThmA}), the bottom horizontal one corresponds to the
trigonometric Casimir connection, and the left vertical arrow is the degeneration homomorphism
\eqref{eq:hol deg}.

We show in this paper that \eqref{eq:partial comm} can be completed to a commutative diagram
via a homomorphism $\jmath:\Yhg\to\wh{\DDg}$. In particular, as $\tau\to +i\infty$, the elliptic
Casimir connection of $\g$ degenerates to the trigonometric connection of $\g$, viewed as
taking values in $\wh{\DDg}$ via the homomorphism $\jmath$. Specifically, choose root vectors
$X_\alpha\in \mathfrak{g}_{\alpha}$ for any $\alpha\in \Phi$ such that $(X_\alpha, X_{-\alpha})=
1$, and let $\{h_i\}$, $\{h^i\}$ be dual bases of $\h$. For any $p,q\geq 0$, let $\Omega_{p, q}$
be the element of $U(\g[u])$ defined by
\begin{equation}
\Omega_{p, q}=\sum_{\alpha\in \Phi}(X_{\alpha}\otimes u^p)(X_{-\alpha}\otimes u^q)+\sum_i
(h_i\otimes u^p)(h^i\otimes u^q) 
\end{equation}

\begin{thm}[Theorem \ref{thm:yangian D(g)}]
\label{conj:intro}
There is a unique algebra homomorphism $\jmath:\Yhg\to \widehat{\DDg}$ such that the
diagram \eqref{eq:partial comm} commutes. It is given by $\hbar\mapsto\lambda/2$, $X
\mapsto X$, $X\in \g$, and 
\[J(X)\mapsto 
\frac{1}{2}K(X)-
\frac{\lambda}{4}\left[Q(X), \sum_{n\geq 0} c_{2n+1} \sum_{p+q=2n}{2n \choose p}(-1)^p\Omega_{p, q}\right]
\] 
where the constants $c_{2n+1}$ are determined by the expansion $\pi i\frac{e^{2\pi i x}+1 }{e^{2\pi i x}-1 }
-\frac{1}{x} =\sum_{n\geq 0} c_{2n+1}x^{2n+1}$. 
\end{thm}

\Omit{
\subsection{Extension to the modular direction}

As shown in \cite{CEE}, the universal KZB-connection can be extended to $\mathfrak{M}_{1, n}$, the moduli space of elliptic curves with $n$ marked points. 
In \cite{TLY1}, we also extended the universal KZB-connection that associated to a root system $\Phi$ in the modular direction by varying $\tau$.
In this section, we extend the elliptic Casimir connection $\nabla_{\Ell, C}$ in the $\tau$-direction to a flat connection with values in the deformed double current algebra.
\subsubsection{}
We have the following derivation of $\Aell$. 
Let $\dd$ be the Lie algebra with generators $\Delta_0, d, X$, and $\delta_{2m}( m\geq 1)$, and relations
\begin{align}
&[d, X]=2X, \,\ [d, \Delta_0]=-2\Delta_0, \,\ [X, \Delta_0]=d, \notag\\
&[\delta_{2m}, X]=0, \,\ [d, \delta_{2m}]=2m\delta_{2m}, \,\ (\ad\Delta_0)^{2m+1}(\delta_{2m})=0 \label{derivation2}.
\end{align}
The Lie subalgebra generated by $\Delta_0, d, X$ is a copy of $\sl{2}$ in $\dd$. We can decompose $\dd$ as $\dd=\dd_+\rtimes \sl{2}$, where $\dd_+$ is the subalgebra generated by $\{\delta_{2m}\mid m\geq 1\}$. 
It is shown in \cite[Proposition 7.1]{TLY1} that $\dd$ acts on $\Aell$ by derivation. 
 (See Proposition \ref{prop:d is deriv} for the action of $\dd$ on $\Aell$).

\begin{thm}[Theorem \ref{thm: sl2}, Proposition \ref{prop:derivation action}, Theorem \ref{thm: derivation on D(g)}]
\label{ThmC}
There is an inner action of the derivation algebra $\dd$ on $\DDg$. In other words, there exist elements 
$\IE, \IF, \IH$, and $\{\IE_{2m}\}_{m\in \N}$ in $\DDg$, such that $\{\IE, \IF, \IH\}$ forms an  $\sl{2}$-triple, which together with $\{\IE_{2m}\}_{m\in \N}$ satisfy the relation \eqref{derivation2} of $\dd$. Furthermore, the action of  $\dd$ on $\DDg$  is given by 
\[
\Delta_0\mapsto [\IH,  -], \,\ X\mapsto[\IH,  -], \,\ d \mapsto [\IF, -], \,\ \text{and $\delta_{2m} \mapsto [\IE_{2m}, -]$, for any $m\in \N$.}
\]
\end{thm}
In Proposition \ref{prop:sl2 action}, we prove that there is an action of $\SL_2(\C)$ on $\DDg$ via linear transformations of the two lattices $\g[u], \g[v]$.
The action of the $\sl{2}$-triple $\{\IE, \IF, \IH\}$ in Theorem \ref{ThmC} is induced from this $\SL_2(\C)$-action. 
\subsubsection{}
Let $\M_{1, n}$ be the moduli space of pointed elliptic curves associated to a root system $\Phi$. 
More explicitly, let $\mathfrak{H}\ni\tau$ be the upper half plane.
The semidirect product $(Q^\vee\oplus Q^\vee)\rtimes \SL_2(\Z)$ acts on $\h \times \mathfrak{H}$. 
For $(\bold{n}, \bold{m}) \in (Q^\vee\oplus  Q^\vee)$ and $(z, \tau)\in \h \times \mathfrak{H}$, the action is given by translation:
$(\bold{n}, \bold{m})*(z, \tau):=
(z+\bold{n}+\tau\bold{m}, \tau)$. For $\left(\begin{smallmatrix}
a & b \\
c & d
\end{smallmatrix}\right)\in \SL_2(\Z)$, the action is given by
$\left(\begin{smallmatrix}
a & b \\
c & d
\end{smallmatrix}\right)*(z, \tau):=(\frac{z}{c\tau+d}, \frac{a\tau+b}{c\tau+d})$. 
 Let $\alpha(-): \h\to \C$ be the map induced by the root $\alpha\in \Phi$. 
 We define $\widetilde{H}_{\alpha, \tau}\subset \h\times \mathfrak{H}$ to be
\[
\widetilde{H}_{\alpha, \tau}=\{(z, \tau)\in \h\times\mathfrak{H}\mid \alpha(z) \in \Lambda_\tau=\Z+\tau \Z\}.
\]
The elliptic moduli space $\M_{1, n}$ is defined to be the quotient of 
$\h \times \mathfrak{H}\setminus\bigcup_{\alpha\in \Phi^+, \tau\in \mathfrak{H}} \widetilde{H}_{\alpha, \tau}$
by the action of $(Q^\vee \oplus Q^\vee) \rtimes \SL_2(\Z)$.

Let $g(z, x|\tau):=k_x(z, x|\tau)$ be the derivative of function $k(z, x|\tau)$ with respect to variable $x$. 
Set $a_{2n}:=-\frac{(2n+1)B_{2n+2}(2i\pi)^{2n+2}}{(2n+2)!}$, where $B_n$ is the Bernoulli numbers and let $E_{2n+2}(\tau)$ be the Eisenstein series. 
Consider the following function on $\h\times \mathfrak{H}$
\begin{align*}
\Delta:=\Delta(\underline{\alpha}, \tau)=
&-\frac{1}{2\pi i}\IH-\frac{1}{2\pi i}\sum_{n\geq 1}a_{2n}E_{2n+2}(\tau)\IE_{2n}+
\frac{1}{2\pi i}\sum_{\beta \in \Phi^+}g(\beta, \ad\frac{ x_{\beta^\vee}}{2}|\tau)(\frac{\lambda}{2}\kappa_\beta-\frac{Z}{h^{\vee}}).
\end{align*}
This is a meromorphic function on $\h \times \mathfrak{H}$ valued in $\widehat{\DDg}$. It has only poles along the hyperplanes $\bigcup_{\alpha\in \Phi^+, \tau\in \mathfrak{H}} \widetilde{H}_{\alpha, \tau}$.

\begin{thm}
\label{thm:extension}
The following $\widehat{\DDg}$-valued elliptic Casimir connection on $\M_{1, n}$ is flat.
\begin{align*}
\nabla&=\nabla_{\Ell, C}-\Delta d\tau\\
&=\nabla_{\Ell, C}+\left( \frac{1}{2\pi i}\IH+\frac{1}{2\pi i}\sum_{n\geq 1}a_{2n}E_{2n+2}(\tau)\IE_{2n}-\frac{1}{2\pi i}\sum_{\beta \in \Phi^+}g(\beta, \ad\frac{ x_{\beta^\vee}}{2}|\tau)(\frac{\lambda}{2}\kappa_\beta-\frac{Z}{h^{\vee}})\right)d\tau,
\end{align*}
\end{thm}
}

\Omit{
\subsection{Monodromy conjecture} 

The monodromy of the elliptic Casimir connection yields an action of the elliptic braid group
on representations of $\DDg$. The following is an elliptic analogue of the description of the
monodromy of the rational (resp. trigonometric) connection of $\g$ in terms of the quantum
Weyl group operators of $\Uhg$ (resp. the quantum loop algebra $U_\hbar(L\g)$ \cite
{TL-Duke,TL-IMRN,TL-qCqT,TL-JAlg}.

\begin{conj}\label{conj:TLY}
The monodromy of the elliptic Casimir connection $\nabla_{\Ell, C}$ is described by the quantum Weyl group operators of the quantum toroidal algebra $U_{\hbar}(\g^{\tor})$.
\end{conj}
\Omit{
Note that there are not many finite dimensional representations of $\DDg$. One issue in addressing Conjecture \ref{conj:TLY} is to find a correct category of representations of $\DDg$ on which to take monodromy. These should have in particular finite dimensional weight spaces under $\h$ so that the monodromy is well-defined and be invariant under some action of $\SL(2, \Z)$ so that the connection has a chance to be extended in the modular direction.
}

\Omit{
In \cite{GTL3}, Gautam and Toledano Laredo constructed an explicit equivalence of categories between
finite dimensional representations of the Yangian $\Yhg$ and finite dimensional representations of the quantum loop algebra $U_\hbar(L\g)$, when $\g$ is a finite dimensional simple Lie algebra.
In the case when $\g$ is a Kac-Moody Lie algebra, the functor in \cite{GTL3} establishes an equivalence between subcategories of
the two categories $\sO_{\Int}(Y_\hbar(\widehat{\g}))$ of the affine Yangian and $\sO_{\Int}(U_\hbar(\g^{\tor}))$ of the quantum toroidal algebra.
\begin{problem}\label{prob:fun}
Construct a functor from a suitable category of representations of the affine Yangian $Y_\hbar(\widehat{\g})$ to a suitable category of the deformed double current algebra $\DDg$.
\end{problem}
This should give rise to the correct category of representations of the deformed double current algebra and the quantum toroidal algebra in Conjecture \ref{conj:TLY}.
An approach to Problem \ref{prob:fun} is to 
construct an injective homomorphism from $Y_\hbar(\widehat{\g})$ to a completion of $\DDg$, which becomes an isomorphism after appropriate completions. Such a homomorphism, in the affine case, from the quantum loop algebra to a completion of the Yangian is constructed in \cite{GTL-Selecta} by Gautam-Toledano Laredo. 
}
}

\subsection{Rational Cherednik algebras} 

Recall that a \fd representation of $\g$ is {\it small} if $2\alpha$ is not a weight for
any $\alpha\in\Phi$ \cite{Br,Re,Re2}. The rational Casimir connection of $\g$, when
taken with values in the zero weight space $V[0]$ of a small representation coincides
with the Coxeter KZ connection of the corresponding Weyl group with values in the
$W$--module $V[0]$, namely \cite[\S 9]{TL-Acta}
\[\nabla_{CKZ}=d-\sum_{\alpha\in\Phi_+}\frac{d\alpha}{\alpha}k_\alpha s_\alpha\]
where $k_\alpha=\hbar(\alpha,\alpha)$. A similar statement holds for the trigonometric
connection $\nabla_{\trig, C}$ \cite[\S 7]{TL-JAlg}. Namely, if $V$ is a $\Yhg$--module
whose restriction to $\g$ is small, the zero weight space $V[0]$ carries a natural action
of the degenerate affine Hecke algebra $\mathfrak{H}$ of $W$, and the restriction of
$\nabla_{\trig, C}$ to $V[0]$ coincides with Cherednik's affine KZ connection for
$\mathfrak{H}$ \cite{Ch1}, up to abelian terms.

In this paper, we establish the elliptic analog of these statements by comparing
the elliptic Casimir connection $\nabla_{\Ell, C}$ to the elliptic connection valued
in the rational Cherednik algebra of $W$ constructed in \cite{TLY1}.

Let $H_{\hbar, c}$ be the rational Cherednik algebra of $W$. $H_{\hbar, c}$ is
generated by the group algebra $\C W$, together with a copy of $S\h$ and $S
\h^*$, and depends on two sets of parameters (see \cite{EG}, or Definition \ref
{def:RCA} for the defining relations). 
\Omit{
The rational Cherednik algebra $H_{\hbar, c}$ is the quotient of the algebra $\C W\ltimes T(\h\oplus \h^*)[\hbar]$ (where T denotes the tensor algebra) by the ideal
generated by the relations
\[[x, x'] = 0, \,\  [y, y'] = 0, \,\ 
[y, x]=\hbar\langle y, x\rangle-\sum_{s\in S}c_s\langle \alpha_s, y\rangle\langle \alpha_s^\vee, x\rangle s,\]
where $x, x'\in \h^*$, $y, y'\in \h$.
In particular, the generators of $H_{\hbar, c}$ can be chosen as elements in $W$, and $\{x_u, y_{u'}\mid u\in \h, u' \in \h^*\}$. 
}
In \cite{TLY1}, we constructed a homomorphism $\Aell\to H_{\hbar,c}$. This
yields a flat, $W$--equivariant elliptic connection valued in (a completion of)
$H_{\hbar,c}$, which is given by
\[\nabla_{H_{\hbar, c}}=d+\sum_{\alpha \in \Phi^+}\frac{2c_{\alpha}}{(\alpha|\alpha)}k(\alpha, \ad(\frac{ \alpha^\vee}{2})|\tau) s_{\alpha} d\alpha
   -\sum_{\alpha \in \Phi^+} \frac{\hbar}{h^\vee} \frac{\theta'(\alpha|\tau)}{\theta(\alpha|\tau)} d\alpha
   +\sum_{i=1}^{n} u^i du_i\]

Let now $V$ be a \fd representation of $\DDg$, whose restriction to $\g$ is small.
\begin{thm}[Theorem \ref{thm:KZ and Casimir}]
\leavevmode
\begin{enumerate}
\item The canonical $W$--action on the zero weight space $V[0]$ together with the assignment
\[
x_u \mapsto Q(u), \,\  y_{u'} \mapsto K(u'), \,\ \text{where $x_u, y_{u'} \in H_{\hbar, c}$ for $u\in \h, u' \in \h^*$}
\] yields an action of the rational Cherednik algebra $H_{\hbar, c}$ on $V[0]$.
\item The elliptic Casimir connection with values in $\End(V[0])$ is equal to the sum of the elliptic KZ connection $\nabla_{H_{\hbar, c}}$ and the scalar valued one-form
\[\mathcal{A}=\frac{\lambda}{2}
\sum_{\alpha \in \Phi^+}\left(\frac{ 2h^\vee_l+h^\vee_s}{h^\vee} -(\alpha, \alpha)\right) \frac{\theta'(\alpha|\tau)}{\theta(\alpha|\tau)}   d\alpha, \]
where $\frac{ 2h^\vee_l+h^\vee_s}{h^\vee}$ is a constant only depending on the root system $\Phi$. 
\end{enumerate}
\end{thm}

\subsection{$\left(\gla{k},\gla{n}\right)$ duality for the KZB and elliptic Casimir connections}

Let $M_{k,n}$ be the vector space of $k\times n$ matrices, and $\C[M_{k,n}]$ its ring of
regular functions. It was discovered in \cite{TL-Duke} that the commuting actions of $\gla
{k}$ and $\gla{n}$ on $\C[M_{k,n}]$ give rise to an identification of the rational KZ connection
on $n$ points for $\gla{k}$ with values in $\C[M_{k,1}]^{\otimes n}\cong \C[M_{k,n}]$ with
the rational Casimir connection for $\gla{n}$ with values in $\C[M_{k,n}]$ \cite{TL-Duke}.

A similar statement holds in the trigonometric case \cite{GTL-sl2}. Namely, the trigonometric
KZ connection for $\gla{k}$ with values in $\C[M_{k,1}]^{\otimes n}$, which depends upon
an additional diagonal matrix $\bfs=\diag(s_1,\ldots,s_k)$, coincides with the trigonometric
Casimir connection of $\gla{n}$, when $\C[M_{k,n}]$ is regarded as the tensor product
$\C[M_{1,n}](r_1)\otimes\cdots\otimes \C[M_{1,n}](r_k)$ of $k$ evaluation modules of
the Yangian $Y_\hbar\gla{n}$, where each evaluation point $r_i$ is a function of $s_i$.

In this paper, we prove that $\left(\gla{k},\gla{n}\right)$ duality identifies the (elliptic)
KZB connection for $\gla{k}$ \cite{Ber98a,FW}, and the elliptic Casimir connection for
$\gla{n}$.

\subsubsection{}

Let $\h_k\subset \gla{k}$ be the Cartan subalgebra of $\gla{k}$, and denote by $\h_k^{\reg}$
the set of diagonal matrices in $\h_k$ with distinct eigenvalues. Then, $\C[\h_k^{\reg}]\otimes
\C[M_{k,n}]$ is a module over $B_n=\Diff(\h_k^{\reg})\otimes U(\gl_k)^{\otimes n}$.  Let 
$\mathcal{H}(\gla{k}, \h_k)$ be the Hecke algebra associated to $(\gla{k}, \h_k)$ introduced
in \cite{CEE}. By definition, $\mathcal{H}(\gla{k}, \h_k)$ is a subquotient of $B_n=\Diff
(\h_k^{\reg})\otimes U(\gl_k)^{\otimes n}$ (see \cite[Section 6.3]{CEE}, or Definition
\ref{def:Hecke}). In particular, the zero weight space $\C[\h_k^{\reg}]\otimes \C[M_{k,n}]
[0]$ is a module over $\mathcal{H}(\gla{k}, \h_k)$.

Let $\AAell{n}$ be the Lie algebra with generators $\{x_i,y_i\}_{i=1}^n$ and $\{t_{ij}\}_{1\leq
i<j\leq n}$ (see \cite[Section 6.3]{CEE}, or Sect. \ref{se:duality}). $\AAell{n}$ is a split central
extension of $\Aelln{n}$ for the root system $\sfA_{n-1}$, with kernel spanned by $\ol{x}=
\sum_i x_i$ and $\ol{y}=\sum_i y_i$. In \cite[Prop. 41]{CEE}, the authors construct an
algebra homomorphism $\AAell{n} \to \mathcal{H}(\gla{k}, \h_k) \subset B_n/B_n
\h_k^{\diag}$, which factors through $\Aelln{n}$ and is given by
\begin{gather*}
x_i \mapsto \sum_{a=1}^k x_a\otimes E_{aa}^{(i)},\,\ \phantom{1234}
\qquad
y_i \mapsto -\sum_{a=1}^k \partial_a\otimes E_{aa}^{(i)}+
\sum_{j=1}^n\sum_{1\leq a\neq b\leq k} \frac{1}{x_b-x_a}\otimes E_{ab}^{(i)}E_{ba}^{(j)}\\
t_{ij}\mapsto \sum_{1\leq a, b \leq k}E_{ab}^{(i)}E_{ba}^{(j)}
\end{gather*}
This gives rise in particular to an elliptic connection with values in the zero weight
space $\C[\h_k^{\reg}]\otimes \C[M_{k,n}][0]$. As pointed out in \cite{CEE}, this
connection coincides with the KZB connection for $\gla{k}$ on $n$ points defined
in \cite{Ber98a, FW}.

\Omit{
Consider the KZB connection $\nabla_{\mathcal{H}(\gla{k}, \h_k)}$ in \cite{CEE}

\yaping{Do not need to write the following connection. }

\begin{align*}
\nabla_{\mathcal{H}(\gla{k}, \h_k)}=d&-\sum_{1\leq i\neq j\leq n}\sum_{1\leq a, b\leq k}
k(z_i-z_j, \ad (\Sigma_{c=1}^k x_c E_{cc}^{(i)})|\tau)(E_{ab}^{(i)} E_{ba}^{(j)})dz_i\\
&+\sum_{1\leq i, j\leq n}\sum_{1\leq a\neq b\leq k} \frac{E_{ab}^{(i)} E_{ba}^{(j)}}{ x_b-x_a}dz_i
-\sum_{i=1}^n \sum_{a=1}^k \partial_a E_{aa}^{(i)}dz_i
\end{align*}
that valued in the zero weight space $(\C[\h_k^{\reg}]\otimes \C[M_{k,1}]^{\otimes n})[0]$. 
As pointed in \cite{CEE}, this flat connection coincides with that of \cite{Ber98a, FW}.
}

\subsubsection{}

Consider now the double deformed current algebra $D_{\lambda, \beta}(\sl{n})$ for $\sl{n}$
(\cite{G1} and \ref{ss:DD sl}). The latter depends in fact on two parameters $\lambda,\beta$.

\begin{thm}[Theorem \ref{prop: action of D(sl)}] \label{Intro:action}
The following define an action of the deformed double current algebra 
$D_{\lambda, -\frac{n}{4}\lambda}(\sl{n})$ on $\C[\h_{k}^{\reg}]\otimes \C[M_{1, n}] ^{\otimes k}$:
for $1\leq i\neq j\leq n$, $E_{ij}$ acts by $1\otimes \sum_{a=1}^k (E_{ij})^{(a)}$ and
\begin{align*}
&K(E_{ij}) \text{ acts by $\sum_{a=1}^k  x_a\otimes (E_{ij})^{(a)}$},\\
&Q(E_{ij}) \text{ acts by
$-\sum_{a=1}^k \partial_a\otimes (E_{ij})^{(a)}+\sum_{1\leq a\neq b \leq k} \frac{1}{x_b-x_a}\otimes (\sum_{e=1}^n(E_{ie})^{(a)}(E_{ej})^{(b)}+(E_{ij})^{(a)})$}.
\end{align*}
\end{thm}

By composing with the homomorphism $\Aelln{n}\to D_{\lambda, -\frac{n}{4}\lambda}(\sl{n})$,
Theorem \ref{Intro:action} gives rise to an elliptic Casimir connection for $\sl{n}$, with values
in $\End(\C[\h_{k}^{\reg}]\otimes \C[M_{k,n}])$. 

\subsubsection{}

The following result arises by comparing explicitly the two actions of $\Aelln{n}$ on $(\C[\h_{k}^
{\reg}]\otimes \C[M_{k, n}][0]$.
\Omit{
The conformal blocks of the Wess-Zumino-Novikov-Witten model on genus $1$ curve form a vector bundle on $\mathfrak{M}_{1, n}$. In \cite{FW}, Felder-Wieczerkowski gave an explicit description of the KZB connection on this bundle, after embedding it as a subbundle of the \textit{invariant theta function} bundle. We compare the KZB connection of Felder-Wieczerkowski (reformulated in Theorem \ref{conn:gl_k}) with the elliptic Casimir connection $\nabla_{\Ell, C}$ in Theorem \ref{ThmA}. We get a $(\gla{k}, \gla{n})$-duality in the elliptic case.}

\yaping{Write two parts for the following theorem.  }
\begin{thm}[Theorem \ref{duality}]
\Omit{Under the identification
\[
\C[M_{k, 1}]^{\otimes n} \cong \C[M_{k, n}] \cong \C[M_{1, n}] ^{\otimes k}
\]
the}
The KZB connection on $n$ points for $\gla{k}$ with values in $\C[\h_{k}^{\reg}]\otimes \C[M_{k, 1}]^{\otimes n}[0]$
coincides with the sum of
\begin{enumerate}
\item the elliptic Casimir connection for $\sl{n}$ with parameters $\lambda,-\frac{n}{4}\lambda$
with values in $(\C[\h_{k}^{\reg}]\otimes \C[M_{1, n}]^{\otimes k})[0]$ and
\item the closed abelian one--form given by \[
\mathcal{A}
=\sum_{1\leq i< j \leq n}\frac{\theta'(z_i-z_j|\tau)}{\theta(z_i-z_j|\tau)}\Big(\frac{E_{ii}+E_{jj}}{n}-\frac{1}{n^2} \Big)d z_{ij}.
\]\end{enumerate}
\end{thm}

\Omit{
\subsection{Monodromy conjecture} 
The monodromy of the elliptic Casimir connection yields an action of the elliptic
braid group on representations of $\DDg$. The following is an elliptic analogue of the monodromy theorem of Drinfeld--Kohno \cite{Ko, D}. 
\begin{conj}\label{conj:TLY}
The monodromy of the elliptic Casimir connection $\nabla_{\Ell, C}$ is described by the quantum Weyl group operators of the quantum toroidal algebra $U_{\hbar}(\g^{\tor})$.
\end{conj}
Note that there are not many finite dimensional representations of $\DDg$. One problem in addressing Conjecture \ref{conj:TLY} is to find a correct category of representations of $\DDg$ on which to take monodromy. These should have in particular finite dimensional weight spaces under $\h$ so that the monodromy is well-defined and be invariant under some action of $\SL(2, \Z)$ so that the connection has a chance to be extended in the modular direction.

\Omit{
In \cite{GTL3}, Gautam and Toledano Laredo constructed an explicit equivalence of categories between
finite dimensional representations of the Yangian $\Yhg$ and finite dimensional representations of the quantum loop algebra $U_\hbar(L\g)$, when $\g$ is a finite dimensional simple Lie algebra.
In the case when $\g$ is a Kac-Moody Lie algebra, the functor in \cite{GTL3} establishes an equivalence between subcategories of
the two categories $\sO_{\Int}(Y_\hbar(\widehat{\g}))$ of the affine Yangian and $\sO_{\Int}(U_\hbar(\g^{\tor}))$ of the quantum toroidal algebra.
\begin{problem}\label{prob:fun}
Construct a functor from a suitable category of representations of the affine Yangian $Y_\hbar(\widehat{\g})$ to a suitable category of the deformed double current algebra $\DDg$.
\end{problem}
This should give rise to the correct category of representations of the deformed double current algebra and the quantum toroidal algebra in Conjecture \ref{conj:TLY}.
An approach to Problem \ref{prob:fun} is to 
construct an injective homomorphism from $Y_\hbar(\widehat{\g})$ to a completion of $\DDg$, which becomes an isomorphism after appropriate completions. Such a homomorphism, in the affine case, from the quantum loop algebra to a completion of the Yangian is constructed in \cite{GTL-Selecta} by Gautam-Toledano Laredo. 
}
}

\subsection{Acknowledgments}
Part of the work was done when both authors visited the Mathematical Sciences Research Institute in 2014. We thank Pavel Etingof for helpful discussions. Y. Y. is grateful to the Perimeter Institute for Theoretical Physics and the Institute of Science and Technology Austria for their excellent working conditions. Y. Y. would also like to thank Kevin Costello, Nicolas Guay, and Gufang Zhao for useful discussions. 

\yaping{To do list: In the main text, write the relations of the holonomy Lie algebra of $sl_n$ and $gl_n$.}
\section{The universal elliptic connection of a root system}
\label{sec:conn def}

In this section, we briefly review the Universal Knizhnik-Zamolodchikov-Bernard (KZB) connection
associated to any finite (reduced, crystallographic) root system $\Phi$ constructed in \cite{TLY1}.

\subsection{The Lie algebra $\Aell$}
Let $\h$ be a Euclidean vector space, $\Phi\subset \h^*$ a 
reduced, crystallographic root system. Let $Q \subset \h^*$ be the root lattice
generated by the roots $\{\alpha \mid \alpha\in\Phi\}$ and $P\subset
\h^*$ be the corresponding weight lattice.
Let $Q^\vee\subset \h$ be the coroot lattice generated by the coroots $\alpha^\vee$, with the inner product $(\alpha^\vee, \alpha)=2$. The coroot lattice is dual to the weight lattice $P$. Let $P^\vee \subset \h$ be the dual lattice of $Q$, called the coweight lattice.

\begin{definition}
Let $\Aell$ be the Lie algebra generated by a set of elements $\{t_\alpha\}_{\alpha\in \Phi}$, such that $t_\alpha=t_{-\alpha}$, and two linear maps $x: \h\to A$, $y:\h\to A$. Those generators satisfy the following relations:
\begin{enumerate}
\item For any root subsystem $\Psi$ of $\Phi$ (that is, $\langle \Psi \rangle_\Z \cap \Phi=\Psi$), we have 
\[[t_\alpha, \sum_{\beta \in \Psi^+}t_\beta]=0. \]
\item
$[x(u), x(v)]=0$, $[y(u), y(v)]=0$, for any $u, v\in \h$;
\item
$[y(u), x(v)]=\Sigma_{\gamma\in \Phi^+}\langle v, \gamma\rangle \langle u, \gamma\rangle t_\gamma$.
\item
$[t_\alpha, x(u)]=0$, $[t_\alpha, y(u)]=0$, if $\langle \alpha, u\rangle=0$.
\end{enumerate}
\end{definition}
The Lie algebra $\Aell$ is bigraded, with grading $\deg(x(u))=(1, 0), \deg(y(v))=(0, 1)$, and $\deg(t_\alpha)=(1, 1)$, for any $u, v\in \h$ and $\alpha\in \Phi$. 
\subsection{Theta functions}
\label{theta function}
In this subsection, we recall some basic facts about theta functions that will be used in the paper. 

Let $\Lambda_\tau:=\Z+\Z\tau \subset \C$ and $\mathfrak{H}$ be the upper half plane, i.e. $\mathfrak{H}:=\{z \in \C \mid \Im(z)> 0\}$.
The following properties of $\theta(z| \tau)$ uniquely characterize the theta function $\theta(z| \tau )$: 
\begin{enumerate}
\item
$\theta(z| \tau )$ is a holomorphic function $\C\times \mathfrak{H} \to \C$, such that $\{z\mid\theta(z| \tau )=0\}=\Lambda_\tau$.
\item
$\frac{\partial \theta}{\partial z}(0| \tau)=1$.
\item
$\theta(z+1| \tau )=-\theta(z| \tau )=\theta(-z| \tau )$, and\,\  $\theta(z+\tau| \tau )=-e^{-\pi i \tau}e^{-2 \pi i z}
\theta(z| \tau )$.
\item
$\theta(z| \tau+1)=\theta(z| \tau )$, while $\theta(-z/\tau|-1/\tau)=-(1/\tau)e^{(\pi i/\tau)z^2}\theta(z| \tau )$.
\item Let $q:=e^{2\pi i \tau}$ and $\eta(\tau):=q^{1/24}\prod_{n\geq 1}(1-q^n)$. If we set $\vartheta(z|\tau):=\eta(\tau)^3 \theta(z|\tau)$, then $\vartheta(z|\tau)$ satisfies the differential equation
\[
\frac{\partial \vartheta(z, \tau)}{\partial \tau}
=\frac{1}{4\pi i}\frac{\partial^2 \vartheta(z, \tau)}{\partial z^2}.\]
\end{enumerate}
We have the product formula of $\theta(z|\tau)$:
\begin{equation}\label{prod formula}
\theta(z| \tau)=u^{\frac{1}{2}}\prod_{s>0}(1-q^su)\prod_{s\geq0}(1-q^su^{-1})\frac{1}{2\pi i}\prod_{s>0}(1-q^s)^{-2},
\end{equation} where $u=e^{2\pi i z}$.

\subsection{Principal bundles on elliptic configuration space}
\label{subsec:bundles}
Consider the elliptic curve $\mathcal{E}_\tau:=\C/{\Lambda_\tau}$ with modular parameter $\tau\notin \R$.
Let $T:=\h/(Q^\vee \oplus \tau Q^\vee)$, which non-canonically isomorphic to $\E_\tau^n$, for $n=\rank(Q^{\vee})$. For any root $\alpha\in Q \subset \h^*$, the linear map $\alpha: \h=Q^{\vee}\otimes_{\Z} \C \to \C$ induces a natural map 
\[
\chi_\alpha: \h/(Q^\vee \oplus \tau Q^\vee) \to \E_\tau.\]
Denote kernel of $\chi_\alpha$ by $T_\alpha$, which is a divisor of $T$.
We refer to $\Treg=T\setminus \bigcup_{\alpha\in \Phi}T_{\alpha}$ as the elliptic configuration space associated to $\Phi$. 
If $E=\IR^n$ with standard coordinates $\{\epsilon_i\}$, and $\Phi=\{\epsilon
_i-\epsilon_j\}_{1\leq i\neq j\leq n}\subset E^*$ is the root system of type $\sfA_{n-1}$,
$\Treg$ is the configuration space of $n$ ordered points on the elliptic
curve $\E_\tau$. 

The Lie algebra $\Aell$ is positively bi-graded. 
Let $\widehat{\Aell}$ be the (pro--nilpotent) completion of $\Aell$ with respect to the grading
given by $\deg(x(u))=1=\deg(y(u))$, and $\deg(t_\alpha)=2$, and $\exp(\widehat{\Aell})$ is the
corresponding pro--unipotent group. We now describe a principal bundle $\mathcal P_{\tau, n}$ on the elliptic configuration space $\Treg$ with structure group $\exp(\widehat{\Aell})$. 

The lattice $\Lambda_\tau\otimes Q^\vee$ acts on $\h=\C\otimes Q^\vee$ by translations, and $T=\h/\Lambda_\tau\otimes Q^\vee$ is the quotient of $\h$ by this action of $\Lambda_\tau\otimes Q^\vee$. For any $g\in \exp(\widehat{\Aell})$, and the standard basis $\{\alpha_i^\vee\}_{1\leq i \leq n}$ of $Q^{\vee}$, we define an action of $\Lambda_\tau\otimes Q^{\vee}$ on $\exp(\widehat{\Aell})$ by 
\[
\alpha_i^\vee (g)=g \,\ \text{and} \,\  \tau \alpha_i^\vee (g)=e^{-2\pi i x(\alpha^\vee_i)}g.\] 
The twisted product 
$\widetilde{\mathcal{P}}:=\h\times_{\Lambda_\tau \otimes Q^\vee}\exp(\widehat{\Aell})$ is a principal bundle over $T$ with structure group $\exp(\widehat{\Aell})$. Denote by $\mathcal P_{\tau, n}$ the restriction of this bundle $\widetilde{\mathcal{P}}$ on $\Treg \subset T$.

In other words, let $\pi: \h \to \h/(Q^\vee \oplus \tau Q^\vee)$ be the natural projection. For an open subset $U\subset \Treg$, the sections of $\mathcal P_{\tau, n}$ on $U$ are given by
\[
\{ f: \pi^{-1}(U) \to \exp(\widehat{\Aell})\mid f(z+\alpha^\vee_i)=f(z),
f(z+\tau \alpha^\vee_i)=e^{-2\pi i x(\alpha^\vee_i)}f(z)\}. 
\]
\subsection{The universal KZB connection}
In this subsection, we recall the universal KZB connection associated to root system $\Phi$ in \cite{TLY1}. 
As in \cite{CEE, TLY1}, we set
\begin{equation}\label{function k}
k(z, x|\tau):=\frac{\theta(z+x| \tau)}{\theta(z| \tau)\theta(x| \tau)}-\frac{1}{x}.
\end{equation}
For a fixed $\tau$, the function $k(z, x|\tau)$ belongs to $\Hol(\C-\Lambda_\tau)[[x]]$, which is holomorphic in $x$ in the neighborhood of $x=0$. 
For $x_u\in \Aell$, substituting $x=\ad x_u$ in \eqref{function k}, we get a linear map 
$\Aell\to (\Aell\otimes \Hol(\C-\Lambda_\tau))^{\wedge}$, where $(-)^{\wedge}$ is taking the completion.

We consider the $\widehat{\Aell}-$valued connection on $\Treg$:
\begin{equation}\label{conn any type}
\nabla_{\KZB, \tau}=d-\sum_{\alpha \in \Phi^+} k(\alpha, \ad(\frac{x_{\alpha^\vee}}{2})|\tau)(t_\alpha)d\alpha+\sum_{i=1}^{n}y(u^i)du_i,
\end{equation}
where $\Phi_+\subset \Phi$ is a chosen system of positive roots, $\{u_i\}$, and $\{u^i\}$ are dual basis of $\h^*$ and $\h$ respectively.

Note that the form \eqref{conn any type} is independent of the choice of $\Phi^+$. It follows from the equality $k(z, x|\tau)=-k(-z, -x|\tau)$, which is a direct consequence of the fact that theta function $\theta(z|\tau)$ is an odd function.

\Omit{We show first that the KZB connection $\nabla_{\KZB, \tau}$ \eqref{conn any type} is a connection on the principal bundle $\mathcal{P}_{\tau, n}$. In order to do this, 
we rewrite $\nabla_{\KZB, \tau}$ \eqref{conn any type} as the form
\[
\nabla_{\KZB}=d-\sum_{i=1}^n K_i d\lambda_i^\vee=d-\sum_{i=1}^n \left(\sum_{\alpha \in \Phi^+} (\alpha, \alpha_i) k(\alpha, \ad(\frac{x_{\alpha^\vee}}{2})|\tau)(t_\alpha)-y(\alpha_i)\right) d\lambda_i^\vee. 
\]}
Let $W$ be the Weyl group generated by reflections $\{s_{\alpha} \mid \alpha\in \Phi \}$. Assume there is an action of $W$ on  $\Aell$. 
\begin{theorem}\cite[Theorem A]{TLY1}\label{thm:connection flat}
\begin{enumerate}
\item
The connection $\nabla_{\KZB, \tau}$ \eqref{conn any type} is flat if and only if the relations (1)-(4) in $\Aell$ hold.
\item
The connection $\nabla_{\KZB, \tau}$ is $W$-equivariant if and only if the relations hold:
\[
s_\alpha(t_\gamma)=t_{s_\alpha(\gamma)}, \,\ 
s_\alpha(x(u))=x(s_\alpha u), \,\ s_\alpha(y(v))=y(s_\alpha v). \]
\end{enumerate}
\end{theorem}

\section{The elliptic Casimir connection}

\subsection{The deformed double current algebras}
We assume $\g$ is a semisimple Lie algebra of rank $\geq 3$. 
The deformed double current algebras associated to $\g$ are introduced by Guay in \cite{G1} for $\g=\sl{n}$ with $n\geq 4$ and in \cite{G2} for an arbitrary simple Lie algebra $\g$ of rank $\ge 3$, and $\g \neq \sl{3}$.  
It is a deformation of the universal central extension of $\U(\mathfrak{g}[u, v])$. The double loop presentations of the deformed double current algebras are established in \cite{G1} for $\g=\sl{n}$, $n\geq 4$ and in \cite{GY} for any simple Lie algebra $\g$, with $\rank(\g)\geq 3$ and \cite{GY2} for the  $\g= \sl{2}, \sl{3}$.

We briefly recall the results here. Let $(\cdot,\cdot)$ be the Killing form on $\g$ and let $X_i^{\pm},H_i, 1\le i \le N$ be the Chevalley generators of $\g$ normalized so that $(X_i^+,X_i^-)=1$ and $[X_i^+,X_i^-]=H_i$. For each positive root $\alpha\in \Phi^+$, we choose generators $X_{\alpha}^{\pm}$ of $\g_{\pm\alpha}$ such that $(X_{\alpha}^{+},X_{\alpha}^{-})=1$ and $X_{\alpha_i}^{\pm}=X_i^{\pm}$. If $\alpha>0$, set $X_{\alpha} = X_{\alpha}^+$; if $\alpha<0$, set $X_{\alpha} = X_{-\alpha}^-$.

\begin{definition}\cite[Definition 2.2]{GY}
\label{thm:DDCA2}
Assume $\g$ has rank $\geq 3$, and $\g\neq \sl{3}$. The deformed double current algebra ${D}_\lambda(\mathfrak{g})$ is generated by elements $X$, $K(X)$, $Q(X)$, $P(X)$, $X\in \g$, such that
\begin{enumerate}
\item $K(X), X\in\g$ generate a subalgebra which is an image of $\g\otimes_{\C}\C[v]$ with $X\otimes v \mapsto K(X)$;
\item $Q(X),X\in\g$ generate a subalgebra which is an image of $\g\otimes_{\C}\C[u]$ with $X\otimes u \mapsto Q(X)$;
\item $P(X)$ is linear in $X$, and for any $X, X'\in \g$, $[P(X), X']=P[X, X']$.
\end{enumerate}
and the following relations hold for all root vectors $X_{\beta_1},X_{\beta_2}\in\g$ with $\beta_1\neq -\beta_2$:
\begin{equation}\label{equ:KQ}
[K(X_{\beta_1}),Q(X_{\beta_2})]  = P([X_{\beta_1},X_{\beta_2}]) - \lambda\frac{(\beta_1,\beta_2)}{4}  S(X_{\beta_1},X_{\beta_2}) + \frac{\lambda}{4} \sum_{\alpha\in\Phi} S([X_{\beta_1},X_{\alpha}],[X_{-\alpha},X_{\beta_2}]),
\end{equation} where $S(a_1, a_2)=a_1a_2+a_2a_1\in \\Ug$.
\end{definition}
When $\g=\sl{n}$, the deformed double current algebra $D_{\lambda, \beta}(\sl{n})$ has two deformation parameters $\lambda, \beta$. Let $\{\epsilon_1, \dots, \epsilon_n\}$ be the standard orthogonal basis of $\mathbb{C}^n$.
The set of roots of $\sl{n}$ is denoted by
$\Phi=\{\epsilon_i-\epsilon_j \mid 1\leq i\neq j\leq n\}$, with a choice of positive roots
$\Phi^+=\{\epsilon_i-\epsilon_j \mid 1\leq i < j\leq n\}$. The longest
positive root $\theta$ equals $\epsilon_1-\epsilon_n$. The elementary matrices will be written as $E_{ij}\in \sl{n}$.
So $X_i^+=E_{i, i+1}$, $X_i^-=E_{i, i-1}$ and $H_i = E_{ii}-E_{i+1,i+1}$ for $2\leq i \leq n-1.$ We set $E_{\theta}=E_{1n}$.

In the defining relation of $D_{\lambda, \beta}(\sl{n})$, the relation  \eqref{equ:KQ}  is modified to the following (see \cite[Definition 5.1]{GY}).
\begin{align}\label{rel:two param}
[K(E_{ab}), Q(E_{cd})]=P([E_{ab}, E_{cd}])+(\beta-\frac{\lambda}{2})(\delta_{bc}E_{ad}+&\delta_{ad}E_{cb})-
\frac{\lambda}{4}(\epsilon_a-\epsilon_b, \epsilon_c-\epsilon_d)S(E_{ab}, E_{cd})\\
&+\frac{\lambda}{4}\sum_{1\leq i \neq j\leq n}S([E_{ab}, E_{ij}], [E_{ji}, E_{cd}]).\notag
\end{align}
When $\beta=\frac{\lambda}{2}$, the relation \eqref{rel:two param} coincides with the relation \eqref{equ:KQ}. 

\redtext{
The following equivalent relation of the defining relation \eqref{rel:two param} of $D_{\lambda, \beta}(\sl{n})$ is useful in Section \ref{sub:nkduality}. 
\begin{lemma}\label{lem:rewritten the relation}
In $D_{\lambda, \beta}(\sl{n})$, the relation \eqref{rel:two param} is equivalent to 
\begin{align}
[K(E_{ab}),Q(E_{cd})]=P([E_{ab},E_{cd}])&
 +\frac{ \lambda}{2}\sum_{1\leq j\leq n}\delta_{bc}  E_{aj}E_{jd}  
+\frac{ \lambda}{2}  \sum_{1\leq i\leq n}\delta_{ad} E_{ci} E_{ib}\notag\\&
-  \lambda E_{ad}E_{cb} + (\beta-\frac{ \lambda}{2} - \frac{ \lambda}{4}n)(\delta_{bc}E_{ad}+\delta_{ad}E_{cb})\label{main relation}. 
\end{align}
\end{lemma}
\begin{proof}
It follows from a straightforward computation. For the convenience of the readers, we include a proof here. 
Assume $1\leq a\neq b \leq n$, $1\leq c\neq d \leq n$, and $(a, b)\neq (d, c)$. 
We expand the right hand side of \eqref{rel:two param} using the basic equality $[E_{ij}, E_{kl}]=\delta_{jk} E_{il}-\delta_{li}E_{kj}$ as follows. We have 
\begin{align}
&[K(E_{ab}),Q(E_{cd})]-P([E_{ab},E_{cd}]) \notag\\
=&  (\beta-\frac{ \lambda}{2}  )(\delta_{bc}E_{ad}+\delta_{ad}E_{cb}) - \frac{ \lambda}{4}  (\epsilon_{ab},\epsilon_{cd})S(E_{ab},E_{cd})  + 
\frac{ \lambda}{4}  \sum_{1\leq i\neq j\leq n}S ([E_{ab},E_{ij}],[E_{ji},E_{cd}]) \notag\\
=&  (\beta-\frac{ \lambda}{2} )(\delta_{bc}E_{ad}+\delta_{ad}E_{cb}) - \frac{ \lambda}{4}  (\delta_{ac}-\delta_{ad}-\delta_{bc}+\delta_{bd})S(E_{ab},E_{cd})  + 
\frac{ \lambda}{4}  \sum_{1\leq i\neq j\leq n}S ( \delta_{bi} E_{aj}-\delta_{aj}E_{ib}, \delta_{ic} E_{jd}-\delta_{jd} E_{ci})
 \notag\\
=&  (\beta-\frac{ \lambda}{2} )(\delta_{bc}E_{ad}+\delta_{ad}E_{cb}) - \frac{ \lambda}{4}  (\delta_{ac}-\delta_{ad}-\delta_{bc}+\delta_{bd})S(E_{ab},E_{cd})  \notag\\
&+ 
\frac{ \lambda}{4}  \sum_{1\leq i\neq j\leq n}
\Bigg(\delta_{bi} \delta_{ic} S( E_{aj}, E_{jd})
-\delta_{bi}\delta_{jd}S (  E_{aj},  E_{ci})
-\delta_{aj}\delta_{ic}S ( E_{ib},  E_{jd})
+\delta_{aj}\delta_{jd}  S(E_{ib}, E_{ci})\Bigg)
\notag\\ 
=&  (\beta-\frac{ \lambda}{2} )(\delta_{bc}E_{ad}+\delta_{ad}E_{cb}) - \frac{ \lambda}{4}  (\delta_{ac}-\delta_{ad}-\delta_{bc}+\delta_{bd})S(E_{ab},E_{cd})  \notag\\
&+ \frac{ \lambda}{4}\sum_{1\leq j\leq n, j\neq b}\delta_{bc}  S( E_{aj}, E_{jd})
-\frac{ \lambda}{4}(1-\delta_{bd})S (  E_{ad},  E_{cb})-\frac{ \lambda}{4} (1- \delta_{ac})S ( E_{cb},  E_{ad})+
\frac{ \lambda}{4}  \sum_{1\leq i\leq n, i\neq a}\delta_{ad} S(E_{ib}, E_{ci})
\notag\\
\Omit{=&  (\beta-\frac{ \lambda}{2} )(\delta_{bc}E_{ad}+\delta_{ad}E_{cb}) - \frac{ \lambda}{4}  (\delta_{ac}-\delta_{ad}-\delta_{bc}+\delta_{bd})S(E_{ab},E_{cd}) \notag \\
&
+ \frac{ \lambda}{4}\sum_{1\leq j\leq n}\delta_{bc}  S( E_{aj}, E_{jd})
- \frac{ \lambda}{4}\delta_{bc} S( E_{ab}, E_{bd})
+\frac{ \lambda}{4}\delta_{bd}S (  E_{ad},  E_{cb})+\frac{ \lambda}{4}  \delta_{ac}S ( E_{cb},  E_{ad})+
\frac{ \lambda}{4}  \sum_{1\leq i\leq n}\delta_{ad} S(E_{ib}, E_{ci})
-\frac{ \lambda}{4}  \delta_{ad} S(E_{ab}, E_{ca}) \notag\\&
- \frac{ \lambda}{2}S( E_{ad}, E_{cb})
\notag\\
=&  (\beta-\frac{ \lambda}{2} )(\delta_{bc}E_{ad}+\delta_{ad}E_{cb})+ \frac{ \lambda}{4}\sum_{1\leq j\leq n}\delta_{bc}  S( E_{aj}, E_{jd}) 
+
\frac{ \lambda}{4}  \sum_{1\leq i\leq n}\delta_{ad} S(E_{ib}, E_{ci})
 \notag\\
&
- \frac{ \lambda}{4}\delta_{bc} S( E_{ab}, E_{cd})
+\frac{ \lambda}{4}\delta_{bd}S (  E_{ab},  E_{cb})+\frac{ \lambda}{4}  \delta_{ac}S ( E_{ab},  E_{cd})-\frac{ \lambda}{4}  \delta_{ad} S(E_{ab}, E_{cd})
- \frac{ \lambda}{4}  (\delta_{ac}-\delta_{ad}-\delta_{bc}+\delta_{bd})S(E_{ab},E_{cd}) 
\notag\\&
- \frac{ \lambda}{2}S( E_{ad}, E_{cb})
\notag\\}
=&  (\beta-\frac{ \lambda}{2} )(\delta_{bc}E_{ad}+\delta_{ad}E_{cb})+ \frac{ \lambda}{4}\sum_{1\leq j\leq n}\delta_{bc}  S( E_{aj}, E_{jd}) 
+
\frac{ \lambda}{4}  \sum_{1\leq i\leq n}\delta_{ad} S(E_{ib}, E_{ci})- \frac{ \lambda}{2}  S(E_{ad},E_{cb}) 
\label{eqn:1}
\end{align}
We simplify \eqref{eqn:1} using the assumptions on the indices $a, b, c, d$. As $[E_{ad}, E_{cb}]=\delta_{dc}E_{ab}-\delta_{ab}E_{cd}=0$, we write $S(E_{ad},E_{cb}) =2E_{ad}E_{cb}$. 
As $[E_{ci}, E_{ib}]=E_{cb}-\delta_{cb}E_{ii}$, we write $ S( E_{ci}, E_{ib})=2 E_{ci} E_{ib}-(E_{cb}-\delta_{cb}E_{ii})$. 
As $[ E_{aj}, E_{jd}]=E_{ad}-\delta_{ad}E_{jj}$, we write $S( E_{aj}, E_{jd}) =2E_{aj}E_{jd} -(E_{ad}-\delta_{ad}E_{jj})$. 
Plugging those equalities into \eqref{eqn:1}, we have
\begin{align*}
\Omit{&[K(E_{ab}),Q(E_{cd})]-P([E_{ab},E_{cd}])\\
=&  (\beta-\frac{ \lambda}{2} )(\delta_{bc}E_{ad}+\delta_{ad}E_{cb})+ \frac{ \lambda}{4}\sum_{1\leq j\leq n}\delta_{bc}  S( E_{aj}, E_{jd}) 
+\frac{ \lambda}{4}  \sum_{1\leq i\leq n}\delta_{ad} S(E_{ib}, E_{ci})- \frac{ \lambda}{2}S(E_{ad},E_{cb}) \\}
\eqref{eqn:1}=& (\beta-\frac{ \lambda}{2} )(\delta_{bc}E_{ad}+\delta_{ad}E_{cb})
+ \frac{ \lambda}{4}\sum_{1\leq j\leq n}\delta_{bc}  \Big(2E_{aj}E_{jd} -(E_{ad}-\delta_{ad}E_{jj} ) \Big)\\
&+\frac{ \lambda}{4}  \sum_{1\leq i\leq n}\delta_{ad} \Big(2 E_{ci} E_{ib}-(E_{cb}-\delta_{cb}E_{ii})\Big)
- \frac{ \lambda}{2}  \Big(2E_{ad}E_{cb})\Big)  \\
=&(\beta-\frac{ \lambda}{2} )(\delta_{bc}E_{ad}+\delta_{ad}E_{cb})
+ \frac{ \lambda}{2}\sum_{1\leq j\leq n}\delta_{bc}  E_{aj}E_{jd}  
- \frac{ \lambda}{4}n\delta_{bc} E_{ad} 
+\frac{ \lambda}{4}\delta_{bc} \delta_{ad}\sum_{1\leq j\leq n}E_{jj} 
\\
&+\frac{ \lambda}{2}  \sum_{1\leq i\leq n}\delta_{ad} E_{ci} E_{ib}
-\frac{ \lambda}{4}  n\delta_{ad} E_{cb}
+\frac{ \lambda}{4} \delta_{ad} \delta_{cb}\sum_{1\leq i\leq n}E_{ii}
-  \lambda E_{ad}E_{cb}
 \\
 \Omit{=&
 \frac{ \lambda}{2}\sum_{1\leq j\leq n}\delta_{bc}  E_{aj}E_{jd}  
+\frac{ \lambda}{2}  \sum_{1\leq i\leq n}\delta_{ad} E_{ci} E_{ib}
-  \lambda E_{ad}E_{cb}\\
&
 + (\beta-\frac{ \lambda}{2} - \frac{ \lambda}{4}n)(\delta_{bc}E_{ad}+\delta_{ad}E_{cb})
 +\frac{ \lambda}{2}\delta_{bc} \delta_{ad}\sum_{1\leq i\leq n}E_{ii} 
 \\}
 = &
 \frac{ \lambda}{2}\sum_{1\leq j\leq n}\delta_{bc}  E_{aj}E_{jd}  
+\frac{ \lambda}{2}  \sum_{1\leq i\leq n}\delta_{ad} E_{ci} E_{ib}
-  \lambda E_{ad}E_{cb} + (\beta-\frac{ \lambda}{2} - \frac{ \lambda}{4}n)(\delta_{bc}E_{ad}+\delta_{ad}E_{cb})
\end{align*}
The last equality follows from the assumption that $\epsilon_{a}-\epsilon_{b}\neq \epsilon_{d}-\epsilon_{c}$. Therefore, 
we have $\delta_{ad}\delta_{bc}=0$. 
This completes the proof. 
\end{proof}}

For any root $\beta\in \Phi$ of $\g$, set
\[Z(\beta):=[K(H_{\beta}), Q(H_{\beta})]-\frac{\lambda}{4} \sum_{\alpha\in\Phi} S([H_{\beta},X_{\alpha}],[X_{-\alpha},H_{\beta}]) \in \DDg.\]
\begin{theorem}\cite[Proposition 4.1]{GY}
\label{thm:GY center}
In the deformed double current algebra $\DDg$, we have
\begin{enumerate}
  \item For any two roots $\beta, \gamma\in \Phi$ of $\g$, we have $\frac{Z(\beta)}{(\beta, \beta)}=\frac{Z(\gamma)}{(\gamma, \gamma)}$, In particular, $Z:=\frac{Z(\beta)}{(\beta, \beta)}$ is independent of the choices of $\beta\in \Phi$.
  \item The element $Z$ is central in $\DDg$.
\end{enumerate}
\end{theorem}
In the case of the deformed double current algebra $D_{\lambda, \beta}(\sl{n})$ in type $\sfA_{n-1}$ with two parameters. 
Set 
\begin{align*}
 &Z_n:=\sum_{a=1}^{n}Z_{a, a+1}=\sum_{a=1}^{n}\left([K(H_{a}), Q(H_{a})]-\frac{\lambda}{4}\sum_{1\leq i \neq j\leq n}S([H_{a}, E_{ij}], [E_{ji}, H_{a}])\right),\\
 &\text{\phantom{123456789} where $(a, a+1)=(n, 1)$, when $a=n$. }\\
 &Z_{ab, cd}:=[K(H_{ab}), Q(H_{cd})]-\frac{\lambda}{4}\sum_{1\leq i \neq j\leq n}S([H_{ab}, E_{ij}], [E_{ji}, H_{cd}]).\\
 &W_{ab}:=[K(E_{ab}), Q(E_{ba})]-P(H_{ab})-\frac{\lambda}{4}\sum_{1\leq i \neq j\leq n}S([E_{ab}, E_{ij}], [E_{ji}, E_{ba}])
-\frac{\lambda}{2}S(E_{ab}, E_{ba})
 \end{align*}
 and denote $Z_{ab, ab}$ by $Z_{ab}$ for short.
  \begin{prop}\cite[Proposition 5.1 and Theorem 5.1]{GY}
  \label{prop:center of sl_n}
\begin{romenum}
\item \label{central ele sln}
The element $Z_n$ is central in $D_{\lambda,\beta}(\sl{n})$.
\item \label{prop:two param item 1}
  For any $1\leq a \neq b \leq n$, and $1\leq c\neq b\leq n$, the following relations hold in $D(\sl{n})$, $n\ge 4$.
  \[
  Z_{ab, cd}=(\epsilon_c-\epsilon_d, \epsilon_a-\epsilon_b)W_{ab}+ (\beta-\frac{\lambda}{2})(\epsilon_a+\epsilon_b, \epsilon_c-\epsilon_d)H_{ab}.
  \]
In particular, we have $Z_{ab}=2 W_{ab}$ for any $1\leq a \neq b \leq n$. When $a, b, c, d$ are distinct, we have
$Z_{ab, cd}=0$.
\item \label{prop:two param item 2} \label{prop:two param item 3}
For $1\leq a\neq b \leq n$, and $1\leq c \neq d \leq n$, in $D_{\lambda, \beta}(\sl{n})$, $n\ge 4$, we have:
\begin{align*}
&W_{ab}-W_{cd}=(\beta - \frac{\lambda}{2})(H_{ac}+H_{bd}), \,\ \text{and} \,\
Z_{ab}-Z_{cd}=2(\beta -\frac{\lambda}{2}) (H_{ac}+H_{bd})
\end{align*}
\end{romenum}
\end{prop}

\subsection{The elliptic Casimir connection valued in $\DDg$}
We construct the elliptic Casimir connection with values in the deformed double current algebra $\DDg$ by constructing an algebra homomorphism from $\Aell$ to $\DDg$.

\begin{prop}\label{map from AtoD}
There is an algebra homomorphism from $\Aell \to D_\lambda(\g)$, given by, for $h, h'\in \h$, 
\[
x(h)\mapsto Q(h), \,\ y(h')\mapsto K(h'),\,\  \text{and} \,\ 
t_\alpha \mapsto \frac{\lambda}{2}\kappa_\alpha+\frac{Z}{h^\vee},\]
 where
$\kappa_\alpha:=X_{\alpha}^{+}X_{\alpha}^{-}+X_{\alpha}^{-}X_{\alpha}^{+}$ is the truncated Casimir operator, $h^\vee$ is the dual Coxeter number, and $Z$ is the central element in Theorem \ref{thm:GY center}. 
\end{prop}
\begin{proof}
We only prove that the homomorphism preserves the relation $[y(u), x(v)]=\Sigma_{\gamma\in \Phi^+}\langle v, \gamma\rangle \langle u, \gamma\rangle t_\gamma$ of $\Aell$. All other relations are obviously preserved, proof of which is left as an exercise to the reader. 

As a consequence of Theorem \ref{thm:GY center}, for any $h, h'\in \h$, we have:
\[
[K(h),Q(h')]  =  \frac{\lambda}{4} \sum_{\alpha\in\Phi} S([h,X_{\alpha}],[X_{-\alpha},h])+(h, h')Z
=\frac{\lambda}{2} \sum_{\alpha\in\Phi^+} (h, \alpha)(h', \alpha) \kappa_\alpha+(h, h')Z.
\] 

The claim now follows from the following equality in \cite[Lemma 10.4]{TLY1} that
\[(h, h')=\frac{1}{h^\vee}
\sum_{\alpha\in \Phi^+}(\alpha, h)(\alpha, h'), \,\ \text{for all $h, h' \in \h$},\] where $h^{\vee}$ is the dual Coxeter number.
Therefore, 
\[
[K(h),Q(h')] = \sum_{\alpha\in\Phi^+} (h, \alpha)(h', \alpha) \left(\frac{\lambda}{2}\kappa_\alpha+\frac{Z}{h^\vee}\right). 
\]
This completes the proof. 

\end{proof}
\begin{theorem}
\label{Thm:elliptic Casimir}
The elliptic Casimir connection valued in the deformed double current algebra $\DDg$
\[
\nabla_{\Ell, C}=
d-\frac{\lambda}{2}\sum_{\alpha \in \Phi^+}k(\alpha, \ad(\frac{Q(\alpha^\vee)}{2})|\tau)(\kappa_\alpha)d\alpha
-\sum_{\alpha \in \Phi^+}\frac{\theta'(\alpha|\tau)}{\theta(\alpha|\tau)}\frac{Z}{h^\vee}d\alpha
+\sum_{i=1}^{n}K(u^i)du_i\]
is flat and $W$-equivariant.
\end{theorem}
\begin{proof}
By Theorem \ref{thm:GY center}, $Z$ is in the center of $\DDg$. Therefore, $\ad(\frac{Q(\alpha^\vee)}{2})^n(Z)=0$, for all positive $n\in \N$.  
Recall the function $k(z, x|\tau)$ is given by 
\[
k(z, x|\tau)=\frac{\theta(z+x| \tau)}{\theta(z| \tau)\theta(x| \tau)}-\frac{1}{x} \in \Hol(\C-\Lambda_\tau)[[x]],\] 
and it is holomorphic in the neighborhood of $x=0$. Using the fact that $\theta(z)$ is an odd function, we have $\theta(0)=\theta''(0)=0$, and $\theta'(0)=1$. A simple computation shows $\lim_{x\to 0} k(z, x|\tau)=\frac{\theta'(z)}{\theta(z)}$. 
In other words, viewing $k(z, x|\tau)$ as a function in $x$, the constant term of $k(z, x|\tau)$ with respect to $x$ is then given by $\frac{\theta'(z)}{\theta(z)}$. 
Hence 
\[
k(\alpha, \ad(\frac{Q(\alpha^\vee)}{2})|\tau)(Z)=\frac{\theta'(\alpha|\tau)}{\theta(\alpha|\tau)}(Z).\] 
The conclusion now follows from the criterion in Theorem \ref{thm:connection flat} and Proposition \ref{map from AtoD}. 

\end{proof}

\subsection{The elliptic Casimir connection valued in $D_{\lambda, \beta}(\sl{n})$}\label{ss:DD sl}

In this section, we consider the type $\sfA$ case with two deformation parameters $\lambda, \beta$. 
We construct the elliptic Casimir connection with values in $D_{\lambda, \beta}(\sl{n})$ by constructing an algebra homomorphism from $\Aell$ to $D_{\lambda, \beta}(\sl{n})$.
 
 \begin{prop}\label{prop: map to two param}
 Let $Z_n$ be the central element of $D_{\lambda, \beta}(\sl{n})$ constructed in Proposition \ref{prop:center of sl_n}. 
There is a morphism of algebras $\Aell \to D_{\lambda, \beta}(\sl{n})$, which is given by:
\begin{align*}
&x(u)\mapsto K(u), \,\ y(u)\mapsto Q(u),\\
&
t_{ij}\mapsto \frac{\lambda}{2}(E_{ij}E_{ji}+E_{ji}E_{ij})+\frac{Z_n}{2n^2}+2(\beta - \frac{\lambda}{2})\Big(\frac{E_{ii}+E_{jj}}{n}-\frac{2}{n^2}\sum_{e=1}^n E_{ee}\Big).
\end{align*}
\end{prop}
As a direct consequence, we have
\begin{theorem}\label{conn:two parameters}
The elliptic Casimir connection with two parameters $\lambda, \beta$ valued in the deformed double current algebra $D_{\lambda, \beta}(\sl{n})$
\begin{align}
\nabla_{\Ell, C}&=d-\frac{\lambda}{2}\sum_{1\leq i\neq j\leq n}k(z_i-z_j, \ad K(u_i)|\tau)(E_{ij}E_{ji}+E_{ji}E_{ij}))dz_{ij}+\sum_{i=1}^n Q(u_i)dz_i\\
&-\sum_{1\leq i\neq j \leq n}\frac{\theta'(z_i-z_j|\tau)}{\theta(z_i-z_j|\tau)}\left(\frac{Z_n}{2n^2}+2(\beta - \frac{\lambda}{2})\Big(\frac{E_{ii}+E_{jj}}{n}-\frac{2}{n^2}\sum_{e=1}^n E_{ee}\Big)\right)dz_{ij},\notag
\end{align}
is flat and $S_n$--equivariant.
\end{theorem}

For the rest of this subsection, we prove Proposition \ref{prop: map to two param}. 

Using Proposition \ref{prop:center of sl_n}, we  express $Z_{ab, cd}$ in terms of the central element $Z_n$ and elements in the Cartan subalgebra of $\sl{n}$ as follows. 
\begin{lemma}\label{lem:Z_n1}
We have
\[
Z_{ab, cd}=(\epsilon_a-\epsilon_b, \epsilon_c-\epsilon_d)
\frac{Z_n}{2n}
+2(\beta - \frac{\lambda}{2}) (\delta_{ac}-\delta_{ad})(E_{aa}-\frac{1}{n}\sum_{e=1}^{n}E_{ee})
-2(\beta - \frac{\lambda}{2}) (\delta_{bc}-\delta_{bd})(E_{bb}-\frac{1}{n}\sum_{e=1}^{n}E_{ee}).
\]
\end{lemma}
\begin{proof}
By Proposition \ref{prop:center of sl_n}\eqref{prop:two param item 3}, we have 
$Z_{ab}-Z_{12}=2(\beta -\frac{\lambda}{2}) (H_{a1}+H_{b2})$.
This implies 
\begin{align*}
Z_n=\sum_{e=1}^{n}Z_{e, e+1}
=&\sum_{e=1}^{n} \Big(Z_{12}+2(\beta -\frac{\lambda}{2}) (H_{e1}+H_{e+1, 2}) \Big)
=n Z_{12}+ 2(\beta -\frac{\lambda}{2})\sum_{e=1}^{n} (H_{e1}+H_{e+1, 2}).
\end{align*}
Using Proposition \ref{prop:center of sl_n}\eqref{prop:two param item 2}, we have 
\begin{align*}
W_{ab}
=&\frac{Z_{12}}{2}+(\beta - \frac{\lambda}{2})(H_{a1}+H_{b2})
=\frac{Z-  2(\beta -\frac{\lambda}{2})\sum_{e=1}^{n} (H_{e1}+H_{e+1, 2})}{2n}+(\beta - \frac{\lambda}{2})(H_{a1}+H_{b2})\\
=&\frac{Z}{2n}+(\beta -\frac{\lambda}{2})(E_{aa}+E_{bb}-\frac{2}{n}\sum_{e=1}^n E_{ee})
\end{align*}
Thus, by Proposition \ref{prop:center of sl_n}\eqref{prop:two param item 1}, 
\begin{align*}
Z_{ab, cd}=&(\epsilon_c-\epsilon_d, \epsilon_a-\epsilon_b)W_{ab}+ (\beta-\frac{\lambda}{2})(\epsilon_a+\epsilon_b, \epsilon_c-\epsilon_d)H_{ab}\\
=&(\epsilon_c-\epsilon_d, \epsilon_a-\epsilon_b)\left(\frac{Z_n}{2n}+(\beta -\frac{\lambda}{2})(E_{aa}+E_{bb}-\frac{2}{n}\sum_{e=1}^n E_{ee})\right)+ (\beta-\frac{\lambda}{2})(\epsilon_a+\epsilon_b, \epsilon_c-\epsilon_d)H_{ab}\\
=&(\epsilon_a-\epsilon_b, \epsilon_c-\epsilon_d)
\frac{Z_n}{2n}
+2(\beta - \frac{\lambda}{2}) (\delta_{ac}-\delta_{ad})(E_{aa}-\frac{1}{n}\sum_{e=1}^{n}E_{ee})
-2(\beta - \frac{\lambda}{2}) (\delta_{bc}-\delta_{bd})(E_{bb}-\frac{1}{n}\sum_{e=1}^{n}E_{ee})
\end{align*}
This completes the proof.
\end{proof}
\begin{lemma}\label{lem:Z_n2}
We have the following identity, for any $1\leq a, b, c, d \leq n$,
\begin{align*}
 \sum_{i\neq j}(\epsilon_{ab}, \epsilon_{ij}) (\epsilon_{ij}, \epsilon_{cd})\Big(\frac{E_{ii}+E_{jj}}{n}-\frac{2}{n^2}\sum_{e=1}^n E_{ee}\Big)=2(\delta_{ac}-\delta_{ad})(E_{aa}-\frac{1}{n}\sum_{e=1}^{n}E_{ee})-2(\delta_{bc}-\delta_{bd})(E_{bb}-\frac{1}{n}\sum_{e=1}^{n}E_{ee}).
\end{align*}
\end{lemma}
\begin{proof}
The identity follows from a direct computation. For the convenience of the reader, we include a proof here.
Set $\epsilon_{ab}=\epsilon_{a}-\epsilon_{b}$. 
We simplify the left hand side using the identity: 
\[(\epsilon_{ab}, \epsilon_{ij}) (\epsilon_{ij}, \epsilon_{cd})=
(\delta_{ai}-\delta_{aj}-\delta_{bi}+\delta_{bj}) (\delta_{ic}-\delta_{id}-\delta_{jc}+\delta_{jd}).
\]
After simplification, we have:
\begin{align*}
&\sum_{i\neq j}(\epsilon_{ab}, \epsilon_{ij}) (\epsilon_{ij}, \epsilon_{cd})\Big(\frac{E_{ii}+E_{jj}}{n}-\frac{2}{n^2}\sum_{e=1}^n E_{ee}\Big)\\
\Omit{
=&\sum_{i\neq j}(\delta_{ai}-\delta_{aj}-\delta_{bi}+\delta_{bj}) (\delta_{ic}-\delta_{id}-\delta_{jc}+\delta_{jd})
\Big(\frac{E_{ii}+E_{jj}}{n}-\frac{2}{n^2}\sum_{e=1}^n E_{ee}\Big)\\
=&-\sum_{i\neq j}\Big((\delta_{ai}\delta_{jc}-\delta_{ai}\delta_{jd}-\delta_{ai}\delta_{ic}+\delta_{ai}\delta_{id})
-( \delta_{aj}\delta_{jc}-\delta_{aj}\delta_{jd}-\delta_{aj}\delta_{ic}+\delta_{aj}\delta_{id})\\
&-(\delta_{bi}\delta_{jc}-\delta_{bi}\delta_{jd}-\delta_{bi}\delta_{ic}+\delta_{bi}\delta_{id})
+(\delta_{bj}\delta_{jc}-\delta_{bj}\delta_{jd}-\delta_{bj}\delta_{ic}+\delta_{bj}\delta_{id})\Big)
\Big(\frac{E_{ii}+E_{jj}}{n}-\frac{2}{n^2}\sum_{e=1}^n E_{ee}\Big)\\
=&-(1-\delta_{ac})\Big(\frac{E_{aa}+E_{cc}}{n}-\frac{2}{n^2}\sum_{e=1}^n E_{ee}\Big)
+(1-\delta_{ad})\Big(\frac{E_{aa}+E_{dd}}{n}-\frac{2}{n^2}\sum_{e=1}^n E_{ee}\Big)\\
&+
\sum_{j\neq a}(\delta_{ac}-\delta_{ad})\Big(\frac{E_{aa}+E_{jj}}{n}-\frac{2}{n^2}\sum_{e=1}^n E_{ee}\Big)
+\sum_{i\neq a} (\delta_{ac}-\delta_{ad})\Big(\frac{E_{ii}+E_{aa}}{n}-\frac{2}{n^2}\sum_{e=1}^n E_{ee}\Big)\\
&-(1-\delta_{ac})\Big(\frac{E_{aa}+E_{cc}}{n}-\frac{2}{n^2}\sum_{e=1}^n E_{ee}\Big)
+(1-\delta_{ad})\Big(\frac{E_{dd}+E_{aa}}{n}-\frac{2}{n^2}\sum_{e=1}^n E_{ee}\Big)\\
&+(1-\delta_{bc})\Big(\frac{E_{bb}+E_{cc}}{n}-\frac{2}{n^2}\sum_{e=1}^n E_{ee}\Big)
-(1-\delta_{bd})\Big(\frac{E_{bb}+E_{dd}}{n}-\frac{2}{n^2}\sum_{e=1}^n E_{ee}\Big)\\
&-\sum_{j\neq b}(\delta_{bc}-\delta_{bd})\Big(\frac{E_{bb}+E_{jj}}{n}-\frac{2}{n^2}\sum_{e=1}^n E_{ee}\Big)
-\sum_{i\neq b}(\delta_{bc}-\delta_{bd})\Big(\frac{E_{ii}+E_{bb}}{n}-\frac{2}{n^2}\sum_{e=1}^n E_{ee}\Big)\\
&
+(1-\delta_{bc})\Big(\frac{E_{cc}+E_{bb}}{n}-\frac{2}{n^2}\sum_{e=1}^n E_{ee}\Big)
-(1-\delta_{bd})\Big(\frac{E_{dd}+E_{bb}}{n}-\frac{2}{n^2}\sum_{e=1}^n E_{ee}\Big)\\
}
=
&2\sum_{j}(\delta_{ac}-\delta_{ad})\Big(\frac{E_{aa}+E_{jj}}{n}-\frac{2}{n^2}\sum_{e=1}^n E_{ee}\Big)
-2\sum_{j}(\delta_{bc}-\delta_{bd})\Big(\frac{E_{bb}+E_{jj}}{n}-\frac{2}{n^2}\sum_{e=1}^n E_{ee}\Big)
\\
=&2(\delta_{ac}-\delta_{ad})(E_{aa}-\frac{1}{n}\sum_{e=1}^n E_{ee} )
-2(\delta_{bc}-\delta_{bd})(E_{bb}-\frac{1}{n}\sum_{e=1}^n E_{ee} ).
\end{align*}
Thus, the conclusion follows.
\end{proof}
\begin{proof}[Proof of Proposition \ref{prop: map to two param}]
Using the identity
 \[
(\epsilon_{ab}, \epsilon_{cd}) =\frac{1}{n}\sum_{i< j}(\epsilon_{ab}, \epsilon_{ij})(\epsilon_{ij}, \epsilon_{cd}),
 \] together with Lemma \ref{lem:Z_n1}, Lemma \ref{lem:Z_n2}, we have
 \begin{align*}
 Z_{ab, cd}=&(\epsilon_a-\epsilon_b, \epsilon_c-\epsilon_d)\frac{Z_n}{2n}
+2(\beta - \frac{\lambda}{2}) (\delta_{ac}-\delta_{ad})(E_{aa}-\frac{1}{n}\sum_{e=1}^{n}E_{ee})
-2(\beta - \frac{\lambda}{2}) (\delta_{bc}-\delta_{bd})(E_{bb}-\frac{1}{n}\sum_{e=1}^{n}E_{ee})\\
=&\sum_{i< j}(\epsilon_{ab}, \epsilon_{ij})(\epsilon_{ij}, \epsilon_{cd})\frac{Z_n}{2n^2}
+\sum_{i< j}(\epsilon_{ab}, \epsilon_{ij}) (\epsilon_{ij}, \epsilon_{cd})2(\beta - \frac{\lambda}{2})\Big(\frac{E_{ii}+E_{jj}}{n}-\frac{2}{n^2}\sum_{e=1}^n E_{ee}\Big)
 \end{align*}
 By the definition of $ Z_{ab, cd}$, we get:
 \begin{align*}
 [K(H_{ab}), Q(H_{cd})]&=
\sum_{i< j}(\epsilon_{ab}, \epsilon_{ij}) (\epsilon_{ij}, \epsilon_{cd})
\left(
\frac{\lambda}{2}(E_{ij}E_{ji}+E_{ji}E_{ij})+\frac{Z_n}{2n^2}+2(\beta - \frac{\lambda}{2})\Big(\frac{E_{ii}+E_{jj}}{n}-\frac{2}{n^2}\sum_{e=1}^n E_{ee}\Big)
\right)
 \end{align*}
 Thus, the  map $\Aell\to D_{\lambda, \beta}(\sl{n})$ is well-defined. This completes the proof of Proposition \ref{prop: map to two param}.
\end{proof}

\section{Yangians and the deformed double current algebra}

In this section, we show that as the imaginary part of $\tau$ approaches $\infty$, the connection
$\nabla_{\Ell, C}$ degenerates to a trigonometric connection of the form considered in \cite{TL-JAlg}. 
This gives a map from the trigonometric Lie algebra $A_{\trig}$ to a completion of $\DDg$.  
\redtext{
We define the following grading of $\DDg$
\[
\deg(X)=0, \deg(K(X))=0, \deg(Q(X))=1, \text{and} \,\ \deg(\lambda)=1.
\]
Let $\widehat{\mathcal{D}_{\lambda}(\g)}$ be the completion of $\mathcal{D}_{\lambda}(\g)$ with respect to this grading.}
We expect this map $A_{\trig} \to \widehat{\mathcal{D}_{\lambda}(\g)}$ extends to a map from the Yangian $\Yhg$ to 
$\widehat{\mathcal{D}_{\lambda}(\g)}$.  

\subsection{Trigonometric Casimir connection}
Toledano Laredo constructed the trigonometric Casimir connection by 
specializing the universal trigonometric connection $\nabla_{\trig}$ \eqref{equ:trig universal}. 
It is obtained via a Lie algebra homomorphism from the coefficient algebra $A_{\trig}$ of $\nabla_{\trig}$ to the Yangian of $\g$. 
We now recall it here. 

Let $H=\Hom_\Z(Q,\C^*)$ be the complex algebraic torus with Lie
algebra $\h$ and coordinate ring given by the group algebra $\C Q$.
We denote the function corresponding to $\lambda\in Q$ by $e^
\lambda\in\C[H]$, and set
\begin{equation*}
H\rreg=H\setminus\bigcup_{\alpha\in\Phi}\{e^\alpha=1\}
\end{equation*}
to be the complement of the root hypertori of the maximal torus $H$.

This trigonometric Casimir connection $\nabla_{\trig, C}$ is a connection on the trivial vector bundle $H_{\reg}\times V$, where the fiber $V$ is a finite-dimensional representation of the Yangian $\Yhg$, which is a deformation of $U(\g[s])$.

\begin{definition}
\label{def:Yangian}
The Yangian $\Yhg$ is the associative algebra over
$\C[\hbar]$ generated by elements $x, J(x), x \in \g$ subject to the relations:
\begin{enumerate}
  \item $\lambda x+ \mu y$ (in $Y(\g)$) $=\lambda x+ \mu y$ ( in $\g$).
  \item $xy-yx=[x, y]$.
  \item $J(\lambda x+ \mu y)=\lambda J(x)+\mu J(y)$.
  \item $[x, J(y)]=J([x, y])$.
  \item $[J(x), J([y, z])]+[J(z), J([x, y])]+[J(y), J([z, x])]=\hbar^2([x, x_a], [[y, x_b], [z, x_c]])\{x^a, x^b, x^c\}$.
  \item $[[J(x), J(y)], [z, J(w)]]+[[J(z), J(w)], [x, J(y)]]=\hbar^2([x, x_a], [[y, x_b], [[z, w], x_c]])\{x^a, x^b, J(x^c)\}$,
\end{enumerate}for any $x, y, z, w \in \g$ and $\lambda, \mu\in \C$, where $\{x_a\}, \{x^a\}$ are dual bases of $\g$ with respect to $(, )$ and
\[
\{z_1, z_2, z_3\}=\frac{1}{24}\sum_{\sigma\in S_3}z_{\sigma(1)}z_{\sigma(2)}z_{\sigma(3)}.
\]
\end{definition}

\begin{theorem}[\cite{TL-JAlg}]
The trigonometric Casimir connection valued in $\Yhg$
\begin{align*}
\nabla_{\trig, C}&=
d-\pi i\hbar\sum_{\alpha\in\Phi_+}\frac{e^{2\pi i \alpha}+1}{e^{2\pi i \alpha}-1}d\alpha\,
\kappa_\alpha
+2du_i\,J(u^i)\end{align*}
is flat and $W$-equivariant.
\end{theorem}

\subsection{}
\label{subsec:deg}
In this section, we describe the degeneration of the elliptic Casimir connection $\nabla_{\Ell, C}$
as the imaginary part of $\tau$ tends to $\infty$. 

Let $q=e^{2\pi i \tau}$. As $\Im\tau\rightarrow+\infty$, we have $q\to 0$. 
Using the product formula \eqref{prod formula} of theta function, one can show
that (see also \cite{TLY1})
\begin{align*}
k(\alpha, \ad(\frac{x_{\alpha^\vee}}{2})|\tau)
 \longrightarrow &2\pi i \left(\frac{1}{e^{2\pi i \alpha}-1}+\frac{e^{2\pi i \ad(\frac{x_{\alpha^\vee}}{2})}}{e^{2\pi i \ad(\frac{x_{\alpha^\vee}}{2})}-1 }\right) -\frac{1}{\ad(\frac{x_{\alpha^\vee}}{2})}. 
\end{align*} 
Therefore, the connection $\nabla_{\Ell, C}$ degenerates to the following. 
\begin{align*}
\nabla_{\Ell, C}\Omit{=&d-\sum_{\alpha \in \Phi^+} k(\alpha, \ad(\frac{Q(\alpha^\vee)}{2})|\tau)\Big(\frac{\lambda}{2}\kappa_\alpha+\frac{Z}{h^\vee}\Big)d\alpha+\sum_{i=1}^{n}K(u^i)du_i\\}
\to & d-\sum_{\alpha \in \Phi^+} \Big(  2\pi i \left(\frac{1}{e^{2\pi i \alpha}-1}
+\frac{e^{2\pi i \ad(\frac{Q(\alpha^\vee)}{2})}}{e^{2\pi i \ad(\frac{Q(\alpha^\vee)}{2})}-1 }\right) -\frac{1}{\ad(\frac{Q(\alpha^\vee)}{2})} \Big)\Big(\frac{\lambda}{2}\kappa_\alpha+\frac{Z}{h^\vee}\Big)d\alpha+\sum_{i=1}^{n}K(u^i)du_i\\
=&d-\sum_{\alpha \in \Phi^+} \left(\frac{\lambda \pi i \kappa_\alpha}{e^{2\pi i \alpha}-1 } +
\frac{\lambda}{2} \Big(\frac{2\pi i e^{2\pi i \ad(\frac{Q(\alpha^\vee)}{2})}}{e^{2\pi i \ad(\frac{Q(\alpha^\vee)}{2})}-1 }
-\frac{1}{\ad(\frac{Q(\alpha^\vee)}{2})} \Big)\kappa_\alpha
+\Big(\frac{e^{2\pi i \alpha}+1 }{e^{2\pi i \alpha}-1}\Big)\pi i\frac{Z}{h^\vee}\right) d\alpha
+\sum_{i=1}^{n}K(u^i)du_i\\
=&d-\sum_{\alpha \in \Phi^+} \left(
\frac{\lambda\pi i}{2}\frac{e^{2\pi i \alpha}+1 }{e^{2\pi i \alpha}-1}\kappa_\alpha
+
\frac{\lambda}{2} \Big(\pi i\frac{e^{2\pi i \ad(\frac{Q(\alpha^\vee)}{2})}+1 }{e^{2\pi i \ad(\frac{Q(\alpha^\vee)}{2})}-1 }
-\frac{1}{\ad(\frac{Q(\alpha^\vee)}{2})} \Big)\kappa_\alpha
+\Big(\frac{e^{2\pi i \alpha}+1 }{e^{2\pi i \alpha}-1}\Big)\pi i\frac{Z}{h^\vee}\right) d\alpha
+\sum_{i=1}^{n}K(u^i)du_i.
\end{align*}
Note that the constant term of $\frac{2\pi i e^{2\pi i \ad(\frac{x_{\alpha^\vee}}{2})}}{e^{2\pi i \ad(\frac{x_{\alpha^\vee}}{2})}-1 }  -\frac{1}{\ad(\frac{x_{\alpha^\vee}}{2})}$ is $\pi i$. 
This gives the following.
\begin{prop}
\label{intr:degeration}
As $\Im\tau\to+\infty$, the elliptic Casimir connection $\nabla_{\Ell, C}$ degenerates to the following flat connection  $\widetilde{\nabla}$:
\begin{align*}
\widetilde{\nabla}=
d-\sum_{\alpha \in \Phi^+} \left(
\frac{\lambda\pi i}{2}\frac{e^{2\pi i \alpha}+1 }{e^{2\pi i \alpha}-1}\kappa_\alpha
+
\frac{\lambda}{2} \Big(\pi i\frac{e^{2\pi i \ad(\frac{Q(\alpha^\vee)}{2})}+1 }{e^{2\pi i \ad(\frac{Q(\alpha^\vee)}{2})}-1 }
-\frac{1}{\ad(\frac{Q(\alpha^\vee)}{2})} \Big)\kappa_\alpha
+\Big(\frac{e^{2\pi i \alpha}+1 }{e^{2\pi i \alpha}-1}\Big)\pi i\frac{Z}{h^\vee}\right) d\alpha
+\sum_{i=1}^{n}K(u^i)du_i.
\end{align*}
 \end{prop}

 By the universality of the trigonometric connection $
\nabla_{\trig}$ \eqref{equ:trig universal}, Proposition \ref{intr:degeration} induces a map from the trigonometric Lie algebra $A_{\trig}$ to the completion of $\DDg$, given by
$t_{\alpha}\mapsto \kappa_{\alpha}$, and for $h\in \h$, 
 \begin{align*}
\delta(h)	&\mapsto 
 -K(h)+
\sum_{\alpha \in \Phi^+}(\alpha, h)\left(
 \frac{\lambda}{2}\Big(\pi i\frac{e^{2\pi i \ad(\frac{Q(\alpha^\vee)}{2})}+1 }{e^{2\pi i \ad(\frac{Q(\alpha^\vee)}{2})}-1 }
-\frac{1}{\ad(\frac{Q(\alpha^\vee)}{2})}\Big) \kappa_\alpha+ \Big(\frac{e^{2\pi i \alpha}+1 }{e^{2\pi i \alpha}-1}\Big)\pi i\frac{Z}{h^\vee}\right).
\end{align*}
Note that the morphism $A_{\trig} \to \widehat{\DDg}$ is not unique. For example, the assignment 
$t_{\alpha}\mapsto \kappa_{\alpha}$, and 
 \begin{align*}
\delta(h)	&\mapsto 
 -K(h)+ \frac{\lambda}{2}
\sum_{\alpha \in \Phi^+}(\alpha, h)\left(
\pi i\frac{e^{2\pi i \ad(\frac{Q(\alpha^\vee)}{2})}+1 }{e^{2\pi i \ad(\frac{Q(\alpha^\vee)}{2})}-1 }
-\frac{1}{\ad(\frac{Q(\alpha^\vee)}{2})}\right) \kappa_\alpha
\end{align*}
also defines a morphism $A_{\trig} \to \widehat{\DDg}$. Indeed, the element $Z\in \DDg$ is central. The relation
\[
[t_{\alpha}, \delta(h)]=0, \,\ \text{if} \,\ \alpha(h)=0
\] in $A_{\trig}$ is preserved if we shift the image of $\delta(h)$ by the central element $Z$. 
\subsection{}
\label{Yangiansection}
We have the following diagram
\[
\xymatrix@R=1em{
A_{\trig}\ar[r] \ar[d]& \widehat{\DDg}\\
\Yhg \ar@{-->}[ur]&
}
\]
where the horizontal map $A_{\trig}\to  \widehat{\DDg}$ follows from the degeneration of the elliptic Casimir connection, and the vertical map $A_{\trig}\to \Yhg$ gives the trigonometric Casimir connection in \cite{TL-JAlg}. 
We show that the Lie algebra homomorphism $A_{\trig}\to \widehat{\DDg}$ extends to an algebra homomorphism 
$\Yhg \to \widehat{\DDg}$. In this section, we give an explicit formula for this extension. 
\begin{theorem}
\label{thm:yangian D(g)}
The following assignment gives an algebra morphism $\Yhg\to \widehat{\DDg}$: $\hbar \mapsto \frac{\lambda}{2}$, and for any $X\in \g$, 
\begin{align*}
X	\mapsto X, \,\ \,\ 
J(X)	\mapsto 
\frac{1}{2}K(X)-
\frac{\lambda}{4}\Big[Q(X), \sum_{n\geq 0} c_{2n+1} \sum_{p+q=2n}{2n \choose p}(-1)^p\Omega_{p, q}\Big], 
\end{align*}
where
\[
\Omega_{p, q}:=\sum_{\alpha\in \Phi}(X_{\alpha}\otimes u^p)(X_{-\alpha}\otimes u^q)+\sum_i
(h_i\otimes u^p)(h^i\otimes u^q),
\] and the constants $c_{2n+1}$ are the coefficients of the expansion $\pi i\frac{e^{2\pi i x}+1 }{e^{2\pi i x}-1 }
-\frac{1}{x} =\sum_{n\geq 0} c_{2n+1}x^{2n+1}$. 
\end{theorem}

We first explain how to obtain the formula in Theorem \ref{thm:yangian D(g)}. 
We require that, when restricting on the enveloping algebra $\Ug\subset \Yhg$ of $\g$, the map is an identity. 
That is, $X\mapsto X$, for $X\in \g$. 
The morphism $A_{\trig} \to \widehat{\DDg}$ implies that, for any $h\in \h$, 
\begin{align*}
2J(h)	&\mapsto 
K(h)- \frac{\lambda}{2}
\sum_{\alpha \in \Phi^+}(\alpha, h)\left(
\pi i\frac{e^{2\pi i \ad(\frac{Q(\alpha^\vee)}{2})}+1 }{e^{2\pi i \ad(\frac{Q(\alpha^\vee)}{2})}-1 }
-\frac{1}{\ad(\frac{Q(\alpha^\vee)}{2})}\right) \kappa_\alpha. 
\end{align*}

We use the relation $[X, J(Y)]=J[X, Y]$ of the Yangian $\Yhg$ to deduce the image of $J(X)$ in Theorem \ref{thm:yangian D(g)}, for $X\in \g$. 
\begin{lemma}\label{eq:Q(v)}
For any $n\in \N$, $h\in \h$, we have
\begin{align*}
\sum_{\alpha\in \Phi^+}(\alpha, h)
\ad(Q(\frac{\alpha^\vee}{2}))^{2n+1}(\kappa_\alpha)
=&\sum_{p+q=2n}{2n \choose p}(-1)^p\Big[Q(h), \Omega_{p, q}\Big]. 
\end{align*}
\end{lemma}
\begin{proof}
We have, for any $m\in \N$, for a fixed root $\alpha\in \Phi$, 
\begin{align*}
&\ad(Q(\frac{\alpha^\vee}{2}))^{m}(\kappa_\alpha)\\
=&\sum_{p+q=m}{m \choose p}\ad(Q(\frac{\alpha^\vee}{2}))^{p}(X_{\alpha})\ad(Q(\frac{\alpha^\vee}{2}))^{q}(X_{-\alpha})
+\sum_{p+q=m}{m \choose p}\ad(Q(\frac{\alpha^\vee}{2}))^{p}(X_{-\alpha})\ad(Q(\frac{\alpha^\vee}{2}))^{q}(X_{\alpha})
\\
=&
\sum_{p+q=m}{m \choose p}(-1)^q(X_{\alpha}\otimes u^p)(X_{-\alpha}\otimes u^q)
+\sum_{p+q=m}{m \choose p}(-1)^p
(X_{-\alpha}\otimes u^p)(X_{\alpha}\otimes u^q).
\end{align*}
Let $h\in \h$ be any element in the Cartan subalgebra $\h\subset \g$. 
We decompose $h$ as $h=(h, \alpha) \frac{\alpha^{\vee}}{2}+h'$, where $(\alpha, h')=0$. 
Using the fact that $[Q(h'), Q(\frac{\alpha^{\vee}}{2})]=0$ and $[Q(h'), \kappa_{\alpha}]=0$, we have
\begin{align*}
\sum_{\alpha\in \Phi^+}(\alpha, h)
\ad(Q(\frac{\alpha^\vee}{2}))^{2n+1}(\kappa_\alpha)
=&\Big[Q(h), 
\sum_{\alpha\in \Phi^+}\ad(Q(\frac{\alpha^\vee}{2}))^{2n}(\kappa_\alpha)
\Big] \notag\\
=&\Big[Q(h), 
\sum_{\alpha\in \Phi} \sum_{p+q=2n}{2n \choose p}(-1)^p(X_{\alpha}\otimes u^p)(X_{-\alpha}\otimes u^q)\Big]\notag\\
=&\sum_{p+q=2n}{2n \choose p}(-1)^p\Big[Q(h), \Omega_{p, q}\Big], 
\end{align*}
where the last equity follows from the relation $[Q(h), h'\otimes u^m]=0$, for any $h'\in \h$, and $m\in \N$. 
This completes the proof. 
\end{proof}
\begin{lemma}\label{lem:J(X)}
Let $X\in \g$ be any element in $\g$ and $X\notin\h$. 
Choose $h\in \h$, such that $[h, X]=X$. We have
\[
J(X):=[J(h), X]=\frac{1}{2}K(X)-
\frac{\lambda}{4}\Big[Q(X), \sum_{n\geq 0} c_{2n+1} \sum_{p+q=2n}{2n \choose p}(-1)^p\Omega_{p, q}\Big].
\]
\end{lemma}
\begin{proof}
Lemma \ref{eq:Q(v)} together with the identity $[\Omega_{p, q}, X]=0$ imply for any $n\in \N$, 
\begin{align}
&\sum_{\alpha\in \Phi^+} (\alpha, h) \Big[\ad(Q(\frac{\alpha^\vee}{2}))^{2n+1}(\kappa_\alpha), \,\ X\Big]
=\sum_{p+q=2n}{2n \choose p}(-1)^p\Big[Q(X), \Omega_{p, q}\Big]. \label{eq:X with O_pq}
\end{align}
The function $\Big(\pi i\frac{e^{2\pi i x}+1 }{e^{2\pi i x}-1 }
-\frac{1}{x} \Big)$ is an odd function in $x$ with non-constant term. 
Indeed, one can easily check that $\lim_{Y\to 0}\Big( \frac{e^{Y}+1}{e^{Y}-1}-\frac{2}{Y}\Big)=0$. 
Therefore, we have the expansion in the neighborhood of $x=0$, 
\[
\Big(\pi i\frac{e^{2\pi i x}+1 }{e^{2\pi i x}-1 }
-\frac{1}{x} \Big)=\sum_{n\geq 0} c_{2n+1}x^{2n+1}, \,\ \text{
where $c_{2n+1}$ is the efficient.  }
\]
By \eqref{eq:X with O_pq}, we have
\begin{align*}
\Big[\sum_{\alpha \in \Phi^+}(\alpha, h)
 \Big(\pi i\frac{e^{2\pi i \ad(\frac{Q(\alpha^\vee)}{2})}+1 }{e^{2\pi i \ad(\frac{Q(\alpha^\vee)}{2})}-1 }
-\frac{1}{\ad(\frac{Q(\alpha^\vee)}{2})} \Big)\kappa_\alpha, X \Big]
=&\sum_{n\geq 0} c_{2n+1} \Big[\sum_{\alpha \in \Phi^+}(\alpha, h)
 \ad(\frac{Q(\alpha^\vee)}{2})^{2n+1}(\kappa_\alpha), X \Big]\\
 =&\sum_{n\geq 0} c_{2n+1} \sum_{p+q=2n}{2n \choose p}(-1)^p\Big[Q(X), \Omega_{p, q}\Big].
\end{align*}
As a consequence, the image of $J(X)$ is
\begin{align*}
J(X)	&\mapsto 
\frac{1}{2}K(X)-
\frac{\lambda}{4}\Big[Q(X), \sum_{n\geq 0} c_{2n+1} \sum_{p+q=2n}{2n \choose p}(-1)^p\Omega_{p, q}\Big].
\end{align*}
This completes the proof. 
\end{proof}
Notice that the formula in Lemma \ref{lem:J(X)} works for any element $X\in \g$. 
The relation $[J(X_\beta), X_{-\beta}]=J(H_\beta)$ in the Yangian forces the image of $J(h)$ to be
\begin{align*}
J(h)\mapsto \frac{1}{2}K(h)- \frac{\lambda}{4}
\sum_{\alpha \in \Phi^+}(\alpha, h)\left(
\pi i\frac{e^{2\pi i \ad(\frac{Q(\alpha^\vee)}{2})}+1 }{e^{2\pi i \ad(\frac{Q(\alpha^\vee)}{2})}-1 }
-\frac{1}{\ad(\frac{Q(\alpha^\vee)}{2})}\right) \kappa_\alpha
\end{align*}
In other words, there is no shift of the central element $Z$ in the above formula. 
\begin{lemma}\label{lem:Jeq}
The assignment in Lemma \ref{lem:J(X)} preserves the relation $[J(X), Y]=J([X, Y])$, for any $X, Y\in \g$. 
\end{lemma}
\begin{proof}
This follows from the relations $[K(X), Y]=K([X, Y])$, $[Q(X), Y]=Q([X, Y])$, and $[\Omega_{p, q}, X]=0$, for any $X, Y\in \g$. 
\end{proof}
\subsection{Proof of Theorem \ref{thm:yangian D(g)}}

We use the following presentation of $\Yhg$, which was obtained by Guay-Nakajima-Wendlandt in \cite{GNW} for a symmetrizable Kac-Moody algebra $\g$ under certain assumption on the Cartan matrix. It eliminates the relation (1.6) in  \cite[Theorem 1.2]{Le} when $\g$ is finite dimensional. For our purposes, we state the result when $\g$ is finite dimensional, and $\g\neq\sl{2}$.
\begin{theorem} (\cite[Thm. 1.2]{Le}, \cite[Thm. 2.12]{GNW})
\label{thm:Yangpres}
The Yangian $\Yhg$ is isomorphic to the $\C$-algebra generated by elements $X_{i,r}^{\pm}, H_{i,r}$ for $i\in I$ and $r=0,1$ which satisfy the following relations for $i, j\in I$:
\begin{gather*}
[H_{i,r},H_{j,s}] = 0, \;\; [H_{i,0},X_{j,s}^{\pm}] = \pm (\alpha_i, \alpha_j) X_{j,s}^{\pm} \;\;  \forall\, r,s\in\{0,1\}\\
[X_{i,r}^+, X_{j,s}^-] = \delta_{ij} H_{i,r+s}\; \text{ for } r+s=0,1\\
[X_{i,1}^{\pm},X_{j,0}^{\pm}] - [X_{i,0}^{\pm},X_{j,1}^{\pm}] = \pm\frac{\hbar (\alpha_i, \alpha_j)}{2} S(X_{i,0}^{\pm},X_{j,0}^{\pm})\\
[H_{i,1},X_{j,0}^{\pm}] - [H_{i,0},X_{j,1}^{\pm}] = \pm\frac{\hbar (\alpha_i, \alpha_j)}{2} S(H_{i,0},X_{j,0}^{\pm})
\Omit{\begin{equation}
\big[ [X_{i,1}^+,X_{i,1}^-],H_{i,1}\big]=0, \label{YLe3}
\end{equation}}
ad(X_{i,0}^{\pm})^{1-a_{ij}}(X_{j,0}^{\pm}) = 0
\end{gather*}
where $I$ is the set of vertices of the Dynkin diagram of $\g$, and $a_{ij}$ is the $ij$-th entry of the Cartan matrix. \end{theorem}

The isomorphism between this presentation of $\Yhg$ and the one given in Definition \ref{def:Yangian} sends $X_{i,1}^{\pm}$ to $J(X_i^{\pm}) - \hbar \omega_i^{\pm}$, and $H_{i, 1}^{\pm}=J(H_i^{\pm})-\hbar \nu_i^{\pm}$, 
where
\begin{equation}
\omega_i^{\pm} = \pm\frac{1}{4}\sum_{\alpha\in\Phi^+} S\big(
[X_i^{\pm},X_{\alpha}^{\pm}],X_{\alpha}^{\mp}\big)-\frac{1}{4}S(X_i^{\pm},H_i) \text{ and } \nu_i = [\omega_i^+,X_i^-] = \frac{1}{4} \sum_{\alpha\in\Phi^+} (\alpha_i, \alpha)S(X_{\alpha}^+, X_{\alpha}^-)-\frac{H_i^2}{2}.
\end{equation}

\begin{proof}[Proof of Theorem \ref{thm:yangian D(g)}]
To show Theorem \ref{thm:yangian D(g)}, most of the relations in Theorem \ref{thm:Yangpres} follow directly from the fact that 
$[J(X),X'] = J([X,X'])$ for all $X,X' \in \g$ in Lemma \ref{lem:Jeq}. 

It remains to show that $[H_{i,1},H_{j,1}]=0$ for all $i,j\in I$. This follows from the map $A_{\trig} \to \widehat{\DDg}$ in \S \ref{subsec:deg} obtained from the degeneration of the elliptic Casimir connection. 
Indeed, we have the following diagram
\begin{equation}
\xymatrix@R=1em{
A_{\trig}\ar[r]^{f_1} \ar[d]& \widehat{\DDg}\\
Y_{\hbar}(f_2) \ar@{-->}[ur]_{f_2}&
}
\end{equation}
The image of $X(H_i)$ in $\widehat{\DDg}$ is given by
\begin{align*}
f_1( X(H_i))=&f_1\Big(\delta(H_i)+\frac{1}{2}\sum_{\alpha\in\Phi_+}\alpha(H_i )\,t_\alpha \Big) \\
= &f_2\Big(-2J(H_i)+\frac{\hbar}{2}\sum_{\alpha\in\Phi_+}(\alpha, \alpha_i)\,\kappa_\alpha\Big)=f_2\Big(-2(H_{i, 1}-\frac{\hbar}{2} H_i^2)\Big)\\
=& -2 f_2(H_{i, 1}) + \hbar H_i^2
\end{align*}
Therefore, the relation $[X(H_i), X(H_j)]=0$ in $A_{\trig}$, and the obvious relation $[f_2(H_{i, 1}), H_i^2]=0$ in $\widehat{\DDg}$ imply that $[f_2(H_{i,1}), f_2(H_{j,1})]=0$. This completes the proof. 
\end{proof}

\section{The elliptic KZ connection}
In this section, we show that the rational Cherednik algebra $H_{\hbar, c}$ of $W$ is, very roughly speaking,
the "Weyl group" of the deformed double current algebra  $\DDg$.
This is an analogy relation between the degenerate affine Hecke algebra and the Yangian in \cite{TL-JAlg}. 
More precisely, we show that if $V$
is a $\DDg$--module whose restriction to $\g$ is \textit{small}, the canonical action of $W$ on the zero weight space $V[0]$ extends to one
of $H_{\hbar, c}$. Moreover, the elliptic Casimir connection with values in
$V[0]$ coincides with elliptic KZ connection  with values
in this $H_{\hbar, c}$--module. 
\subsection{The rational Cherednik algebra}
In this subsection, we recall the definition of the rational Cherednik algebras. For details, see \cite{EG}. 

Let $W$ be a Weyl group and $\h$ be its reflection space. For any reflection $s\in W$, fix $\alpha_s\in \h^*$, such that $s(\alpha_s)=-\alpha_s$.
Write $S$ for the collection of linear functions
\[\{\pm \alpha_s \mid \text{$s$ is reflections in $W$}\}.\]
Let $c: S \to \C, s \mapsto c_s$ be a $W$-invariant function.
\begin{definition}
\label{def:RCA}
The rational Cherednik algebra $H_{\hbar, c}$ is the quotient of the algebra $\C W\ltimes T(\h\oplus \h^*)[\hbar]$ (where T denotes the tensor algebra) by the ideal
generated by the relations
\[[x, x'] = 0, \,\  [y, y'] = 0, \,\ 
[y, x]=\hbar\langle y, x\rangle-\sum_{s\in S}c_s\langle \alpha_s, y\rangle\langle \alpha_s^\vee, x\rangle s,\]
where $x, x'\in \h^*$, $y, y'\in \h$.
\end{definition}

\subsection{The elliptic KZ connection}

The elliptic KZ connection valued in the rational Cherednik algebra $H_{\hbar, c}$ is constructed by Calaque-Enriquez-Etingof,
 in \cite{CEE} for type $\sfA$ and by Toledano Laredo-Yang in \cite{TLY1} for any finite (reduced, crystallographic) root system $\Phi$. 

Let $\Treg= T\setminus \cup_{\alpha\in \Phi}T_{\alpha}$ be elliptic configuration space. This elliptic KZ connection $\nabla_{H_{\hbar, c}}$ is a connection on the vector bundle $\mathcal{P}_{\tau, n}$,  
whose the fiber $V$ is a finite-dimensional representation of the rational Cherednik algebra $H_{\hbar, c}$. 
Let
$\{u_i\}$ and $\{u^i\}$ be the dual basis of $\h$ and $\h^*$, the elliptic KZ connection takes the form:
\begin{equation}\label{conn:RCA}
\nabla_{H_{\hbar, c}}=d+\sum_{\alpha \in \Phi^+}\frac{2c_{\alpha}}{(\alpha|\alpha)}k(\alpha, \ad(\frac{ x(\alpha^\vee)}{2})|\tau) s_{\alpha} d\alpha
   -\sum_{\alpha \in \Phi^+} \frac{\hbar}{h^\vee} \frac{\theta'(\alpha|\tau)}{\theta(\alpha|\tau)} d\alpha
   +\sum_{i=1}^{n} y(u^i) du_i.
\end{equation}
where
\begin{itemize}
  \item $x(\alpha^\vee)\in \h$, $y(u^i)\in \h^*$ and $s_{\alpha}\in W$ is the reflection associated to the root $\alpha$. 
    \item $h^\vee$ is the dual Coxeter number.
\end{itemize}

The connection $\nabla_{H_{\hbar, c}}$ is flat and $W$-equivariant and its monodromy yields a one
parameter family of monodromy representations of \textit{the elliptic braid group} $\pi_1(\Treg/W)$. Furthermore, as shown in \cite{CEE, TLY1}, the monodromy representations factor through the double affine Hecke algebra. 

\subsection{Small $\g$-modules }
\label{sub:small rep}

Recall that a $\g$--module $V$ is small, if $2\alpha$ is not a weight of $V$ for any root $\alpha$ \cite{Br, Re, Re2}. 

\begin{lemma}\cite[Lemma 7.4]{TL-JAlg}
\label{lem:small rep}
If $V$ is an integrable, small $\g$--module, the following holds on the zero weight space $V[0]$
\[
\kappa_{\alpha}=(\alpha, \alpha) (1-s_{\alpha}),
\]
where the right hand side refers to the action of the reflection $s_{\alpha}\in W$ on $V[0]$.
\end{lemma}

\subsection{The elliptic KZ and elliptic Casimir connection}
In this subsection, we focus on the category of finite dimensional representations of $\DDg$, whose restriction to $\g$ is small.

Let $V$ be such module. We have the decomposition $\Phi^+=\Phi^+_{l}\sqcup \Phi^+_{s}$, 
where $\Phi^+_l$ is the set of positive long roots and $\Phi^+_s$ the set of positive short roots.  
Denote the corresponding dual Coxeter numbers of $\Phi_l$ and $\Phi_s$ by $h^\vee_l$ and $h^\vee_s$. 
For ADE case, $\Phi^+_s=\emptyset$, and $h^\vee_s=0$. 
Assume the central element $Z\in \DDg$ acts on $V$ by the scalar $\hbar-\frac{\lambda}{2} (2h^\vee_l+h^\vee_s)$, and $\frac{\lambda}{2}(\alpha, \alpha)^2=2c_s$. 
\begin{theorem}
\label{thm:KZ and Casimir}
(i): The canonical $W$--action on the zero weight space $V[0]$ together with the assignment
\[
x_u \mapsto Q(u), \,\  y_u \mapsto K(u)
\] yield an action of the rational Cherednik algebra on $V[0]$.

(ii): The elliptic Casimir connection with values in $\End(V[0])$ is equal to the sum of the elliptic KZ connection and the scalar valued one-form
\[\mathcal{A}=\frac{\lambda}{2}
\sum_{\alpha \in \Phi^+}\left(\frac{ 2h^\vee_l+h^\vee_s}{h^\vee} -(\alpha, \alpha)\right) \frac{\theta'(\alpha|\tau)}{\theta(\alpha|\tau)}   d\alpha \]
\end{theorem}
For the rest of this subsection, we prove Theorem \ref{thm:KZ and Casimir}. To begin with, we first need two lemmas. 
\begin{lemma}\cite[Lemma 10.4]{TLY1} \label{lem:phi and h}
For any root system $\Phi$, we have $(u|v)=\frac{1}{h^{\vee}}\sum_{\gamma\in \Phi^+} (\gamma, u)(\gamma, v)$.
\end{lemma}

\begin{lemma}
\label{lem:phi_ls}
We have $
\sum_{\alpha\in\Phi^+} (u, \alpha)(v, \alpha) (\alpha, \alpha)
=(2h^\vee_l+h^\vee_s)(u, v)$. 
\end{lemma}
\begin{proof}
For $\alpha\in \Phi^+_l$, we have $(\alpha, \alpha)=2$, and for $\alpha\in \Phi^+_s$, we have $(\alpha, \alpha)=1$. 
Then,
\begin{align*}
\sum_{\alpha\in\Phi^+} (u, \alpha)(v, \alpha) (\alpha, \alpha)
=2\sum_{\alpha\in\Phi^+_l} (u, \alpha)(v, \alpha) +
\sum_{\alpha\in\Phi^+_s} (u, \alpha)(v, \alpha)
=(2h^\vee_l+h^\vee_s)(u, v).
\end{align*}
The last equality follows from Lemma \ref{lem:phi and h}. 
\end{proof}
\begin{proof}[Proof of Theorem \ref{thm:KZ and Casimir}]
We rewrite the elliptic Casimir connection based on Lemma \ref{lem:small rep}.

\begin{align*}
\nabla_{\Ell, C}&=
d-\frac{\lambda}{2}\sum_{\alpha \in \Phi^+} k(\alpha, \ad(\frac{Q(\alpha^\vee)}{2})|\tau)( \kappa_\alpha)d\alpha-\sum_{\alpha \in \Phi^+}\frac{\theta'(\alpha|\tau)}{\theta(\alpha|\tau)}\frac{Z}{h^\vee}d\alpha+\sum_{i=1}^{n}K(u^i)du_i\\
&=
d-\frac{\lambda}{2}\sum_{\alpha \in \Phi^+} k(\alpha, \ad(\frac{Q(\alpha^\vee)}{2})|\tau)( \alpha, \alpha)(1-s_\alpha)d\alpha-\sum_{\alpha \in \Phi^+}\frac{\theta'(\alpha|\tau)}{\theta(\alpha|\tau)}\frac{Z}{h^\vee}d\alpha+\sum_{i=1}^{n}K(u^i)du_i\\
&=
d+\frac{\lambda}{2}\sum_{\alpha \in \Phi^+} k(\alpha, \ad(\frac{Q(\alpha^\vee)}{2})|\tau)( \alpha, \alpha)s_\alpha d\alpha
-\sum_{\alpha \in \Phi^+} \frac{\theta'(\alpha|\tau)}{\theta(\alpha|\tau)}
\Big(\frac{\lambda( \alpha, \alpha)}{2}
+\frac{Z}{h^\vee}
\Big)d\alpha
+\sum_{i=1}^{n}K(u^i)du_i.
\end{align*}
Consider the one-form:
\begin{align*}
\mathcal{A}
=\sum_{\alpha \in \Phi^+}\left(\frac{\hbar-Z}{h^\vee} -\frac{\lambda}{2}(\alpha, \alpha)\right) \frac{\theta'(\alpha|\tau)}{\theta(\alpha|\tau)}   d\alpha
\Omit{&=
\sum_{\alpha \in \Phi^+}\left(\frac{\frac{\lambda}{2} (2h^\vee_l+h^\vee_s)}{h^\vee} -\frac{\lambda}{2}(\alpha, \alpha)\right) \frac{\theta'(\alpha|\tau)}{\theta(\alpha|\tau)}   d\alpha
\\
&=\frac{\lambda}{2}
\sum_{\alpha \in \Phi^+}\left(\frac{ 2h^\vee_l+h^\vee_s}{h^\vee} -(\alpha, \alpha)\right) \frac{\theta'(\alpha|\tau)}{\theta(\alpha|\tau)}   d\alpha}.
\end{align*} 
By assumption, $Z$ acts on $V[0]$ by the scalar $\hbar-\frac{\lambda}{2} (2h^\vee_l+h^\vee_s)$.
The one form $\mathcal{A}$ on $V[0]$ is simplifies as follows. 
\begin{align*}
\mathcal{A}
=\sum_{\alpha \in \Phi^+}\left(\frac{\hbar-Z}{h^\vee} -\frac{\lambda}{2}(\alpha, \alpha)\right) \frac{\theta'(\alpha|\tau)}{\theta(\alpha|\tau)}   d\alpha
&=\frac{\lambda}{2}
\sum_{\alpha \in \Phi^+}\left(\frac{ 2h^\vee_l+h^\vee_s}{h^\vee} -(\alpha, \alpha)\right) \frac{\theta'(\alpha|\tau)}{\theta(\alpha|\tau)}   d\alpha.
\end{align*}

Using the assumption
$\frac{\lambda}{2}(\alpha, \alpha)^2=2c_s$, we have
\[
\nabla_{\Ell, C}-\mathcal{A}
=d+\sum_{\alpha \in \Phi^+}\frac{2c_s}{(\alpha|\alpha)}k(\alpha, \ad(\frac{ Q(\alpha^\vee)}{2})|\tau) s_{\alpha} d\alpha
-\sum_{\alpha \in \Phi^+} \frac{\hbar}{h^\vee} \frac{\theta'(\alpha|\tau)}{\theta(\alpha|\tau)} d\alpha
+\sum_{i=1}^{n}K(u^i)du_i.
\]
Note that $\mathcal{A}$ is a scalar valued one-form. Then, the flatness of $\nabla_{\Ell, C}$ implies the flatness of $\nabla_{\Ell, C}-\mathcal{A}$.
The connection $\nabla_{\Ell, C}-\mathcal{A}$ is also $W$--equivariant since both $\nabla_{\Ell, C}$ and $\mathcal{A}$ are. Define an action of the rational Cherednik algebra on $V[0]$ by:
\[
x_u \mapsto Q(u), \,\  y_u \mapsto K(u).
\]
We now check this assignment preserves the defining relations of $H_{\hbar, c}$. 
On $V[0]$, by the relation of $H_{\hbar, c}$, the action of $[y(u), x(v)]$ is the same as the action of 
\begin{equation}\label{eq:DH}
\hbar(u, v)-\sum_{\alpha\in \Phi^+} c_s(\alpha, u)(\alpha^\vee, v) s_{\alpha}
=\hbar(u, v)-\frac{\lambda}{2}\sum_{\alpha\in \Phi^+} (\alpha, \alpha)(\alpha, u)(\alpha, v) s_{\alpha},
\end{equation}
On the other hand, using the relation of $\DDg$, the action of $[K(u), Q(v)]$ is the same as the action of 
\begin{align}
\frac{\lambda}{2} \sum_{\alpha\in\Phi^+} (u, \alpha)(v, \alpha) \kappa_\alpha+(u, v)Z
&=\frac{\lambda}{2} \sum_{\alpha\in\Phi^+} (u, \alpha)(v, \alpha) (\alpha, \alpha) (1-s_{\alpha})+(u, v)Z \notag\\
&=\frac{\lambda}{2} \sum_{\alpha\in\Phi^+} (u, \alpha)(v, \alpha) (\alpha, \alpha)+(u, v)Z-\frac{\lambda}{2} \sum_{\alpha\in\Phi^+} (u, \alpha)(v, \alpha) (\alpha, \alpha) s_{\alpha}.\label{eq:H and D}
\end{align}
By Lemma \ref{lem:phi_ls}, and the assumption that $Z$ acts on $V[0]$ by $\hbar-\frac{\lambda}{2} (2h^\vee_l+h^\vee_s)$, we have
$\frac{\lambda}{2} \sum_{\alpha\in\Phi^+} (u, \alpha)(v, \alpha) (\alpha, \alpha)+(u, v)Z=\hbar(u, v)$. Therefore, \eqref{eq:H and D} coincides with \eqref{eq:DH}. 
This completes the proof. 
\end{proof}

\section{The dual pair $(\gl_k, \gl_n)$}\label{se:duality}

In this section, we establish a duality between the KZB connection associated to  $\gl_k$ and the
elliptic Casimir connection associated to  $\gl_n$. 

\subsection{Reductions}

In this section, we recall one realization of the KZB connection from \cite[\S 6.3]{CEE}. Let $\g=\gl_k$.
We have a decomposition $\gl_k=\h_k \oplus \n_k$, where $\h_k$ is the set of diagonal matrices, with
natural basis $\{E_{ii}\mid 1\leq i \leq k\}$, and $\n_k$ is the set of off-diagonal matrices, with natural
basis $\{E_{ij}\mid 1\leq i\neq j \leq k\}$.

Identify $\h^*_k$ with $\h_k$ by the nondegenerate bilinear form $\langle X, Y\rangle\mapsto \tr(XY)$.
Denote by $\h_k^{*\reg} \subset \h_k^*$ the subset of diagonal matrices with distinct eigenvalues, that is
$\h_k^{*\reg}=\{\lambda\in \h^*\mid \prod_{1\leq i\neq j\leq k}(\lambda_i-\lambda_j)\neq 0\}$, where $\lambda=\sum_{i=1}^k\lambda_iE_{ii}$.

Let $\Diff(\h_k)$ be the algebra of algebraic differential operators on $\h_k$.
It has generators $x_{i}$, $\partial_{i}$, for $1\leq i \leq k$, and relations:
$E_{ii}\mapsto x_{i}$, $E_{ii}\mapsto \partial_{i}$ are linear, and
$[x_{i}, x_{j}]=[\partial_{i}, \partial_{j}]=0$,
$[\partial_{i}, x_{j}]=\delta_{ij}$.

Set $P:=\prod_{1\leq i\neq j\leq k}(x_i-x_j)\in S(\h)\subset \Diff(\h)$. 
The embedding $\h_k\subset \gl_k$ induces a map 
\[\h_k \to B_n:=\Diff(\h_k)[\frac{1}{P}]\otimes U(\gl_k)^{\otimes n},\,\ h \mapsto 1\otimes \sum_{a=1}^{n} h^{(a)}. \]
Denote by $\h_k^{\diag}$ its image.
\Omit{
Let
\[r=\sum_{1\leq i\neq j \leq k}\frac{1}{x_j-x_i}\otimes E_{ij}\otimes E_{ji}\in S(\h)[\frac{1}{P}]\otimes \n^{\otimes 2}.\]
}

\begin{definition}\cite[Sec. 6.3]{CEE}
\label{def:Hecke}
The Hecke algebra $\mathcal{H}(\g_k, \h_k)$ is defined to be
\[
\{
x\in B_n \mid \sum_{i=1}^{n} h^{(i)}( x)\in B_n \h_k^{\diag}, \,\ \text{for any $h\in \h_k$}
\}/B_n \h_k^{\diag}.
\]
\end{definition}
Let $V_1, \dots, V_n$ be $\g$--modules, then $S(\h_k)[1/P]\otimes (\otimes_{i=1}^n V_i)$ is a module over $B_n$.
Let $(S(\h_k)[1/P]\otimes (\otimes_{i=1}^n V_i))^{\h_k}=(S(\h_k)[1/P]\otimes (\otimes_{i=1}^n V_i)) [0]$ be the zero weight space.
Then, $(S(\h)[1/P]\otimes (\otimes_{i=1}^n V_i))^{\h_k}$ is a module over $B_n/B_n\h_k^{\diag}$, hence a module over 
$\mathcal{H}(\gla{k}, \h_k)$.

\begin{definition}\cite[Sect. 1.1]{CEE}
Let $\AAell{n}$ be the Lie algebra with generators $x_i, y_i$ $(i=1, \cdots, n)$, and $t_{ij}$ $(i\neq j \in \{1, \cdots, n\})$ and relations
\begin{align*}
&t_{ij}=t_{ji}, [t_{ij}, t_{ik}+t_{jk}]=0, [t_{ij, t_{kl}}]=0\\
&[x_i, y_j]=t_{ij}, [x_i, x_j]=[y_i, y_j]=0, [x_i, y_i]=-\sum_{\{j\mid j\neq i\}} t_{ij}\\
&[x_i, t_{jk}]=[y_i, t_{jk}]=0, [x_i+x_j, t_{ij}]=[y_i+y_j, t_{ij}]=0
\end{align*}
($i, j, k, l$ are distinct). In this Lie algebra $\AAell{n}$, $\sum_i x_i$ and $\sum_i y_i$ are central. We then define 
\[
\bar{\mathfrak{t}}_{1, n}:=\AAell{n}/( \sum_i x_i, \sum_i y_i). 
\]
\end{definition}

\begin{prop}\cite[Prop. 41]{CEE}
\label{prop:map to Hecke}
\begin{enumerate}
\item
There is an algebra homomorphism $\AAell{n} \to \mathcal{H}(\gla{k}, \h_k) \subset B_n/B_n \h_k^{\diag}$ given by:
\begin{align*}
&x_i \mapsto \sum_{a=1}^k x_a\otimes E_{aa}^{(i)},\,\ \phantom{1234}\,\ 
y_i \mapsto -\sum_{a=1}^k \partial_a\otimes E_{aa}^{(i)}+
\sum_{j=1}^n\sum_{1\leq a\neq b\leq k} \frac{1}{x_b-x_a}\otimes E_{ab}^{(i)}E_{ba}^{(j)},\\
&t_{ij}\mapsto \sum_{1\leq a, b \leq k}E_{ab}^{(i)}E_{ba}^{(j)}, \,\ \text{for $1\leq i\neq j\leq n$. }
\end{align*}
\item The homomorphism factors through the quotient $\AAell{n}\twoheadrightarrow \bar{\mathfrak{t}}_{1, n}$.
\Omit{
, such that, we have the following commutative diagram
\[\xymatrix@R=1em @C=1em{
\AAell{n} \ar[rr] \ar@{->>}[dr]& &\mathcal{H}(\gla{k}, \h_k)\\
&\bar{\mathfrak{t}}_{1, n}\ar[ur]&
}
\]}
\end{enumerate}
\end{prop}
\redtext{In \cite[Prop. 41]{CEE}, an algebra homomorphism from  $\bar{\mathfrak{t}}_{1, n}$ to 
the Hecke algebra $\mathcal{H}(\mathfrak{g}, \h)$ associated to an arbitrary semisimple Lie algebra $\g$ is constructed. 
For the convenience of the reader, we provide a proof in the special case when the Hecke algebra is $\mathcal{H}(\gla{k}, \h_k)$. 
From the proof, it is clear that the map $\AAell{n} \to B_n/B_n\h^{\diag}$ does not factor through $B_n$. 
Indeed, as we will see that, the relations
$[x_i, y_i]=-\sum_{ \{j \mid j\neq i\}} t_{ij}$ and $[y_i, y_j]=0$ of $\AAell{n}$ are not preserved
under the map $\AAell{n}\to B_n$. }
\Omit{Indeed, 
\[
[x_i, y_i]+\sum_{ \{j \mid j\neq i\}} t_{ij}\mapsto 
\sum_{j=1}^n \sum_{1\leq a\leq k} 
E_{aa}^{(i)}E_{aa}^{(j)},\,\  \text{which is non-zero in $B_n$.}\]
}
We now prove the Proposition. 
\begin{proof}
The relations $[x_a, x_b]=0$, $t_{ab}=t_{ba}$, $[x_a, t_{bc}]=0$, where $a, b, c$ are distinct, are obviously preserved.
We now check that the relation $[x_a, y_b]=t_{ab}$ for $a\neq b$ is preserved under the map. We have
\begin{align*}
[x_a, y_b]=&[\sum_m x_{E_{mm}}\otimes E_{mm}^{(a)}, -\sum_i \partial_i\otimes E_{ii}^{(b)}+\sum_c \sum_{1\leq i\neq j \leq k} \frac{1}{x_j-x_i}\otimes E_{ij}^{(b)}E_{ji}^{(c)}]\\
=&\sum_i E_{ii}^{(a)}E_{ii}^{(b)}+[\sum_m x_{E_{mm}}\otimes E_{mm}^{(a)}, \sum_{1\leq i\neq j \leq k} \frac{1}{x_j-x_i}\otimes E_{ij}^{(b)}E_{ji}^{(a)}]\\
=&\sum_i E_{ii}^{(a)}E_{ii}^{(b)}+\sum_{m\neq i}E_{im}^{(b)}E_{im}^{(a)}=t_{ab}.
\end{align*}
We check that the map preserves the relation
$[x_a, y_a]=-\sum_{ \{b \mid b\neq a\}} t_{ab}$ as follows. 
\begin{align*}
[x_a, y_a]
=&[\sum_m x_{E_{mm}}\otimes E_{mm}^{(a)}, -\sum_i \partial_i\otimes E_{ii}^{(a)}+\sum_c \sum_{1\leq i\neq j \leq k} \frac{1}{x_j-x_i}\otimes E_{ij}^{(a)}E_{ji}^{(c)}]\\
=&\sum_i [\partial_i\otimes E_{ii}^{(a)}, x_i\otimes E_{ii}^{(a)}]
+\sum_{m, c} \sum_{1\leq i\neq j \leq k} \frac{x_{E_{mm}}}{x_j-x_i}\otimes [E_{mm}^{(a)}, E_{ij}^{(a)}E_{ji}^{(c)}]\\
=&\sum_i  E_{ii}^{(a)}
+\sum_{m, c} \sum_{1\leq i\neq j \leq k} \frac{x_{E_{mm}}}{x_j-x_i}\otimes 
\Big(
(\delta_{mi}E_{mj}^{(a)}-\delta_{mj}E_{im}^{(a)})E_{ji}^{(c)} 
+ \delta_{ac}E_{ij}^{(a)}( \delta_{mj} E_{mi}^{(a)} -\delta_{mi}E_{jm}^{(a)})\Big)\\
=&\sum_i  E_{ii}^{(a)}
+\sum_{c} \sum_{1\leq i\neq j \leq k} \frac{x_i}{x_j-x_i}\otimes 
\Big(E_{ij}^{(a)}E_{ji}^{(c)} - \delta_{ac}E_{ij}^{(a)}E_{ji}^{(a)}\Big)\\
&+\sum_{c} \sum_{1\leq i\neq j \leq k} \frac{x_j}{x_j-x_i}\otimes 
\Big(
-E_{ij}^{(a)}E_{ji}^{(c)} 
+ \delta_{ac}E_{ij}^{(a)}(E_{ji}^{(a)} )\Big)\\
=&\sum_i  E_{ii}^{(a)}E_{ii}^{(a)}
-\sum_{\{c\mid c\neq a\}} \sum_{1\leq i\neq j \leq k} 
\Big(E_{ij}^{(a)}E_{ji}^{(c)}\Big)\\
=&\sum_i  E_{ii}^{(a)}E_{ii}^{(a)}
-\sum_{\{c\mid c\neq a\}} \sum_{1\leq i, j \leq k} 
E_{ij}^{(a)}E_{ji}^{(c)}
+\sum_{\{c\mid c\neq a\}} \sum_{1\leq i\leq k} 
E_{ii}^{(a)}E_{ii}^{(c)}
\\
=&-\sum_{\{c\mid c\neq a\}} t_{ac}
+\sum_{c} \sum_{1\leq i\leq k} 
E_{ii}^{(a)}E_{ii}^{(c)}
\end{align*}
For any $1\leq i\leq k$, the element $\sum_{1\leq c\leq n} E_{ii}^{(a)}E_{ii}^{(c)}=E_{ii}^{(a)}\cdot (\sum_{1\leq c\leq n} E_{ii}^{(c)})\in B_n \h_k^{\diag}$ is zero in $B_n/B_n \h_k^{\diag}$.
Thus, we have the relation
\[
[x_a, y_a]=-\sum_{ \{b \mid b\neq a\}} t_{ab}.
\]
We check that the relation $[y_a, t_{bc}]=0$ is mapped to zero under the map. We have
\begin{align*}
[y_a, t_{bc}]=&[-\sum_i \partial_i\otimes E_{ii}^{(a)}+\sum_m \sum_{1\leq i\neq j \leq k} \frac{1}{x_j-x_i}\otimes E_{ij}^{(a)}E_{ji}^{(m)}, \sum_{ij}E_{ij}^{(b)}E_{ji}^{(c)}]\\
=&[\sum_{1\leq i\neq j \leq k} \frac{1}{x_j-x_i}\otimes E_{ij}^{(a)}(E_{ji}^{(b)}+E_{ji}^{(c)}), \sum_{ij}E_{ij}^{(b)}E_{ji}^{(c)}]=0
\end{align*}
We now show the relation $[y_a, y_b]=0$ is preserved under the map. We will use the following identity: for any $A, B, C, D$,
\[
[A\otimes B, C\otimes D]=[A, C]\otimes BD+CA \otimes [B, D].
\]
For $1\leq a\neq b \leq n$, by definition, we have
\begin{align}
[y_a, y_b] 
=&\Big[-\sum_{i=1}^k \partial_i\otimes E_{ii}^{(a)}, \sum_{c=1}^n\sum_{1\leq i\neq j \leq k} \frac{1}{x_j-x_i}\otimes E_{ij}^{(b)}E_{ji}^{(c)}\Big] \tag{A} \label{A}\\
&+\Big[\sum_{c=1}^n\sum_{1\leq i\neq j \leq k} \frac{1}{x_j-x_i}\otimes E_{ij}^{(a)}E_{ji}^{(c)},
-\sum_{i=1}^k \partial_i\otimes E_{ii}^{(b)}\Big] \tag{B} \label{B}\\
&+
\Big[\sum_{c=1}^n\sum_{1\leq i\neq j \leq k} \frac{1}{x_j-x_i}\otimes E_{ij}^{(a)}E_{ji}^{(c)},
\sum_{d=1}^n\sum_{1\leq i\neq j \leq k} \frac{1}{x_j-x_i}\otimes E_{ij}^{(b)}E_{ji}^{(d)}\Big] \tag{C} \label{C}
\end{align}
We first compute the term \eqref{A}. We have
\begin{align*}
\eqref{A}=&-\sum_{i=1}^k \sum_{c=1}^n\sum_{1\leq s\neq t \leq k}\Big[\partial_i, \frac{1}{x_t-x_s}\Big]\otimes E_{ii}^{(a)}E_{st}^{(b)}E_{ts}^{(c)}
-\sum_{i=1}^k \sum_{c=1}^n\sum_{1\leq s\neq t \leq k} \frac{1}{x_t-x_s} \cdot \partial_i \otimes \Big[E_{ii}^{(a)}, E_{st}^{(b)}E_{ts}^{(c)}\Big]\\
=&-\sum_{c=1}^n\sum_{1\leq s\neq t \leq k}\frac{-1}{(x_t-x_s)^2}\otimes E_{tt}^{(a)}E_{st}^{(b)}E_{ts}^{(c)}
-\sum_{c=1}^n\sum_{1\leq s\neq t \leq k}\frac{1}{(x_t-x_s)^2}\otimes E_{ss}^{(a)}E_{st}^{(b)}E_{ts}^{(c)}\\
&-\sum_{1\leq s\neq t \leq k} \frac{1}{x_t-x_s} \cdot \partial_t \otimes E_{st}^{(b)}E_{ts}^{(a)}
+ \sum_{1\leq s\neq t \leq k} \frac{1}{x_t-x_s} \cdot \partial_s\otimes E_{st}^{(b)}E_{ts}^{(a)}
\end{align*}
By symmetry of \eqref{A} and \eqref{B}, we get
\begin{align*}
\eqref{B}=&\sum_{c=1}^n\sum_{1\leq s\neq t \leq k}\frac{-1}{(x_t-x_s)^2}\otimes E_{tt}^{(b)}E_{st}^{(a)}E_{ts}^{(c)}
+\sum_{c=1}^n\sum_{1\leq s\neq t \leq k}\frac{1}{(x_t-x_s)^2}\otimes E_{ss}^{(b)}E_{st}^{(a)}E_{ts}^{(c)}\\
&+\sum_{1\leq s\neq t \leq k} \frac{1}{x_t-x_s} \partial_t \otimes E_{st}^{(a)}E_{ts}^{(b)}
- \sum_{1\leq s\neq t \leq k} \frac{1}{x_t-x_s} \partial_s \otimes E_{st}^{(a)}E_{ts}^{(b)}
\end{align*}
Thus, after cancelation, we conclude
\begin{align*}
\eqref{A}+\eqref{B}=&-\sum_{c=1}^n\sum_{1\leq s\neq t \leq k}\frac{-1}{(x_t-x_s)^2}\otimes E_{tt}^{(a)}E_{st}^{(b)}E_{ts}^{(c)}
-\sum_{c=1}^n\sum_{1\leq s\neq t \leq k}\frac{1}{(x_t-x_s)^2}\otimes E_{ss}^{(a)}E_{st}^{(b)}E_{ts}^{(c)}\\
+&\sum_{c=1}^n\sum_{1\leq s\neq t \leq k}\frac{-1}{(x_t-x_s)^2}\otimes E_{tt}^{(b)}E_{st}^{(a)}E_{ts}^{(c)}
+\sum_{c=1}^n\sum_{1\leq s\neq t \leq k}\frac{1}{(x_t-x_s)^2}\otimes E_{ss}^{(b)}E_{st}^{(a)}E_{ts}^{(c)}
\end{align*}
We now compute the term \eqref{C}. We have
\begin{align}
\eqref{C}=&\sum_{c, d=1}^n\sum_{1\leq i\neq j, s\neq t \leq k} \frac{1}{(x_j-x_i)(x_t-x_s)}\otimes \Big[ E_{ij}^{(a)}E_{ji}^{(c)}, E_{st}^{(b)}E_{ts}^{(d)}\Big] \notag\\
=&\sum_{c=1}^n \left(\sum_{i\neq j, i\neq t}\frac{1}{(x_j-x_i)(x_t-x_i)}\otimes E_{ij}^{(a)}E_{jt}^{(b)}E_{ti}^{(c)}
-\sum_{i\neq j, j\neq s}\frac{1}{(x_j-x_i)(x_j-x_s)}\otimes E_{ij}^{(a)}E_{si}^{(b)}E_{js}^{(c)}\right) \notag\\
&+\sum_{c=1}^n \left(\sum_{i\neq j, i\neq s}\frac{1}{(x_j-x_i)(x_i-x_s)}\otimes E_{ij}^{(a)}E_{si}^{(b)}E_{js}^{(c)}
-\sum_{i\neq j, j\neq t}\frac{1}{(x_j-x_i)(x_t-x_j)}\otimes E_{ij}^{(a)}E_{jt}^{(b)}E_{ti}^{(c)}\right)\notag\\
&+\sum_{c=1}^n \left(\sum_{i\neq j, j\neq s}\frac{1}{(x_j-x_i)(x_j-x_s)}\otimes E_{is}^{(a)}E_{sj}^{(b)}E_{ji}^{(c)}
-\sum_{i\neq j, i\neq t}\frac{1}{(x_j-x_i)(x_t-x_i)}\otimes E_{tj}^{(a)}E_{it}^{(b)}E_{ji}^{(c)}\right) \label{C: inter}
\end{align}
When the indices $i, j, s$ or $i, j, t$ are distinct, using the identity
\[
\frac{1}{(x_j-x_i)(x_t-x_i)}-\frac{1}{(x_j-x_i)(x_t-x_j)}=-\frac{1}{(x_t-x_j)(x_t-x_i)}, 
\]
it is obvious that the corresponding summands in \eqref{C: inter} add up to zero.
Thus, we simplify the term \eqref{C} as follows. 
\begin{align*}
\eqref{C}=&\sum_{c=1}^n \left(\sum_{i\neq j}\frac{1}{(x_j-x_i)^2}\otimes E_{ij}^{(a)}E_{jj}^{(b)}E_{ji}^{(c)}
-\sum_{i\neq j}\frac{1}{(x_j-x_i)^2}\otimes E_{ij}^{(a)}E_{ii}^{(b)}E_{ji}^{(c)}\right)\\
&+\sum_{c=1}^n \left(\sum_{i\neq j}\frac{-1}{(x_j-x_i)^2}\otimes E_{ij}^{(a)}E_{ji}^{(b)}E_{jj}^{(c)}
-\sum_{i\neq j}\frac{-1}{(x_j-x_i)^2}\otimes E_{ij}^{(a)}E_{ji}^{(b)}E_{ii}^{(c)}\right)\\
&+\sum_{c=1}^n \left(\sum_{i\neq j}\frac{1}{(x_j-x_i)^2}\otimes E_{ii}^{(a)}E_{ij}^{(b)}E_{ji}^{(c)}
-\sum_{i\neq j}\frac{1}{(x_j-x_i)^2}\otimes E_{jj}^{(a)}E_{ij}^{(b)}E_{ji}^{(c)}\right)
\end{align*}
To summarise, we have
\begin{align}\label{eq:check yy}
[y_a, y_b]=&\eqref{A}+\eqref{B}+\eqref{C} \notag\\
=&\sum_{c=1}^n \left(\sum_{i\neq j}\frac{-1}{(x_j-x_i)^2}\otimes E_{ij}^{(a)}E_{ji}^{(b)}E_{jj}^{(c)}
-\sum_{i\neq j}\frac{-1}{(x_j-x_i)^2}\otimes E_{ij}^{(a)}E_{ji}^{(b)}E_{ii}^{(c)}\right)
\end{align}
The above formula shows that $[y_a, y_b]\in B_n\h_k^{\diag}$. Therefore, the relation $[y_a, y_b]=0$ is preserved under the map.
It is clear that the map from $\AAell{n}$ to $B_n/B_n\h_k^{\diag}$ factors through the Hecke algebra $\mathcal{H}(\gla{k}, \h_k)$. 
Indeed the images of the generators $x_a, y_b$ and $t_{ab}$ lie in $\mathcal{H}(\gla{k}, \h_k)$.
This completes the proof of (1). 

To prove (2), we check the two relations $\sum_{i=1}^{n} x_i=0$, and $\sum_{i=1}^{n} y_i=0$. 
We have 
\[
\sum_{i=1}^{n} x_i\mapsto \sum_{i=1}^n \sum_{a=1}^k x_a\otimes E_{aa}^{(i)}= \sum_{a=1}^k x_a\otimes (\sum_{i=1}^n E_{aa}^{(i)}),
\] which is in $B_n \h_{k}^{\diag}$, therefore, the map preserves the relation $\sum_{i=1}^{n} x_i=0$. 
To check the map preserves the relation $\sum_{i=1}^{n} y_i=0$. We have
\begin{align*}
\sum_{i=1}^{n} y_i \mapsto 
&-\sum_{i=1}^n \sum_{a=1}^k \partial_a\otimes E_{aa}^{(i)}
+\sum_{1\leq i, j \leq n} \sum_{ 1\leq a\neq b\leq k} \frac{1}{x_b-x_a} \otimes E_{ab}^{(i)}E_{ba}^{(j)}\\
=& -\sum_{i=1}^n \sum_{a=1}^k \partial_a\otimes E_{aa}^{(i)}
+\sum_{1\leq i\leq n} \sum_{ 1\leq a\neq b\leq k} \frac{1}{x_b-x_a} \otimes E_{ab}^{(i)}E_{ba}^{(i)}\\
=& -\sum_{i=1}^n \sum_{a=1}^k \partial_a\otimes E_{aa}^{(i)}
+\frac{1}{2}\Big(
\sum_{1\leq i\leq n} \sum_{ 1\leq a\neq b\leq k} \frac{1}{x_b-x_a} \otimes E_{ab}^{(i)}E_{ba}^{(i)}
+\sum_{1\leq i\leq n} \sum_{ 1\leq a\neq b\leq k} \frac{1}{x_a-x_b} \otimes E_{ba}^{(i)}E_{ab}^{(i)}\Big)\\
=& -\sum_{i=1}^n \sum_{a=1}^k \partial_a\otimes E_{aa}^{(i)}
+\frac{1}{2}
\sum_{1\leq i\leq n} \sum_{ 1\leq a\neq b\leq k} \frac{1}{x_b-x_a} \otimes (E_{aa}-E_{bb})^{(i)}
\end{align*}
\end{proof}
The above term lies in $B_n \h_{k}^{\diag}$. This completes the proof.

Let $C(\mathcal{E}, n)$ be the configuration space of $n$ unordered points on the elliptic curve. The map in Proposition \ref{prop:map to Hecke} gives the following realization of the KZB connection. 
\begin{theorem}\cite[Sect. 6.3]{CEE}
\label{conn:gl_k}
The following connection on $C(\mathcal{E}, n)$ valued in the Hecke algebra $\mathcal{H}(\gla{k}, \h_k)$
\begin{align}\label{equ:CEE KZB}
\nabla_{\mathcal{H}(\gla{k}, \h_k)}=d&-\sum_{1\leq i\neq j\leq n}\sum_{1\leq a, b\leq k}
k(z_i-z_j, \ad (\Sigma_{c=1}^k x_c E_{cc}^{(i)})|\tau)(E_{ab}^{(i)} E_{ba}^{(j)})dz_i\\
&+\sum_{1\leq i, j\leq n}\sum_{1\leq a\neq b\leq k} \frac{E_{ab}^{(i)} E_{ba}^{(j)}}{ x_b-x_a}dz_i
-\sum_{i=1}^n \sum_{a=1}^k \partial_a E_{aa}^{(i)}dz_i, \notag
\end{align}
is flat and $\mathfrak{S}_n$--equivariant.  
\end{theorem}

Let $M_{k,n}$ be the vector space of $k\times n$ matrices and let $\C[M_{k,n}]$ be the ring of regular functions on $M_{k,n}$. Thus, $\C[M_{k,n}]$ is a polynomial ring with $kn$ variables. 
The infinite dimensional vector space $\C[\h_k^{\reg}]\otimes \C[M_{k,n}]$ is a module over $B_n=\Diff(\h_k^{\reg})\otimes U[\gl_k]^{\otimes n}$. The corresponding zero weight space $(\C[\h_k^{\reg}]\otimes \C[M_{k,n}])[0]$ is a module over the Hecke algebra $\mathcal{H}(\gla{k}, \h_k)$. 
The connection $\nabla_{\mathcal{H}(\gla{k}, \h_k)}$ in Theorem \ref{conn:gl_k} can be valued in the zero weight space $(\C[\h_k^{\reg}]\otimes \C[M_{k,n}])[0]$ which carries an action of $\mathcal{H}(\gla{k}, \h_k)$. 
\subsection{The $(\gla{k}, \gla{n})$ duality}
\label{sub:nkduality}
In this subsection, we construct an action of $D_{\lambda, \beta}(\sl{n})$ on the space $(\C[\h_k^{\reg}]\otimes \C[M_{k,n}])$, when $\beta=-\frac{n}{4}\lambda$. Using such action, we deduce a $(\gla{k}, \gla{n})$ duality in the elliptic setting. 

The group $\GL_k\times \GL_n$ acts on $\C[M_{k,n}]$ by
\[
(g_k, g_n)p(x)=p(g_k^t x g_n), \,\, \text{for $(g_k, g_n)\in \GL_k\times \GL_n$, and $p(x)\in \C[M_{k,n}]$. }
\]
It induces an action of the dual pair $(\gla{k}, \gla{n})$ of the corresponding Lie algebras. To distinguish between the elements of the Lie algebras $\gla{n}$ and $\gla{k}$, we denote by $X^{(p)}$ the elements of $\mathfrak{gl}_p$. In this notation, we have 
\[
(X^{(p)})^{(i)}=1\otimes \cdots \otimes X^{(p)}\otimes \cdots\otimes 1\in U(\mathfrak{gl}_p)^{\otimes n},\] where $X^{(p)}$ lies in the $i$--th copy of $U(\mathfrak{gl}_p)$.

Write $\mathbb{C}[M_{k,n}]=\mathbb{C}[x_{a, j}]_{1\leq a\leq k, 1\leq j\leq n}$. We have the isomorphism
\[
\mathbb{C}[x_{a, 1}]_{1\leq a\leq k}\otimes\cdots \otimes \mathbb{C}[x_{a, n}]_{1\leq a\leq k} \cong \mathbb{C}[M_{k,1}]^{\otimes n}
\cong \mathbb{C}[M_{k,n}] \cong \mathbb{C}[M_{1,n}]^{\otimes k} \cong 
\mathbb{C}[x_{1, j}]_{1\leq j\leq n}\otimes\cdots \otimes \mathbb{C}[x_{k, j}]_{1\leq j\leq n}.
\]
The action of $(\gla{k}, \gla{n})$ on $\mathbb{C}[M_{k,n}]$ is given by
\[
(E_{ab}^{(k)})^{(i)}\mapsto x_{ai}\partial_{bi}, \,\ 
(E_{ij}^{(n)})^{(a)} \mapsto x_{ai}\partial_{aj}. 
\]
\begin{prop}(\cite[Sect. 6]{GTL-sl2}, \cite[Section 3]{TL-Duke})
\label{prop:dual pair}
The following identities hold on $\C[M_{k,n}]$.
\begin{enumerate}
  \item For $1\leq a \leq k, 1\leq i\leq n$, we have $(E_{aa}^{(k)})^{(i)}=(E_{ii}^{(n)})^{(a)}$;
  \item If $1\leq i\neq j \leq n$ and $1\leq a\neq b \leq k$, then $(E_{ab}^{(k)})^{(i)}(E_{ba}^{(k)})^{(j)}=(E_{ij}^{(n)})^{(a)}(E_{ji}^{(n)})^{(b)}$;
  \item If $1\leq a\neq b\leq k, 1\leq i\leq n$, then $(E_{ab}^{(k)})^{(i)}(E_{ba}^{(k)})^{(i)}=(E_{ii}^{(n)})^{(a)}(E_{ii}^{(n)})^{(b)}+(E_{ii}^{(n)})^{(a)}$.
\end{enumerate}
\end{prop}
\begin{proof}
We only show the last equality. 
We have 
\begin{align*}
(E_{ab}^{(k)})^{(i)}(E_{ba}^{(k)})^{(i)}
\mapsto& x_{ai}\partial_{bi} x_{bi}\partial_{ai}
=x_{ai}\partial_{ai}\partial_{bi} x_{bi}=
x_{ai}\partial_{ai} x_{bi}\partial_{bi}
+x_{ai}\partial_{ai},
\end{align*}
which is the same as the action of $(E_{ii}^{(n)})^{(a)}(E_{ii}^{(n)})^{(b)}+(E_{ii}^{(n)})^{(a)}$. 
\Omit{The action of the dual pair $(\gla{n}, \gla{k})$ is given by mapping the elementary matrix $E_{ab}^{(k)}, E_{ij}^{(n)}$ to
\[
E_{ab}^{(k)}\mapsto \sum_{j=1}^{n}x_{aj}\partial_{bj},\,\ 
E_{ij}^{(n)}\mapsto \sum_{a=1}^{k}x_{ai}\partial_{aj}.
\]
We have $(E_{ii}^{(n)})^{(a)}=x_{ai}\partial_{ai}=(E_{aa}^{(k)})^{(i)}$, the first identity follows.

If $a\neq b$, then:
$(E_{ab}^{(k)})^{(i)}(E_{ba}^{(k)})^{(j)}=x_{ai}\partial_{bi}x_{bj}\partial_{aj}$
and $(E_{ij}^{(n)})^{(a)}(E_{ji}^{(n)})^{(b)}=x_{ai}\partial_{aj}x_{bj}\partial_{bi}$
as operators on $\C[M_{k,n}]$.

So when $i\neq j$, we have $x_{ai}\partial_{bi}x_{bj}\partial_{aj}=x_{ai}\partial_{aj}x_{bj}\partial_{bi}$, which shows the second identity.

We also have $x_{ai}\partial_{bi}x_{bi}\partial_{ai}=x_{ai}\partial_{ai}x_{bi}\partial_{bi}+x_{ai}\partial_{ai}$, which gives the third identity.}
\end{proof}
\begin{lemma}
Let $1\leq a\leq k$, the following holds on $\C[M_{k, n}]$. 
\[
(E_{ij}^{(n)})^{(a)}(E_{pq}^{(n)})^{(a)}=(E_{iq}^{(n)})^{(a)}(E_{pj}^{(n)})^{(a)}
-\delta_{pq} (E_{ij}^{(n)})^{(a)}
+\delta_{jp} (E_{iq}^{(n)})^{(a)}, \] 
for $E_{ij}, E_{pq}\in \gla{n}$. 
\end{lemma}
\begin{proof}
We have
\begin{align*}
(E_{ij}^{(n)})^{(a)}(E_{pq}^{(n)})^{(a)}\mapsto&
x_{ai}\partial_{aj} x_{ap}\partial_{aq} 
=x_{ai} x_{ap}\partial_{aj}\partial_{aq} 
+\delta_{jp} x_{ai} \partial_{aq} \\
=&x_{ai} \partial_{aq}x_{ap} \partial_{aj}
-\delta_{pq} x_{ai} \partial_{aj}
+\delta_{jp} x_{ai} \partial_{aq} 
\end{align*}
This is the same as the action of 
$(E_{iq}^{(n)})^{(a)}(E_{pj}^{(n)})^{(a)}
-\delta_{pq} (E_{ij}^{(n)})^{(a)}
+\delta_{jp} (E_{iq}^{(n)})^{(a)}$. 
\end{proof}

Let $\h_{k}\subset \sl{k}$ be the Cartan subalgebra of $ \sl{k}$. 
\begin{theorem} \label{prop: action of D(sl)}
There is an action of the deformed double current algebra $D_{-1, \frac{n}{4}}(\sl{n})$ on $\C[\h_{k}^{\reg}]\otimes \C[M_{k,n}]$ such that, for $1\leq i\neq j\leq n$, $E_{ij}$ acts by $1\otimes \sum_{a=1}^k (E_{ij}^{(n)})^{(a)}$ and
\begin{align*}
&K(E_{ij}) \text{ acts by $\sum_{a=1}^k  x_a\otimes (E_{ij}^{(n)})^{(a)}$},\\
&Q(E_{ij}) \text{ acts by
$-\sum_{a=1}^k \partial_a\otimes (E_{ij}^{(n)})^{(a)}+\sum_{1\leq a\neq b \leq k} \frac{1}{x_b-x_a}\otimes (\sum_{e=1}^n(E_{ie}^{(n)})^{(a)}(E_{ej}^{(n)})^{(b)}+(E_{ij}^{(n)})^{(a)})$},\\
&P(E_{ij}) \text{ acts by
$-\sum_{a=1}^kY_{a} \otimes E_{ij}^{(a)}
+\sum_{1\leq a\neq b \leq k} \frac{x_a}{x_b-x_a} \otimes E_{ij}^{(a)}+\frac{1}{2}\sum_{1\leq a\neq b \leq k} \frac{x_b+x_a}{x_b-x_a} \otimes \Big(\sum_{e=1}^n E_{ie}^{(a)}E_{ej} ^{(b)}\Big)$,}\\
& Z_n \text{ acts by the scalar $2(n+1)$}, 
\end{align*}
where $Y_a:=\frac{\partial_a x_a+x_a \partial_a}{2}\in \Diff(\h_k)$, for $1\leq a\leq k$.
\end{theorem}
\begin{remark}
We have an isomorphism $D_{a\lambda, a\beta}(\sl{n})\cong D_{\lambda, \beta}(\sl{n})$, for $a\neq 0$. The isomorphism is given by 
\[X\mapsto X, \,\ K(X)\mapsto a_1 K(X), \,\ 
Q(X)\mapsto a_2 Q(X), \,\ P(X)\mapsto a_1a_2 P(X),
\]
where $a=a_1a_2$. 
Therefore, Theorem \ref{prop: action of D(sl)} gives an action of $D_{\lambda, -\frac{n}{4}\lambda}(\sl{n})$ on $\C[\h_{k}^{\reg}]\otimes \C[M_{k,n}]$.
\end{remark}

Theorem \ref{prop: action of D(sl)} gives the following $(\gla{k}, \gla{n})$ duality in the elliptic case.
 \begin{theorem}\label{duality}
Under the identification
\[
\C[M_{k, 1}]^{\otimes n} \cong \C[M_{k, n}] \cong \C[M_{1, n}] ^{\otimes k},
\]
the KZB connection for $\gla{k}$ in Theorem \ref{conn:gl_k} with values in $\C[\h_{k}^{\reg}]\otimes \C[M_{k, 1}]^{\otimes n}[0]$ coincides with
the sum of
\begin{enumerate}
\item the elliptic Casimir connection with two parameters for $\sl{n}$ with values in $\C[\h_{k}^{\reg}]\otimes \C[M_{1, n}]^{\otimes k}[0]$ and
\item the closed one-form given by
\[
\mathcal{A}
=\sum_{1\leq i< j \leq n}\frac{\theta'(z_i-z_j|\tau)}{\theta(z_i-z_j|\tau)}\Big(\frac{E_{ii}+E_{jj}}{n}-\frac{1}{n^2} \Big)d z_{ij}.
\]
\end{enumerate}
\end{theorem}
\begin{proof}
Let $\epsilon_1, \cdots, \epsilon_n$ be the standard orthonormal basis of $\C^n$. The root system of 
$\sl{n}$ is given by $\{\epsilon_{i}-\epsilon_j\mid 1\leq i\neq j\leq n\}$. 
There is a natural embedding $\iota_1: \Aell(\sl{n}) \inj \AAell{n}$ by
\[
x(u)\mapsto \sum_{i=1}^n(u, \epsilon_i) x_i, \,\ 
y(u)\mapsto \sum_{i=1}^n(u, \epsilon_i) y_i, \,\ 
t_{\epsilon_{i}-\epsilon_j} \mapsto t_{ij}. 
\]
Indeed, we have
\begin{align*}
[y(u), x(v)]=
&[\sum_{1\leq i \leq n}(u, \epsilon_i) y_i, 
\sum_{1\leq i \leq n}(v, \epsilon_i) x_i,]
=\sum_{1\leq i\neq j\leq n} (u, \epsilon_i)(u, \epsilon_i-\epsilon_j)t_{ij}\\&
=\frac{1}{2}\sum_{1\leq i\neq j\leq n} (u, \epsilon_i-\epsilon_j)(u, \epsilon_i-\epsilon_j)t_{ij}
=\sum_{1\leq i< j\leq n} (u, \epsilon_i-\epsilon_j)(u, \epsilon_i-\epsilon_j)t_{ij}
\end{align*}
All other relations of $\Aell(\sl{n})$ are obviously preserved under $\iota_1$.

To summerise, we have two algebra homomorphisms from $\Aell(\sl{n})$ to $\End((\C[\h_k^{\reg}]\otimes \C[M_{k,n}])[0])$:
\begin{equation}\label{equ:action two}
\xymatrix@R=1em{
\Aell(\sl{n})\ar@{^{(}->}[r]^{\iota_1}\ar[d]_{\iota_2} &\AAell{n}\ar[r]^(0.3){a_1} &\End((\C[\h_k^{\reg}]\otimes \C[M_{k,n}])[0])\\
D_{\lambda, \beta}(\sl{n})\ar[rru]_(0.4){a_2}&
}
\end{equation}
The horizontal map $a_1 \iota_1$ is obtained by the composition of $\iota_1$ with the map $a_1$ in Proposition \ref{prop: map to two param}. 
It gives rise to the KZB connection for $\gla{k}$ in Theorem \ref{conn:gl_k}. 
The other map is the composition $a_2 \iota_2$ in Proposition \ref{prop:map to Hecke}, and it gives rise to the elliptic connection for $\sl{n}$. 

\redtext{There is a subtlety that the diagram \eqref{equ:action two} is not commutative. 
By Proposition \ref{prop:map to Hecke}, the action of $ t_{ij}$ using $a_1\iota_1$ is given by
\begin{align}
a_1\iota_1(t_{ij})= \sum_{1\leq a, b\leq k}(E_{ab}^{(k)})^{(i)}(E_{ba}^{(k)})^{(j)}=&\frac{1}{2}\left(
\sum_{1\leq a, b\leq k}\Big(
(E_{ij}^{(n)})^{(a)}(E_{ji}^{(n)})^{(b)}+(E_{ji}^{(n)})^{(b)}(E_{ij}^{(n)})^{(a)}
\Big)-\sum_{a=1}^k(E_{ii}^{(n)})^{(a)}-\sum_a(E_{jj}^{(n)})^{(a)}\right) \notag\\
=&\frac{1}{2}\Big(E_{ij}E_{ji}+E_{ji}E_{ij}-E_{ii}-E_{jj}\Big), \label{eq:tij-hor}
\end{align}
where the second equality follows from Proposition \ref{prop:dual pair} (see also \cite[(3.12)]{TL-Duke}). }

On the other hand, by Proposition \ref{prop: map to two param} and Theorem \ref{prop: action of D(sl)}, 
the image of $t_{ij}$ under $a_2\iota_2$ is given by 
\begin{align*}
a_2\iota_2(t_{ij})=&
\frac{\lambda}{2}(E_{ij}E_{ji}+E_{ji}E_{ij})+\frac{Z_n}{2n^2}+2(\beta - \frac{\lambda}{2})
\Big(\frac{E_{ii}+E_{jj}}{n}-\frac{2}{n^2}\sum_{e=1}^n E_{ee}\Big)\\
=& -\frac{1}{2}(E_{ij}E_{ji}+E_{ji}E_{ij})+\frac{2(n+1)}{2n^2}+2(\frac{n}{4} + \frac{1}{2})
\Big(\frac{E_{ii}+E_{jj}}{n}-\frac{2}{n^2}\Big)\\
=&-\frac{1}{2}(E_{ij}E_{ji}+E_{ji}E_{ij}-E_{ii}-E_{jj})+ \frac{E_{ii}+E_{jj}}{n}-\frac{1}{n^2}
\end{align*}
The difference of $a_1\iota_1(t_{ij})$ and $a_2\iota_2(t_{ij})$ is $\frac{E_{ii}+E_{jj}}{n}-\frac{1}{n^2}. $
For the generators $x(u), y(u)\in \Aell(\sl{n})$, we have 
\[
a_1\iota_1(x(u))= a_2\iota_2(x(u)), \,\ a_1\iota_1(y(u))= a_2\iota_2(y(u)). 
\] 
As a consequence, the difference of the elliptic Casimir connection in Theorem \ref{conn:two parameters} 
and the KZB connection in Theorem \ref{conn:gl_k} is given by the difference of $a_1\iota_1(t_{ij})$ and $a_2\iota_2(t_{ij})$. 
Therefore, 
\begin{align*}
&\mathcal{A}=\nabla_{\Ell, C}-\nabla_{\mathcal{H}(\gla{k}, \h_k)}
=\sum_{1\leq i< j \leq n}\frac{\theta'(z_i-z_j|\tau)}{\theta(z_i-z_j|\tau)}\Big(\frac{E_{ii}+E_{jj}}{n}-\frac{1}{n^2} \Big)d z_{ij}.
\end{align*}
This completes the proof. 
\end{proof}
\subsection{}
In this section, we explain how to get the formulas in Theorem \ref{prop: action of D(sl)} using Proposition \ref{prop:dual pair}.  

The enveloping algebra $U(\sl{n})$ is a subalgebra of $D_{\lambda, \beta}(\sl{n})$. We require the action of $U(\sl{n})$ to be the one induced from the natural $\GL_n$ action on $ \C[M_{k,n}]$. That is, $E_{ij}\mapsto 1\otimes \sum_{a=1}^k (E_{ij}^{(n)})^{(a)}$. 

We now use diagram \eqref{equ:action two} to deduce the action of $K(E_{ij})$ and $Q(E_{ij})$, for $1\leq i\neq j\leq n$. 
The action of $x_i\in \Aell(\sl{n})$ on $\C[\h_{k}^{\reg}]\otimes \C[M_{k,n}]$ is given by $\sum_{a=1}^k x_a\otimes (E_{aa}^{(k)})^{(i)}$. By Proposition \ref{prop:dual pair}, on $\C[\h_{k}^{\reg}]\otimes \C[M_{k,n}]$, we have the identity 
\[\sum_{a=1}^k x_a\otimes (E_{aa}^{(k)})^{(i)}=\sum_{a=1}^k x_a\otimes (E_{ii}^{(n)})^{(a)}.\]
Therefore, the action of $K(E_{ij})\in D_{\lambda,  \beta}(\sl{n})$, $1\leq i\neq j\leq n$, is given by:
\[
K(E_{ij}) \mapsto \Big[\sum_{a=1}^k x_a\otimes (E_{ii}^{(n)})^{(a)}, E_{ij}\Big]= \sum_{a=1}^k x_a\otimes (E_{ij}^{(n)})^{(a)}.
\]
It is obvious that the assignment preserves the relation $[K(X), Y]=K[X, Y]$, for any $X, Y\in \sl{n}$. Thus, it gives an action of $ \sl{n}[v]$ on $\C[\h_{k}^{\reg}]\otimes \C[M_{k,n}]$. 

Similarly, using Proposition \ref{prop:dual pair}, we rewrite the action of $y_i\in \Aell(\sl{n})$ on $\C[\h_{k}^{\reg}]\otimes \C[M_{k,n}]$ by
\[
 -\sum_{a=1}^k \partial_a\otimes E_{aa}^{(i)}+
\sum_{j=1}^n\sum_{1\leq a\neq b\leq k} \frac{1}{x_b-x_a}\otimes E_{ab}^{(i)}E_{ba}^{(j)}
= -\sum_{a=1}^k \partial_a\otimes (E_{ii}^{(n)})^{(a)}+
\sum_{1\leq a\neq b \leq k} \frac{1}{x_b-x_a}\otimes (\sum_{j=1}^n(E_{ij}^{(n)})^{(a)}(E_{ji}^{(n)})^{(b)}+(E_{ii}^{(n)})^{(a)}).\]
Let $i\neq l$,  in order to get an action of $Q(E_{il})$, we 
apply  $[ \,\ , E_{il}]$ to the above identity. We have
\begin{align*}
&[-\sum_{a=1}^k \partial_a\otimes (E_{ii}^{(n)})^{(a)}, E_{il}]+
[\sum_{1\leq a\neq b \leq k} \frac{1}{x_b-x_a}\otimes (\sum_{j=1}^n (E_{ij}^{(n)})^{(a)}(E_{ji}^{(n)})^{(b)}+(E_{ii}^{(n)})^{(a)}), E_{il}]\\
=&-\sum_{a=1}^k \partial_a\otimes (E_{il}^{(n)})^{(a)}+\sum_{1\leq a\neq b \leq k} \frac{1}{x_b-x_a}\otimes (\sum_{j=1}^n (E_{ij}^{(n)})^{(a)}(E_{jl}^{(n)})^{(b)}+(E_{il}^{(n)})^{(a)})
\end{align*}
Therefore, the action of $Q(E_{il})$, $1\leq i\neq l\leq n$, is given by:
\begin{equation}\label{eq:about Q}
Q(E_{il})\mapsto -\sum_{a=1}^k \partial_a\otimes (E_{il}^{(n)})^{(a)}+\sum_{1\leq a\neq b \leq k} \frac{1}{x_b-x_a}\otimes (\sum_{j=1}^n (E_{ij}^{(n)})^{(a)}(E_{jl}^{(n)})^{(b)}+(E_{il}^{(n)})^{(a)}).
\end{equation}
It is straightforward to check that the assignment preserves the relation $[Q(X), Y]=Q[X, Y]$, for any $X, Y\in \sl{n}$. 
\begin{prop}
There is an action of $ \sl{n}[u]$ on $\C[\h_{k}^{\reg}]\otimes \C[M_{k,n}]$, such that, 
$E_{ij} \mapsto 1\otimes \sum_{a=1}^k (E_{ij}^{(n)})^{(a)}$, and the action of $Q(E_{ij})$ is given by \eqref{eq:about Q}, for any $1\leq i\neq j\leq n$. 
\end{prop}
\begin{proof}
We have known that $[Q(X), Y]=Q[X, Y]$, for any $X, Y\in \sl{n}$.
It is well-known that $Y_{\hbar}(\sl{n})$ is a flat deformation of $\sl{n}[u]$. 
By Theorem \ref{thm:Yangpres} (set $\hbar=0$), we only need to verify the following relation $[Q(H_i), Q(H_j)]=0$, for any
$i, j\in I$. 

By \eqref{eq:check yy}, we have the following relation in $\Diff(\h_k)^{\reg}\otimes U(\gla{k})^{\otimes n}$, for $1\leq i, j\leq n$, 
\begin{align*}
[y_i, y_j]=&[Q(E_{ii}), Q(E_{jj})]  \notag\\
=&\sum_{1\leq s\neq t \leq k}\frac{1}{(x_t-x_s)^2}  \otimes \sum_{e=1}^n \left( -E_{st}^{(i)}E_{ts}^{(j)}E_{tt}^{(e)}
+E_{st}^{(i)}E_{ts}^{(j)}E_{ss}^{(e)}\right)
\end{align*}
On $\C[M_{k, n}]$, the operator $\sum_{e=1}^n (E_{tt}^{(k)})^{(e)}$ acts by 
$\sum_{e=1}^n (E_{tt}^{(k)})^{(e)}=\sum_{e=1}^n (E_{ee}^{(n)})^{(t)}$. 
Therefore, 
\[
\sum_{e=1}^n \left( -E_{st}^{(i)}E_{ts}^{(j)}E_{tt}^{(e)}
+E_{st}^{(i)}E_{ts}^{(j)}E_{ss}^{(e)}\right)
\] acts by zero on $\C[M_{k, n}]$. 

This completes the proof. 
\end{proof}

\subsection{Proof of Theorem \ref{prop: action of D(sl)}}
In this subsection, we prove Theorem \ref{prop: action of D(sl)}. We show, when the two parameters 
are $\lambda=-1, \,\  \beta=\frac{n}{4}$, the formula in Theorem \ref{prop: action of D(sl)} gives a well-defined action of $D_{-1, \frac{n}{4}}(\sl{n})$ on $\C[\h_{k}^{\reg}]\otimes \C[M_{k,n}]$.

We have shown that the two subalgebras $\sl{n}[u]$ and $\sl{n}[v]$ act on $\C[\h_{k}^{\reg}]\otimes \C[M_{k,n}]$. It is a straightforward computation to show $[X, P(X')]=P([X, X'])$, for any $X, X'\in \sl{n}$. It remains to check the main defining relation \eqref{main relation} in Lemma \ref{lem:rewritten the relation} on $\C[\h_{k}^{\reg}]\otimes \C[M_{k,n}]$. 

We now compute the action of $[K(E_{ij}), Q(E_{st})]$ on the vector space $\C[\h_{k}^{\reg}]\otimes  \C[M_{k,n}]$.
Using the formulas in Theorem \ref{prop: action of D(sl)}, we have:
\begin{align}
[K(E_{ij}), Q(E_{st})]
=&[\sum_a \partial_a\otimes (E_{st}^{(n)})^{(a)}, \sum_a x_a\otimes (E_{ij}^{(n)})^{(a)}] \label{KQ term 1}\\
+&[\sum_a x_a\otimes (E_{ij}^{(n)})^{(a)},  \sum_{1\leq a\neq b \leq k} \frac{1}{x_b-x_a}\otimes (\sum_e(E_{se}^{(n)})^{(a)}(E_{et}^{(n)})^{(b)}+(E_{st}^{(n)})^{(a)})] \label{KQ term 2}
\end{align}
Set $Y_a:=\frac{\partial_a x_a+x_a \partial_a}{2}$. Using the substitution $\partial_a x_a=Y_a+\frac{1}{2}$, and $x_a \partial_a =Y_a-\frac{1}{2}$, we compute the equation \eqref{KQ term 1} as follows. 
\begin{align*}
\eqref{KQ term 1}:=[\sum_a \partial_a\otimes (E_{st}^{(n)})^{(a)}, \sum_a x_a\otimes (E_{ij}^{(n)})^{(a)}]
=&\sum_a \partial_a x_a \otimes (E_{st}^{(n)})^{(a)}(E_{ij}^{(n)})^{(a)}-
\sum_a x_a \partial_a\otimes (E_{ij}^{(n)})^{(a)}(E_{st}^{(n)})^{(a)}\\
=&\sum_{a}Y_{a}[E_{st}, E_{ij}]^{(a)}+\frac{1}{2}\sum_{a}(E_{st}^{(a)}E_{ij}^{(a)}+E_{ij}^{(a)}E_{st}^{(a)}).
\end{align*} 
We then compute the equation \eqref{KQ term 2}. We have
\begin{align}
\eqref{KQ term 2}
=&\sum_{c=1}^k \sum_{1\leq a\neq b \leq k} \frac{x_c}{x_b-x_a} \otimes \Big[(E_{ij}^{(n)})^{(c)}, \sum_e(E_{se}^{(n)})^{(a)}(E_{et}^{(n)})^{(b)}+(E_{st}^{(n)})^{(a)}\Big] \notag\\
=&\sum_{1\leq a\neq b \leq k} \frac{x_a}{x_b-x_a} \otimes  \sum_e([E_{ij},E_{se}]^{(a)}(E_{et}^{(n)})^{(b)}
+\sum_{1\leq a\neq b \leq k} \frac{x_b}{x_b-x_a} \otimes \sum_e(E_{se}^{(n)})^{(a)}([E_{ij},E_{et}]^{(b)}\label{KQ term 2.5}\\
&+\sum_{1\leq a\neq b \leq k} \frac{x_a}{x_b-x_a} \otimes [E_{ij}, E_{st}]^{(a)}.\notag
\end{align}
We simplify  \eqref{KQ term 2.5} using the following identities. \begin{align*}
&\sum_{e=1}^n [E_{ij},E_{se}]^{(a)}E_{et}^{(b)}
=\sum_{e=1}^n(\delta_{js} E_{ie}-\delta_{ei}E_{sj})^{(a)}E_{et}^{(b)}
= \delta_{js}\sum_{e=1}^n E_{ie}^{(a)} E_{et}^{(b)}
-E_{sj}^{(a)}E_{it}^{(b)}.\\
&\sum_{e=1}^n E_{se}^{(a)}([E_{ij},E_{et}]^{(b)}
=\sum_{e=1}^n E_{se}^{(a)}
(\delta_{je} E_{it}-\delta_{it} E_{ej})^{(b)}
= E_{sj}^{(a)} E_{it}^{(b)}
-\delta_{it} \sum_{e=1}^n E_{se}^{(a)} E_{ej}^{(b)}.
\end{align*} 
We have 
\begin{align*}
\eqref{KQ term 2.5}:=&\sum_{1\leq a\neq b \leq k} \frac{x_a}{x_b-x_a} \otimes  \sum_e([E_{ij},E_{se}]^{(a)}(E_{et}^{(n)})^{(b)}
+\sum_{1\leq a\neq b \leq k} \frac{x_b}{x_b-x_a} \otimes \sum_e(E_{se}^{(n)})^{(a)}([E_{ij},E_{et}]^{(b)}\\
=&
\sum_{1\leq a\neq b \leq k} \frac{x_a}{x_b-x_a} \otimes  (\delta_{js}\sum_{e=1}^n E_{ie}^{(a)} E_{et}^{(b)}
-E_{sj}^{(a)}E_{it}^{(b)})
+\sum_{1\leq a\neq b \leq k} \frac{x_b}{x_b-x_a} \otimes (E_{sj}^{(a)} E_{it}^{(b)}
-\delta_{it} \sum_{e=1}^n E_{se}^{(a)} E_{ej}^{(b)})\\
=&\sum_{1\leq a\neq b \leq k}E_{sj}^{(a)} E_{it}^{(b)}+
\sum_{1\leq a\neq b \leq k}\frac{1}{2}(\frac{x_a+x_b}{x_b-x_a}-1) \otimes  \delta_{js}\sum_{e=1}^n E_{ie}^{(a)} E_{et}^{(b)}
-\sum_{1\leq a\neq b \leq k} \frac{1}{2}(\frac{x_a+x_b}{x_b-x_a}+1) \otimes 
\delta_{it} \sum_{e=1}^n E_{se}^{(a)} E_{ej}^{(b)}\\
=&\sum_{1\leq a\neq b \leq k}E_{sj}^{(a)} E_{it}^{(b)}+
\sum_{1\leq a\neq b \leq k}\frac{1}{2}(\frac{x_a+x_b}{x_b-x_a}) \otimes  \Big(\delta_{js}\sum_{e=1}^n E_{ie}^{(a)} E_{et}^{(b)}
-\delta_{it} \sum_{e=1}^n E_{se}^{(a)} E_{ej}^{(b)}\Big)\\&
-\frac{1}{2} \sum_{1\leq a\neq b \leq k} \Big(\delta_{js}\sum_{e=1}^n E_{ie}^{(a)} E_{et}^{(b)}+\delta_{it} \sum_{e=1}^n E_{se}^{(a)} E_{ej}^{(b)}\Big)\\
=&E_{sj} E_{it}-\sum_{1\leq a\leq k}E_{sj}^{(a)} E_{it}^{(a)}
+
\sum_{1\leq a\neq b \leq k}\frac{1}{2}(\frac{x_a+x_b}{x_b-x_a}) \otimes  \Big(\delta_{js}\sum_{e=1}^n E_{ie}^{(a)} E_{et}^{(b)}
-\delta_{it} \sum_{e=1}^n E_{se}^{(a)} E_{ej}^{(b)}\Big)\\
&
-\frac{1}{2} \Big(\delta_{js}\sum_{e=1}^n E_{ie} E_{et}+\delta_{it} \sum_{e=1}^n E_{se} E_{ej}\Big)
+\frac{1}{2} \sum_{1\leq a\leq k} \Big(\delta_{js}\sum_{e=1}^n E_{ie}^{(a)} E_{et}^{(a)}+\delta_{it} \sum_{e=1}^n E_{se}^{(a)} E_{ej}^{(a)}\Big)
\end{align*}
To summarize, based on the computations of equations \eqref{KQ term 1}, \eqref{KQ term 2} and \eqref{KQ term 2.5}, we have
\begin{align}
&[K(E_{ij}), Q(E_{st})] \notag\\
=&\sum_{a}Y_{a}[E_{st}, E_{ij}]^{(a)}+\frac{1}{2}\sum_{a}(E_{st}^{(a)}E_{ij}^{(a)}+E_{ij}^{(a)}E_{st}^{(a)})+\sum_{1\leq a\neq b \leq k} \frac{x_a}{x_b-x_a} \otimes [E_{ij}, E_{st}]^{(a)}\notag
\\&+E_{sj} E_{it}-\sum_{1\leq a\leq k}E_{sj}^{(a)} E_{it}^{(a)}
+
\sum_{1\leq a\neq b \leq k}\frac{1}{2}(\frac{x_a+x_b}{x_b-x_a}) \otimes  \Big(\delta_{js}\sum_{e=1}^n E_{ie}^{(a)} E_{et}^{(b)}
-\delta_{it} \sum_{e=1}^n E_{se}^{(a)} E_{ej}^{(b)}\Big) \notag\\
&
-\frac{1}{2} \Big(\delta_{js}\sum_{e=1}^n E_{ie} E_{et}+\delta_{it} \sum_{e=1}^n E_{se} E_{ej}\Big)
+\frac{1}{2} \sum_{1\leq a\leq k} \Big(\delta_{js}\sum_{e=1}^n E_{ie}^{(a)} E_{et}^{(a)}+\delta_{it} \sum_{e=1}^n E_{se}^{(a)} E_{ej}^{(a)}\Big) \notag\\
=&
P([E_{ij},  E_{st}])
+\frac{1}{2}\sum_{a}(E_{st}^{(a)}E_{ij}^{(a)}+E_{ij}^{(a)}E_{st}^{(a)})+E_{sj} E_{it}-\sum_{1\leq a\leq k}E_{sj}^{(a)} E_{it}^{(a)} \notag\\
&
-\frac{1}{2} \Big(\delta_{js}\sum_{e=1}^n E_{ie} E_{et}+\delta_{it} \sum_{e=1}^n E_{se} E_{ej}\Big)
+\frac{1}{2} \sum_{1\leq a\leq k} \Big(\delta_{js}\sum_{e=1}^n E_{ie}^{(a)} E_{et}^{(a)}+\delta_{it} \sum_{e=1}^n E_{se}^{(a)} E_{ej}^{(a)}\Big) \notag\\
=&
P([E_{ij},  E_{st}])+E_{sj} E_{it}-\frac{1}{2} \Big(\delta_{js}\sum_{e=1}^n E_{ie} E_{et}+\delta_{it} \sum_{e=1}^n E_{se} E_{ej}\Big) 
\notag\\&+\frac{1}{2}\sum_{1\leq a\leq k}(E_{st}^{(a)}E_{ij}^{(a)}+E_{ij}^{(a)}E_{st}^{(a)})
-\sum_{1\leq a\leq k}E_{sj}^{(a)} E_{it}^{(a)}
+\frac{1}{2} \sum_{1\leq a\leq k} \Big(\delta_{js}\sum_{e=1}^n E_{ie}^{(a)} E_{et}^{(a)}+\delta_{it} \sum_{e=1}^n E_{se}^{(a)} E_{ej}^{(a)}\Big) \label{eqn:[KQ-P]}
\end{align} 
We compute the action of the equation \eqref{eqn:[KQ-P]} on the vector space $\C[M_{k, n}]=(\C[M_{1, n}])^{\otimes k}$. 
We use the following identities. 
\begin{align*}
\frac{1}{2}\sum_{1\leq a\leq k}(E_{st}^{(a)}E_{ij}^{(a)}+&E_{ij}^{(a)}E_{st}^{(a)})
-\sum_{1\leq a\leq k}E_{sj}^{(a)} E_{it}^{(a)}
\mapsto \frac{1}{2}(x_{as}\partial_{at}x_{ai}\partial_{aj}+x_{ai}\partial_{aj}x_{as}\partial_{at})- x_{as}\partial_{aj}x_{ai}\partial_{at}\\
=&\frac{1}{2}\Big(x_{as}\partial_{aj} x_{ai}\partial_{at} +\delta_{it} x_{as}\partial_{aj}
+x_{as}\partial_{aj}x_{ai}\partial_{at}
+\delta_{sj} x_{ai}\partial_{at}\Big)- x_{as}\partial_{aj}x_{ai}\partial_{at}\\
=&\frac{1}{2}\Big(\delta_{it} x_{as}\partial_{aj}
+\delta_{sj} x_{ai}\partial_{at}\Big)=\frac{1}{2}(\delta_{it} E_{sj}^{(a)} +\delta_{sj}E_{it}^{(a)}). 
\end{align*}
\begin{align*}
\sum_{e=1}^n E_{ie}^{(a)} E_{et}^{(a)}
\mapsto &\sum_{e=1}^n x_{ai}\partial_{ae}x_{ae}\partial_{at}
=\sum_{e=1}^n x_{ai}\partial_{ae}\partial_{at}x_{ae}- x_{ai}\partial_{at}\\&
=\sum_{e=1}^n x_{ai}\partial_{at}x_{ae}\partial_{ae}
+\sum_{e=1}^n x_{ai}\partial_{at}- x_{ai}\partial_{at}
=\sum_{e=1}^n E_{it}^{(a)}E_{ee}^{(a)}
+\sum_{e=1}^n E_{it}^{(a)}- E_{it}^{(a)}=nE_{it}^{(a)}. 
\end{align*}
Similarly, $\sum_{e=1}^n E_{se}^{(a)} E_{ej}^{(a)}\mapsto nE_{sj}^{(a)}. $
Therefore, on $(\C[M_{1, n}])^{\otimes k}$,  \eqref{eqn:[KQ-P]}  acts by
\begin{align*}
\eqref{eqn:[KQ-P]} 
\Omit{:=
&\frac{1}{2}\sum_{1\leq a\leq k}(E_{st}^{(a)}E_{ij}^{(a)}+E_{ij}^{(a)}E_{st}^{(a)})
-\sum_{1\leq a\leq k}E_{sj}^{(a)} E_{it}^{(a)}
+\frac{1}{2} \sum_{1\leq a\leq k} \Big(\delta_{js}\sum_{e=1}^n E_{ie}^{(a)} E_{et}^{(a)}+\delta_{it} \sum_{e=1}^n E_{se}^{(a)} E_{ej}^{(a)}\Big) \\}
= \frac{1}{2}\sum_{1\leq a\leq k}(\delta_{ti} E_{sj}^{(a)}+\delta_{js}E_{it}^{(a)})
+\frac{n}{2} \sum_{1\leq a\leq k} 
\Big(\delta_{js} E_{it}^{(a)}+\delta_{it}  E_{sj}^{(a)}\Big) 
= \frac{1}{2}(1+n) (\delta_{ti} E_{sj}+\delta_{js}E_{it}). 
\end{align*}

Therefore, on $\End(\C[\h_{k}^{\reg}]\otimes \C[M_{k,n}])$ we have the following identity
\[
[K(E_{ij}), Q(E_{st})]=P([E_{ij},  E_{st}])+E_{sj} E_{it}-\frac{1}{2} \Big(\delta_{js}\sum_{e=1}^n E_{ie} E_{et}+\delta_{it} \sum_{e=1}^n E_{se} E_{ej}\Big)
+\frac{1}{2}(1+n) (\delta_{ti} E_{sj}+\delta_{js}E_{it}).
\]
Comparing the above equality with Lemma \ref{lem:rewritten the relation} equation \eqref{main relation}, we have 
\[
\lambda=-1, \,\ \beta=\frac{1}{4}n. 
\]

In the rest of this section, we compute the action of the central element $Z_n\in D_{-1, \frac{n}{4}}(\sl{n})$ on $\C[\h_{k}^{\reg}]\otimes \C[M_{k,n}]$. 
\begin{lemma}\label{prop:central action}
The central element $Z_n\in D_{-1, \frac{n}{4}}(\sl{n})$ acts by the scalar $2(n+1)$ on $\C[\h_{k}^{\reg}]\otimes \C[M_{k,n}]$. 
\end{lemma}
\begin{proof}
Recall that, by definition, we have
\begin{align*}
Z_n=\sum_{a=1}^{n}Z_{a, a+1}=\sum_{a=1}^{n}\left([K(H_{a}), Q(H_{a})]-\frac{\lambda}{4}\sum_{1\leq i \neq j\leq n}S([H_{a}, E_{ij}], [E_{ji}, H_{a}]). \right)
\end{align*}
We first compute the action of 
\begin{equation}\label{eq:Z0}
Z_{12}=[K(H_{12}), Q(H_{12})]-\frac{\lambda}{4}\sum_{1\leq i \neq j\leq n}S([H_{12}, E_{ij}], [E_{ji}, H_{12}])
\end{equation} on $\C[\h_{k}^{\reg}]\otimes \C[M_{k,n}]$ using the formulas in Theorem  \ref{prop: action of D(sl)}. 
We have
\begin{align*}
 &[K(H_{12}), Q(H_{12})]\\
\mapsto 
&\Big[\sum_{a=1}^k  x_a\otimes H_{12}^{(a)}, 
-\sum_{a=1}^k \partial_a\otimes H_{12}^{(a)}+
\sum_{1\leq a\neq b \leq k} \frac{1}{x_b-x_a}
\otimes \big(\sum_{e=1}^nE_{1e}^{(a)}E_{e1}^{(b)}-\sum_{e=1}^nE_{2e}^{(a)}E_{e2}^{(b)} +H_{12}^{(a)}\big)\Big]\\
=&\sum_{a=1}^k H_{12}^{(a)}H_{12}^{(a)}
+\sum_{1\leq a\neq b\leq k }\frac{x_a}{ x_b-x_a}[H_{12}^{(a)}, \sum_{e=1}^nE_{1e}^{(a)}E_{e1}^{(b)}-\sum_{e=1}^nE_{2e}^{(a)}E_{e2}^{(b)}]\\
&\phantom{123456789009}+\sum_{1\leq a\neq b\leq k }\frac{x_b}{ x_b-x_a}[H_{12}^{(b)}, \sum_{e=1}^nE_{1e}^{(a)}E_{e1}^{(b)}-\sum_{e=1}^nE_{2e}^{(a)}E_{e2}^{(b)}]\\
=&\sum_{a=1}^k H_{12}^{(a)}H_{12}^{(a)}
-\sum_{1\leq a\neq b\leq k }\sum_{e=1}^{n} \Big( 
(\epsilon_{12}, \epsilon_{1e}) E_{1e}^{(a)}E_{e1}^{(b)}
-(\epsilon_{12}, \epsilon_{2e}) E_{2e}^{(a)}E_{e2}^{(b)}\Big)
\Big)\\
=&\sum_{a=1}^k H_{12}^{(a)}H_{12}^{(a)}
-\sum_{1\leq a\neq b\leq k}\Big( 
\sum_{e=1}^nE_{1e}^{(a)}E_{e1}^{(b)}+\sum_{e=1}^nE_{2e}^{(a)}E_{e2}^{(b)}
-E_{11}^{(a)}E_{11}^{(b)}-E_{22}^{(a)}E_{22}^{(b)}+E_{12}^{(a)}E_{21}^{(b)}
+E_{21}^{(a)}E_{12}^{(b)}
\Big)
\end{align*}
Therefore, $[K(H_{12}), Q(H_{12})]$ acts on $\C[\h_{k}^{\reg}]\otimes \C[M_{k,n}]$ by 
\begin{align}
&-\left(\sum_{e=1}^nE_{1e}E_{e1}+\sum_{e=1}^nE_{2e}E_{e2}
-E_{11}E_{11}-E_{22}E_{22}+E_{12}E_{21}
+E_{21}E_{12}\right) \notag\\&
+\sum_{1\leq a \leq k}\Big( 
\sum_{e=1}^nE_{1e}^{(a)}E_{e1}^{(a)}+\sum_{e=1}^nE_{2e}^{(a)}E_{e2}^{(a)}
+E_{12}^{(a)}E_{21}^{(a)}
+E_{21}^{(a)}E_{12}^{(a)}
-2E_{11}^{(a)}E_{22}^{(a)}
\Big) \notag\\
=&(n+1) (E_{11}+E_{22})
-\left(\sum_{e=1}^nE_{1e}E_{e1}+\sum_{e=1}^nE_{2e}E_{e2}
-E_{11}E_{11}-E_{22}E_{22}+E_{12}E_{21}
+E_{21}E_{12}\right), \label{eq:Z1}
\end{align}
where the last equality follows from the following identities
\[
\sum_{e=1}^nE_{1e}^{(a)}E_{e1}^{(a)}\mapsto 
\sum_{e=1}^n x_{a1} \partial_{ae}x_{ae} \partial_{a1}
=\sum_{e=1}^n x_{a1} \partial_{a1}x_{ae}  \partial_{ae}
+\sum_{e=1}^n x_{a1} \partial_{a1}
- x_{a1} \partial_{a1}
=n E_{11}^{(a)}. 
\]
\[
E_{12}^{(a)}E_{21}^{(a)}
+E_{21}^{(a)}E_{12}^{(a)}
-2E_{11}^{(a)}E_{22}^{(a)}
\mapsto  x_{a1} \partial_{a2} x_{a2} \partial_{a1}+
 x_{a2} \partial_{a1} x_{a1} \partial_{a2}-2 x_{a1} \partial_{a1} x_{a2} \partial_{a2}
 =x_{a1} \partial_{a1}+x_{a2} \partial_{a2}=E_{11}^{(a)}+E_{22}^{(a)}. 
\]
On the other hand, 
we have
\begin{align}
&\frac{\lambda}{4}\sum_{1\leq i \neq j\leq n}S([H_{12}, E_{ij}], [E_{ji}, H_{12}]) \notag\\
=&\frac{\lambda}{2}\Big(\sum_{\{j\mid j\neq 1 \}} S(E_{1j}, E_{j1})+\sum_{\{j\mid j\neq 2 \}} S(E_{2j}, E_{j2})+2S(E_{12}, E_{21})\Big) \notag\\
=&\lambda\Big( 
\sum_{\{j\mid j\neq 1 \}} E_{1j}E_{j1}+\sum_{\{j\mid j\neq 2 \}} E_{2j}E_{j2}+S(E_{12}, E_{21})\Big)
+\frac{\lambda}{2} \Big(\sum_{\{j\mid j\neq 1 \}} H_{j1}+ \sum_{\{j\mid j\neq 2 \}} H_{j2}\Big) \notag\\
=& 
\lambda\Big( 
\sum_{e=1}^n E_{1e}E_{e1}+\sum_{e=1}^n E_{2e}E_{e2}+S(E_{12}, E_{21})
-E_{11}E_{11}-E_{22}E_{22}
\Big)
+\frac{\lambda}{2} \Big(
2 \sum_{e=1}^n E_{ee}-nE_{11}-nE_{22}\Big)\label{eq:Z2}
\end{align}
By \eqref{eq:Z0}, \eqref{eq:Z1} and \eqref{eq:Z2}, when $\lambda=-1$, 
the element $Z_{12}$ acts on $\C[\h_{k}^{\reg}]\otimes \C[M_{k,n}]$ by
\[
Z_{12}\mapsto (\frac{n}{2}+1) (E_{11}+E_{22})+\sum_{e=1}^n E_{ee}. 
\]
By symmetry, we know $Z_{n}$ acts by 
\begin{align*}
Z_n=\sum_{a=1}^{n}Z_{a, a+1}
\mapsto &\sum_{a=1}^{n}\Big(  (\frac{n}{2}+1) (E_{aa}+E_{a+1, a+1})+\sum_{e=1}^n E_{ee}\Big)
=2(n+1) \sum_{e=1}^n E_{ee}. 
\end{align*}
\end{proof}
This completes the proof of Theorem \ref{prop: action of D(sl)}.

\section{Flat connection on the elliptic moduli space $\M_{1, n}$}
\redtext{In this section, we collect some results from \cite[Sections 7, 8, 9]{TLY1} about the extension of the 
KZB connection associated to $\Phi$ to the $\tau$ direction, which we will need in Section \ref{sec:sl2} and Section \ref{sec:extension}. }
\subsection{Derivation of the Lie algebra $\Aell$}
\label{sec:derivation}
Let $\dd$ be the Lie algebra defined in \cite{CEE, TLY1} with generators $\Delta_0, d, X$, and $\delta_{2m}( m\geq 1)$, and relations
\begin{align*}
&[d, X]=2X,\,\ [d, \Delta_0]=-2\Delta_0, \,\ [X, \Delta_0]=d,\\
&
[\delta_{2m}, X]=0,\,\ [d, \delta_{2m}]=2m\delta_{2m},\,\ (\ad\Delta_0)^{2m+1}(\delta_{2m})=0.
\end{align*}
There is a Lie algebra morphism $ \dd \to \Der(\Aell)$ in \cite[Proposition 7.1]{TLY1}, denoted by $\xi \mapsto \tilde{\xi}$. The image $\tilde{\xi}$ of $\xi$ acts on $\Aell$ by the following formulas.
\begin{align*}
&\tilde{d}(x(u))=x(u), \phantom{1234} \tilde{d}(y(u))=-y(u), \phantom{1234}   \tilde{d}(t_{\alpha})=0, \\
&\tilde{X}(x(u))=0, \phantom{123456}   \tilde{X}(y(u))=x(u), \phantom{12345}  \tilde{X}(t_{\alpha})=0, \\
&   \widetilde{\Delta}_0(x(u))=y(u),  \phantom{123}  \widetilde{\Delta}_0(y(u))=0, \phantom{123456} 
\widetilde{\Delta}_0(t_{\alpha})=0, \\
& \tilde{\delta}_{2m}(x(u))=0, \phantom{12345} 
\tilde{\delta}_{2m}(t_{\alpha})=[t_\alpha, (\ad\frac{x(\alpha^\vee)}{2})^{2m}(t_\alpha)],\\
\text{and} \phantom{1234} \,\ &\tilde{\delta}_{2m}(y(u))=
\frac{1}{2}\sum_{\alpha\in \Phi^+}\alpha(u)\sum_{p+q=2m-1}[(\ad\frac{x(\alpha^\vee)}{2})^{p}
   (t_\alpha), (\ad-\frac{x(\alpha^\vee)}{2})^{q}(t_\alpha)].
\end{align*}
\begin{prop}\cite[Proposition 7.1]{TLY1}
\label{prop:d is deriv}
The above map $\dd \to \Der(\Aell)$ is a Lie algebra homomorphism.
\end{prop}

\subsection{A principal bundle}
Let $e, f, h$ be the standard basis of $\sl{2}$. 
There is a Lie algebra morphism
$\dd \to \sl{2}$ defined by $\delta_{2m} \mapsto 0, d \mapsto h, X \mapsto e, \Delta_0\mapsto f.$
Let $\dd_+ \subset \dd$ be the kernel of this homomorphism. 
Since the morphism has a section, which is given by $e\mapsto X, f\mapsto \Delta_0$ and $h \mapsto d$, we have
a semidirect decomposition $\dd=\dd_+ \rtimes \sl{2}.$ As a consequence, we have the decomposition
\[ \Aell \rtimes \dd=(\Aell \rtimes \dd_+)\rtimes \sl{2}.\]

The Lie algebra $\Aell \rtimes \dd_+$ is positively graded.
The $\Z^2-$grading of $\dd$ and $\Aell$ is defined by 
\begin{align*}
&\deg(\Delta_0)=(-1, 1), \,\ \deg(d)=(0, 0), \,\  \deg(X)=(1, -1), \,\ \deg(\delta_{2m})=(2m+1, 1)\\
\text{and}\,\ &\deg(x(u))=(1, 0)\,\ \deg(y(u))=(0, 1), \,\ \deg(t_\alpha)=(1, 1).
\end{align*}

We form the following semidirect products \[G_n:=\exp(\widehat{\Aell\rtimes \dd_+})\rtimes \SL_2(\C), \]
where $\widehat{\Aell\rtimes \dd_+}$ is the completion of $\Aell\rtimes \dd_+$ with respect to the grading above. 

Let $Q^\vee\subset \h$ be the coroot lattice of $\h$. The semidirect product $(Q^\vee\oplus  Q^{\vee})\rtimes \SL_2(\Z)$ acts on $\h \times \mathfrak{H}$ as follows. 
For $(\bold{n},  \bold{m}) \in (Q^{\vee}\oplus  Q^{\vee})$ and $(z, \tau)\in \h \times \mathfrak{H}$, the action is given by translation:
$(\bold{n}, \bold{m})*(z, \tau):=
(z+\bold{n}+\tau\bold{m}, \tau)$. For $\left(\begin{smallmatrix}
a & b \\
c & d
\end{smallmatrix}\right)\in \SL_2(\Z)$, the action is given by
$\left(\begin{smallmatrix}
a & b \\
c & d
\end{smallmatrix}\right)*(z, \tau):=(\frac{z}{c\tau+d}, \frac{a\tau+b}{c\tau+d})$. 
 Let $\alpha(-): \h\to \C$ be the map induced by the root $\alpha\in \Phi$. 
 We define $\widetilde{H}_{\alpha, \tau}\subset \h\times \mathfrak{H}$ to be
\[
\widetilde{H}_{\alpha, \tau}=\{(z, \tau)\in \h\times\mathfrak{H}\mid \alpha(z) \in \Lambda_\tau=\Z+\tau \Z\}.
\]

We define the elliptic moduli space $\M_{1, n}$ to be the quotient of 
$\h \times \mathfrak{H}\setminus\bigcup_{\alpha\in \Phi^+, \tau\in \mathfrak{H}} \widetilde{H}_{\alpha, \tau}$ by the action of $(Q^{\vee} \oplus Q^{\vee}) \rtimes \SL_2(\Z)$.  
Let $\pi: \h \times \mathfrak{H}\setminus\bigcup_{\alpha\in \Phi^+, \tau\in \mathfrak{H}} \widetilde{H}_{\alpha, \tau} \to \M_{1, n}$ be the natural projection.
We define a principal $G_n-$bundle $P_n$ on the elliptic moduli space $\M_{1, n}$. 

For $u\in \C^*$, $u^d:=\left(\begin{smallmatrix}
u & 0 \\
0 & u^{-1}
\end{smallmatrix}\right)\in \SL_2(\C)\subset G_n$ and for $v\in \C, e^{vX}:=\left(\begin{smallmatrix}
1 & v \\
0 & 1
\end{smallmatrix}\right)\in \SL_2(\C)\subset G_n$.
\begin{prop}\cite[Proposition 8.3]{TLY1}
\label{prop: G_n bundle}
There exists a unique principal $G_n-$bundle $P_n$ over $\M_{1, n}$, such that a section of $U \subset \M_{1, n}$ is a function
$f: \pi^{-1}(U)\to G_n$, with the properties that
\begin{align*}
&f(z+\alpha^\vee_i|\tau)=f(z|\tau), \,\ f(z+\tau\alpha^\vee_i|\tau)=e^{-2\pi ix_{\alpha_i^\vee}}f(z|\tau), \\
& f(z|\tau+1)=f(z|\tau), \,\ 
f(\frac{z}{\tau}|-\frac{1}{\tau})=\tau^d\exp(\frac{2\pi i}{\tau}(\sum_iz_i x_{\alpha_i^\vee}+X))f(z|\tau).
\end{align*}
\end{prop}

\subsection{Flat connection on the elliptic moduli space}
\label{subsec:tau direction}
In this section, we collect some facts about the universal flat connection constructed in \cite[Section 9]{TLY1} on the bundle $P_n$, which is an extension of the universal KZB connection $\nabla_{\KZB, \tau}$ to the $\tau$-direction. 
Recall, in Section \ref{sec:conn def} \eqref{function k}, we have the function $ 
k(z, x|\tau)=\frac{\theta(z+x| \tau)}{\theta(z| \tau)\theta(x| \tau)}-\frac{1}{x} \in \Hol(\C-\Lambda_\tau)[[x]]$. Let 
\[g(z, x|\tau):=k_x(z, x|\tau)=
\frac{\theta(z+x|\tau)}{\theta(z|\tau)\theta(x|\tau)} 
\left(
\frac{\theta'}{\theta}(z+x|\tau)-\frac{\theta'}{\theta}(x|\tau)\right)+\frac{1}{x^2}
\]
be the derivative of function $k(z, x|\tau)$ with respect to variable $x$. We have $g(z, x|\tau) \in \Hol(\C-\Lambda_\tau)[[x]]$.

For a power series $\psi(x)=\sum_{n\geq 1}b_{2n}x^{2n}\in \C[\![x]\!]$ with positive even degrees, we define two elements in $\widehat{\Aell\rtimes \dd}$ by \[
\delta_\psi:=\sum_{n\geq 1}b_{2n}\delta_{2n}, \,\ 
\Delta_\psi:=\Delta_0+\delta_\psi=\Delta_0+\sum_{n\geq 1}b_{2n}\delta_{2n}.\]
As in \cite{CEE}, we consider the following power series
\[
\varphi(x)=g(0, 0|\tau)-g(0, x|\tau)=
-\frac{1}{x^2}-(\frac{\theta'}{\theta})'(x|\tau)
+\Big(\frac{1}{x^2}+(\frac{\theta'}{\theta})'(x|\tau)\Big)|_{x=0}
\in \C[\![x]\!]\]
 which has positive even degrees. 
Set $a_{2n}:=-\frac{(2n+1)B_{2n+2}(2i\pi)^{2n+2}}{(2n+2)!}$, where $B_n$ are the Bernoulli numbers given by the expansion $\frac{x}{e^x-1}=\sum_{r\geq 0}\frac{B_r}{r!}x^r$.
Then, the function $\varphi(x)$ has the expansion $\varphi(x)=\sum_{n\geq 1}a_{2n}E_{2n+2}(\tau)x^{2n}$, for some $E_{2n+2}(\tau)$ only depending on $\tau$. 
By our convention, we have the following two elements in $\widehat{\Aell\rtimes \dd}$:
 \[
 \delta_{\varphi}=\sum_{n\geq 1}a_{2n}E_{2n+2}(\tau)\delta_{2n}, \,\ \text{and}\,\
\Delta_{\varphi}=\Delta_0+\delta_{\varphi}=\Delta_0+\sum_{n\geq 1}a_{2n}E_{2n+2}(\tau)\delta_{2n}. \]
Consider the following function on $\h\times \mathfrak{H}$: 
\begin{align*}
\Delta:=\Delta(\underline{\alpha}, \tau)=
&-\frac{1}{2\pi i}\Delta_{\varphi}+\frac{1}{2\pi i}\sum_{\beta \in \Phi^+}g(\beta, \ad\frac{ x_{\beta^\vee}}{2}|\tau)(t_\beta)\\
=&-\frac{1}{2\pi i}\Delta_0-\frac{1}{2\pi i}\sum_{n\geq 1}a_{2n}E_{2n+2}(\tau)\delta_{2n}+\frac{1}{2\pi i}\sum_{\beta \in \Phi^+}g(\beta, \ad\frac{ x_{\beta^\vee}}{2}|\tau)(t_\beta).
\end{align*}
This is a meromorphic function on $\C^n\times \mathfrak{H}$ valued in 
$\widehat{(\Aell\rtimes \dd_+)\rtimes \mathfrak{n}_+}\subset \Lie(G_n)$, (where $\mathfrak{n}_+=\C\Delta_0\subset \sl{2}$). It has only poles at $\bigcup_{\alpha\in \Phi^+, \tau\in \mathfrak{H}} \widetilde{H}_{\alpha, \tau}$.

\begin{theorem}\cite[Theorem D]{TLY1}\label{thm:moduli}
The following $\widehat{\Aell\rtimes\dd}$-valued KZB connection on $\M_{1, n}$ is flat.
\begin{equation}\label{conn:extension}
\nabla_{\KZB}=\nabla_{\KZB, \tau}-\Delta d\tau=
d-\sum_{\alpha \in \Phi^+} k(\alpha, \ad\frac{x_{\alpha^\vee}}{2}|\tau)(t_\alpha)d\alpha+\sum_{i=1}^{n}
y(u^i)du_i-\Delta d\tau.
\end{equation}
\end{theorem}
\section{The $\sl{2}$-triple in the deformed double current algebras}
\label{sec:sl2}
In this section and Section \ref{sec:extension}, we prove Theorem \ref{thm:extension} by constructing an algebra homomorphism from ${\Aell\rtimes\dd}$ to the deformed double current algebra $\DDg$. This shows the elliptic Casimir connection $\nabla_{\Ell, C}$ extends to a flat connection on $\M_{1, n}$ whose coefficients are in  $\DDg$.

To begin with, in this section, we construct an action of $\sl{2}$ on the deformed double current algebra $\DDg$. We first construct an action of $\sl{2}$ coming from the group $\SL_2(\mathbb{C})$ permutation of the two lattices $\g[u], \g[v]$ of $\DDg$. We then show this action is inner. That is, there exists an  $\sl{2}$-triple $\{\IE, \IF, \IH\}$ in $\DDg$, such that the action of $\sl{2}$ is given by taking the commutator $[X, -]$, for $X\in \{\IE, \IF, \IH\}$.
\subsection{The action of group $\SL_2$}
\begin{prop}
\label{prop:sl2 action}
There is a right action
\Omit{\bcomment{ Also it should not be necessary to
check that the Yangian relations are preserved any more, since we are leaning towards removing
them.}} of $\SL_2(\C)$ on $\DDg$.
For $A=\left(
   \begin{array}{cc}
     a_{11} & a_{12} \\
     a_{21} & a_{22} \\
   \end{array}
 \right)
$ $\in \SL_2(\C)$ and $z\in \g$, the action is given by
\begin{align*}
&z\mapsto z, \,\ \,\ 
K(z)\mapsto a_{11}K(z)+a_{12}Q(z), \,\ \,\ 
Q(z)\mapsto a_{22}Q(z)+a_{21}K(z),\\
\text{and}\,\ &P(z)\mapsto (a_{11}a_{22}+a_{12}a_{21})P(z)+a_{11}a_{21}[K(y), K(w)]+a_{12}a_{22}[Q(y), Q(w)],
\end{align*}
where $y, w$ is determined by the equality $z=[y, w]$. 

In particular, we have an order 4 automorphism of $\DDg$ (see \cite[Proposition 12.1]{G1}):
\[z\mapsto z, \,\ K(z)\mapsto -Q(z), \,\ Q(z)\mapsto K(z), \,\ P(z)\mapsto -P(z).\]
\end{prop}
\begin{remark}
The right action in Proposition \ref{prop:sl2 action} can be made to a left $\SL_2(\C)$-action on $\DDg$. We define the left action by $A\cdot z:=z A^{T}$, where $A^T$ is the transpose of $A$. 
\Omit{
For $A=\left(
   \begin{array}{cc}
     a_{11} & a_{12} \\
     a_{21} & a_{22} \\
   \end{array}
 \right)
$ $\in \SL_2(\C)$, this left action of $\SL_2(\C)$ on $\DDg$ is given by
\begin{align*}
&z\mapsto z, \,\ \,\ 
K(z)\mapsto a_{11}K(z)+a_{21}Q(z), \,\ \,\ 
Q(z)\mapsto a_{22}Q(z)+a_{12}K(z),\\
&P(z)\mapsto (a_{11}a_{22}+a_{21}a_{12})P(z)+a_{11}a_{12}[K(y), K(w)]+a_{21}a_{22}[Q(y), Q(w)],
\end{align*}
where $z\in \g$ and $y, w$ is defined such that $z=[y, w]$. }\end{remark}
\begin{proof}[Proof of Proposition \ref{prop:sl2 action}]
We first check that the action of $A\in \SL_{2}(\C)$ preserves the defining relations of $\DDg$. We only check that the action preserves the relation \eqref{equ:KQ}, as the other defining relations are obviously preserved.
For simplicity, set
\[C(X_{\beta_1}, X_{\beta_2}):=-
\frac{\lambda}{4}({\beta_1}, {\beta_2}) S(X_{\beta_1}, X_{\beta_2})+
\frac{\lambda}{4}\sum_{\alpha\in \Phi} S( [X_{\beta_1}, X_{\alpha}], [X_{-\alpha}, X_{\beta_2} ]).\]
The equality $S(a_1, a_2)=S(a_2, a_1)$ implies that $C(X_{\beta_1}, X_{\beta_2})=C( X_{\beta_2}, X_{\beta_1})$.
The relation \eqref{equ:KQ} can be rewritten as
$[K(x), Q(y)]=P([x, y])+C(x, y)$.

Under the action of $A\in \hbox{SL}_2(\mathbb{C})$, it is straightforward to compute the image of $[K(x), Q(y)]$ under the action. We have
\begin{align}
[K(x), &Q(y)] \mapsto [a_{11}K(x)+a_{12}Q(x), a_{22}Q(y)+a_{21}K(y)] \notag\\
             &=a_{11}a_{22}[K(x), Q(y)]+a_{12}a_{21}[Q(x), K(y)]+a_{11}a_{21}[K(x), K(y)]+a_{12}a_{22}[Q(x), Q(y)] \notag\\
             &=a_{11}a_{22}(P([x, y])+C(x, y))-a_{12}a_{21}(P([y, x])+C(y, x))+a_{11}a_{21}[K(x), K(y)]+a_{12}a_{22}[Q(x), Q(y)] \notag\\
             &=(a_{11}a_{22}+a_{12}a_{21})P([x, y])+a_{11}a_{21}[K(x), K(y)]+a_{12}a_{22}[Q(x), Q(y)]+C(x, y).\label{eq:action1}
\end{align}
It is straightforward to see that \eqref{eq:action1} coincides with $(P[x, y]+C(x, y))_{\cdot} A$.
Thus, the action of $A$ preserves the relation \eqref{equ:KQ}.

We then check that it defines a right action of $\hbox{SL}_2$.
It is obvious that  $z_\cdot \hbox{Id}=z$, for any $z\in \DDg$, where $\hbox{Id}$ is the identity matrix of $\SL_2$.
For any $A, B \in \hbox{SL}_2(\C)$, and for any $z\in D_\lambda(\g)$,
it is a direct calculation to show $(z_{\cdot} A)_\cdot B=z_{\cdot} (AB)$.
For the convenience of the readers, we show the less obvious case when $z=P([x, y])$ as follows.

Let $A=$
$\left(\begin{array}{cc}
     a_{11} & a_{12} \\
     a_{21} & a_{22} \\
   \end{array}
 \right)$ and
 $B=$
 $\left(\begin{array}{cc}
     b_{11} & b_{12} \\
     b_{21} & b_{22} \\
   \end{array}
 \right)$ be two elements of $\SL_2(\C)$, we have
\begin{align*}
&(P([x, y])_\cdot A)_\cdot B\\
=&(a_{11}a_{22}+a_{12}a_{21})P([x, y])_\cdot B+a_{11}a_{21}[K(x)_\cdot B, K(y)_\cdot B]+a_{12}a_{22}[Q(x)_\cdot B, Q(y)_\cdot B]\\
=&(a_{11}a_{22}+a_{12}a_{21})(b_{11}b_{22}+b_{12}b_{21})P([x, y])
+(a_{11}a_{22}+a_{12}a_{21})(b_{11}b_{21}[K(x), K(y)]+b_{12}b_{22}[Q(x), Q(y)])\\
&+a_{11}a_{21}[b_{11}K(x)+b_{12}Q(x), b_{11}K(y)+b_{12}Q(y)]+a_{12}a_{22}[b_{22}Q(x)+b_{21}K(x), b_{22}Q(y)+b_{21}K(y)]\\
=&P([x, y])((a_{11}a_{22}+a_{12}a_{21})(b_{11}b_{22}+b_{12}b_{21})+2a_{11}a_{21}b_{11}b_{12}
+2a_{12}a_{22}b_{21}b_{22})\\
&+[K(x),K(y)]((a_{11}a_{22}+a_{12}a_{21})b_{11}b_{21}+a_{11}a_{21}b_{11}^2+a_{12}a_{22}b_{21}^2)\\
&+[Q(x),Q(y)]((a_{11}a_{22}+a_{12}a_{21})b_{12}b_{22}+a_{11}a_{21}b_{12}^2+a_{12}a_{22}b_{22}^2)\\
&+C(x, y)(a_{11}a_{21}b_{11}b_{12}-a_{11}a_{21}b_{11}b_{12}+a_{12}a_{22}b_{21}b_{22}-a_{12}a_{22}b_{21}b_{22})\\
=&P([x, y])((a_{11}b_{11}+a_{12}b_{21})(a_{21}b_{12}+a_{22}b_{22})
+(a_{11}b_{12}+a_{12}b_{22})(a_{21}b_{11}+a_{22}b_{21}))\\
&+[K(x),K(y)](a_{11}b_{11}+a_{12}b_{21})(a_{21}b_{11}+a_{22}b_{21})
+[Q(x),Q(y)](a_{11}b_{12}+a_{12}b_{22})(a_{21}b_{12}+a_{22}b_{22})\\
=&P([x, y])_\cdot (AB).
\end{align*}
Therefore, we have a right action of $\SL_2(\C)$ on $\DDg$. This completes the proof. 
\end{proof}

\begin{corollary}
The $\SL_2(\C)$ action in Proposition \ref{prop:sl2 action} induces a Lie algebra $\sl{2}(\C)$ action on $\DDg$.

For $X=$
 $\left(\begin{array}{cc}
     x_{11} & x_{12} \\
     x_{21} & x_{22} \\
   \end{array}
 \right)$ $\in \sl{2}(\C)$, the action is given by
 \begin{align*}
 &z X=0, \,\ K(z) X=x_{11}K(z)+x_{21}Q(z), \,\ Q(z) X=x_{22}Q(z)+x_{12}K(z),\\
&P(z) X=x_{21}[K(y), K(w)]+x_{21}[Q(y), Q(w)], 
\end{align*}
where $z\in \sl{2}(\C)$ and $y, w$ are determined by the quality $z=[y, w]$.
\end{corollary}
 In particular,
the action of $h=$
 $\left(\begin{array}{cc}
     1 & 0 \\
     0 & -1 \\
   \end{array}
 \right) \in \sl{2}$ is given by
 $h z =0$, $h P(z) =0$, $h K(z) =K(z)$, and $h Q(z) =-Q(z)$.

\subsection{The $\sl{2}$ triple of $\DDg$}
In this subsection, we construct an  $\sl{2}$--triple $\IE, \IF, \IH \in D_\lambda(\g)$, such that
for any $z \in \g$, we have
\begin{itemize}
  \item $[\IH,  z]=0$, $[\IH, K(z)] =-K(z)$ and $[\IH, Q(z)] =Q(z)$.
  \item $[\IE,  z ]=0$, $[\IE, K(z)] =Q(z)$,   and  $[\IE, Q(z)]=0$,
  \item $[\IF,  z] =0$, $[\IF, K(z)] =0$,   and $[\IF, Q(z)]=K(z)$.
\end{itemize}

Recall we have two subalgebras $\g[v]$, $\g[u]$ of $\DDg$. 
Write $K(z)$ as the element $z\otimes v$ in the subalgebra $\g[v]$, and $Q(z)$ as $z\otimes u$ in $\g[u]$.
We will need the following higher degree commutation relation of $\DDg$, which is proved in \cite[Proposition 6.1]{GY} when $\g=\sl{n}$. 

\yaping{
I haven't written a proof for the following proposition. For type A, this is
proved in \cite[Prop. 6.1]{GY}, and the proof takes many pages of computation.  
}

The following result is proved in \cite[Prop. 6.1]{GY} in type $\sfA$.

\begin{prop}\label{Ps}
For any $s\ge 0$ and any $X\in\g$, there exists in $\DDg$ an element $P_s(X)$ with the property that the assignment $X \mapsto P_s(X)$ is linear, $[P_s(X),X'] = P_s([X,X'])$ for any $X'\in\g$, and such that, for all root vectors $X_{\beta_1},X_{\beta_2}\in\g$ with $\beta_1\neq -\beta_2$,  the following relation holds:
\begin{align*}
[K(X_{\beta_1}), X_{\beta_2}\otimes u^s]  = P_s([X_{\beta_1},X_{\beta_2}]) &- 
\lambda\frac{(\beta_1,\beta_2)}{4}  \sum_{p+q=s-1} S(X_{\beta_1}\otimes u^p, X_{\beta_2}\otimes u^q) \\&+ 
\frac{\lambda}{4} \sum_{\alpha\in\Phi}  \sum_{p+q=s-1} S([X_{\beta_1},X_{\alpha}]\otimes u^p,[X_{-\alpha},X_{\beta_2}]\otimes u^q). 
\end{align*}
\end{prop}
\begin{remark}
When $s=0$, $P_0(X) = K(X)$ and the right-hand side equals $K([X_{\beta_1}, X_{\beta_2}])$.
\end{remark} 
\begin{remark}
Write $\sum_{s\geq 1} X_{\beta_2}\otimes u^s=X_{\beta_2}\otimes \frac{u}{1-u}$ as a generating series. The relation in Proposition \ref{Ps} is equivalent to the following relation. 
\begin{align*}
[X_{\beta_1}\otimes v, X_{\beta_2}\otimes \frac{u}{1-u}]  = \sum_{s\geq 1}P_s([X_{\beta_1},X_{\beta_2}]) &- 
\lambda\frac{(\beta_1,\beta_2)}{4}   S(X_{\beta_1}\otimes \frac{1}{1-u}, X_{\beta_2}\otimes \frac{1}{1-u}) \\&+ 
\frac{\lambda}{4} \sum_{\alpha\in\Phi}   S([X_{\beta_1},X_{\alpha}]\otimes \frac{1}{1-u},[X_{-\alpha},X_{\beta_2}]\otimes \frac{1}{1-u}). 
\end{align*}
\end{remark}

For any non-zero element $h\in \h$, let $\widetilde{\IE}(h)$ and $\widetilde{\IF}(h)$ be the following elements of $\DDg$:
\begin{align}
&\widetilde{\IE}(h):=\frac{1}{(h, h)}
\left(
[h\otimes v, h\otimes u^3]-\frac{\lambda}{4}\sum_{p+q=2} 
\sum_{\alpha\in \Phi}S([h, X_{\alpha}]\otimes u^p, [X_{-\alpha}, h]\otimes u^q)\right)\label{tilde E}
\\ &
\widetilde{\IF}(h):=\frac{1}{(h, h)}\left([h\otimes u, h\otimes v^3]+\frac{\lambda}{4}\sum_{p+q=2} 
\sum_{\alpha\in \Phi}S([h, X_{\alpha}]\otimes v^p, [X_{-\alpha}, h]\otimes v^q)\right).\label{tilde F}
\end{align}
\begin{lemma}\label{lem:E indep of h}
The elements $\widetilde{\IE}(h)$, $\widetilde{\IF}(h)$ are independent of the choice of $h\in \h$, for $h\neq 0$. \end{lemma}
\begin{proof}
By linearity, it suffices to show for any $\alpha, \beta\in \Phi^+$, we have $
\widetilde{\IE}({\alpha})=\widetilde{\IE}({\beta}), \,\ \widetilde{\IF}(\alpha)=\widetilde{\IF}(\beta)$. This follows from the same proof as \cite[Proposition 6.2]{GY} with $\beta=\frac{\lambda}{2}$ using the higher degree relations in Proposition \ref{Ps}. The idea of the proof is essentially in the proof of \cite[Proposition 4.1]{GY}.
\end{proof}
Based on Lemma \ref{lem:E indep of h}, we denote $\widetilde{\IE}(h)$ by $\widetilde{\IE}$, and 
$\widetilde{\IF}(h)$ by $\widetilde{\IF}$. 
\begin{corollary}
\label{cor [h, h']}
For any $h, h'\in \h$, such that $(h, h')=0$. We have the following identity in $\DDg$
\begin{equation}\label{eq:identity}
[h' \otimes v, h\otimes u^3]
=\frac{\lambda}{4}\sum_{p+q=2}\sum_{\alpha\in \Phi}(h, \alpha)(h', \alpha)
S(X_{\alpha}\otimes u^p, X_{-\alpha} \otimes u^q).\end{equation}
\end{corollary}
\begin{proof}
The claim is true for $h={\alpha}$, and $h'={\beta}$, where $\alpha \in \Phi$ and $\beta\in \Phi$ are two roots such that $(\alpha, \beta)=0$, which follows from the same proof as \cite[Proposition 6.2]{GY}.  
 The general statement follows from Lemma \ref{lem:E indep of h} and linearity of the formula \eqref{eq:identity} in $h$ and $h'$. 
\end{proof}

We will show the following elements form an $\sl{2}$-triple of $\DDg$: 
\[
\IE:=\frac{\widetilde{\IE}}{C},   \,\    \IF:=\frac{\widetilde{\IF}}{C}, \,\ \IH:=[\IE, \IF],
\] 
where the constant $C\in \Q$ depends on the type of the Lie algebra $\g$, which is given by the following formula. 
\begin{equation}
\label{table:constant C}
C=\lambda^2\frac{\sum_{\{\alpha, \beta\in \Phi \mid \alpha+\beta\in \Phi\}}
\Big(1-\frac{(\alpha, \beta)^2}{ (\alpha, \alpha)(\beta, \beta)}\Big)\Big((\beta, \beta)^2+(\alpha, \alpha)^2\Big)
(\alpha+\beta, \alpha+\beta)}{16\dim\h  (\dim\h-1)}.
\end{equation}
In Appendix \ref{appendix}, we compute the constant $C$ explicitly. 
 \Omit{
\begin{equation}\label{table:constant C}
\begin{tabular}{|c|c|}
\hline
The Lie algebra $\g$ & The constant $C$\\
\hline
$A_{n}$ & $\frac{3}{2} (n+1)\lambda^2$\\
\hline
$B_{n}$ & $3(2n-3)\lambda^2$\\
\hline
$C_{n}$ &$\frac{3}{4}(n+2)\lambda^2$\\
\hline
$D_{n}$ &$6(n-2)\lambda^2$\\
\hline
$E_{6}$ &$36\lambda^2$\\
\hline
$E_{7}$ &$72\lambda^2$\\
\hline
$E_{8}$ &$180 \lambda^2$\\
\hline
$F_{4}$ &$\frac{45}{2}\lambda^2$\\
\hline
$G_{2}$ &$\frac{20}{3}\lambda^2$\\
\hline
\end{tabular}
\end{equation}
}

We have the following  main result in this section. 
\begin{theorem}\label{thm: sl2}
With notations as above, the following holds.
\begin{enumerate}
\item
For any $z \in \g$, we have
\begin{enumerate}
  \item $[\IH,  z]=0$, $[\IH, K(z)] =-K(z)$ and $[\IH, Q(z)] =Q(z)$.
  \item $[\IE,  z ]=0$, $[\IE, K(z)] =Q(z)$,  and  $[\IE, Q(z)]=0$,
  \item $[\IF,  z] =0$, $[\IF, K(z)] =0$,   and $[\IF, Q(z)]=K(z)$.
\end{enumerate}
\item
The elements  $\IE, \IF, \IH$ form an  $\sl{2}$--triple.
\end{enumerate}
\end{theorem}
We prove Theorem \ref{thm: sl2} for the rest of this section. 
\subsection{Proof of Theorem \ref{thm: sl2}(2)}
In this subsection, we check that the triple $\IE, \IF, \IH$ form a Lie algebra $\sl{2}$ using the part (1) of Theorem
\ref{thm: sl2}. That is, we check $[\IH, \IE]= 2\IE$, and $[\IH, \IF]= -2\IF$.

By (1) of Theorem \ref{thm: sl2}, we claim by induction that 
\begin{equation}\label{equ:ad E ind}
\ad(\IE)^n(X\otimes v^n)=n! X\otimes u^n, \,\  \text{for any $X\in \g$}. 
\end{equation}
Indeed, let $X=[X_1, X_2]$, we have $X\otimes v^n=[X_1\otimes v, X_2\otimes v^{n-1}]$.
Therefore, 
\begin{align*}
\ad(\IE)^n(X\otimes v^n)
=&\ad(\IE)^n[X_1\otimes v, X_2 \otimes v^{n-1}]\\
=&\sum_{p+q=n} {n\choose p} [\ad(\IE)^p (X_1\otimes v), \ad(\IE)^q (X_2\otimes v^{n-1})]\\
=& n[X_1\otimes u, (n-1)! X_2\otimes u^{n-1}]
=n! X\otimes u^n.
\end{align*}
This shows the claim \eqref{equ:ad E ind}.

To show that $[\IH, \IE]=2\IE$, it suffices to show $
\ad(\IE)^2(\widetilde{\IF})=-2\widetilde{\IE}$. 
It follows from the following two lemmas and the definitions of $\widetilde{\IE}$ \eqref{tilde E} and $\widetilde{\IF}$ \eqref{tilde F}.

\begin{lemma}\label{lem:1}
For any root $\alpha\in \Phi$, we have
\[
\ad(\IE)^2 [H_{\alpha}\otimes u, H_{\alpha}\otimes v^3]=2[H_{\alpha}\otimes v, H_{\alpha}\otimes u^3].
\]
\end{lemma}
\begin{proof}
We have the general identity
\[
\ad(\IE)^n[A, B]=
\sum_{p+q=n} {n\choose p} [\ad(\IE)^p A, \ad(\IE)^q B].
\]
Choose $A$ to be $h\otimes v$, $B$ to be $h\otimes v^3$, and apply the operator $\ad(\IE)^3$ to the equality $[A, B]=0$.
This gives the following identity 
\begin{align*}
&\ad(\IE)^3[h\otimes v, h\otimes v^3]
=3[h\otimes u, \ad(\IE)^2h\otimes v^3]
+6[h\otimes v, h\otimes u^3]=0.
\end{align*}
Rewrite the above identity, we have
\[
[h\otimes u, \ad(\IE)^2h\otimes v^3]
=\ad(\IE)^2 [h\otimes u, h\otimes v^3]
=-2[h\otimes v, h\otimes u^3].
\] This completes the proof.
\end{proof}
The following lemma is a directly consequence of the identity
\[
\ad(\IE)^n(AB)=
\sum_{p+q=n} {n\choose p} \ad(\IE)^p A\ad(\IE)^q B
\]
and Theorem \ref{thm: sl2} (1).
\begin{lemma}\label{lem:2}
For any $p, q\in \N$, such that $p+q=2$. We have
\[
\ad(\IE)^2
\sum_{\alpha\in \Phi}S([h, X_{\alpha}]\otimes v^p, [X_{-\alpha}, h]\otimes v^q)
=2
\sum_{\alpha\in \Phi}S([h, X_{\alpha}]\otimes u^p, [X_{-\alpha}, h]\otimes u^q).
\]
\end{lemma}
Write $\widetilde{\IF}=\widetilde{\IF}(H_{\alpha})$, by Lemma \ref{lem:1} and Lemma \ref{lem:2}, we have
$\ad(\IE)^2(\widetilde{\IF})=-2\widetilde{\IE}$. Therefore, 
\[
[\IH, \IE]=[[\IE, \IF], \IE]=
-\frac{1}{C}\ad(\IE)^2(\widetilde{\IF})=2\frac{1}{C}\widetilde{\IE}=2\IE. 
\]
By symmetry, we have $[\IH, \IF]=-2\IF$. This complete the proof of Theorem \ref{thm: sl2} (2). 

\subsection{Proof of Theorem \ref{thm: sl2} (1)}
\label{subsec:the proof of sl2(1)}
In this subsection, we prove Theorem \ref{thm: sl2} (1). 
Similar to the proof of \cite[Lemma 6.2]{GY} with the choice of $\beta=\frac{\lambda}{2}$, $s=3$, it is straightforward to check that 
$\IE$ commutes with the subalgebra $\g[u]\subset D_\lambda(\g)$. As a consequence, we have $[\IE, z]=0$ and $[\IE, Q(z)]=0$. 
By the symmetry of $\IE$ and $\IF$, it suffices to show $[\IE, K(z)]=Q(z)$.

In the reminder of this section, we compute the commutator $[\widetilde{\IE}, K(z)]$ and identify it with $C Q(z)$, 
for $C\in \Q$ given by the formula \eqref{table:constant C}. This constant $C$ is simplified further in Appendix \ref{appendix}.
As $[\widetilde{\IE}, z]=0$, to get a formula of $[\widetilde{\IE}, K(z)]$ for arbitrary $z\in \g$, we only need to compute $[\widetilde{\IE}, K(h')]$ for some $h'\in \h$. 

Recall by Lemma \ref{lem:E indep of h}, the element $\widetilde{\IE}(h)$ \eqref{tilde E} is independent of $h\in \h$. 
By convenience of the computation, we choose $h, h'\in \h$, such that $(h, h')=0$. 
Under this assumption $(h, h')=0$, the computation of $[\widetilde{\IE}(h), K(h')]$ is essentially in the proof of \cite[Lemma 6.3]{GY}, which uses the relation in Corollary \ref{cor [h, h']}.
For the convenience of the readers, we include the computation here. 

Assume $(h, h')=0$. 
On the one hand, we have
\begin{align}
[[h\otimes v, h\otimes u^3], h' \otimes v]
=&-[h\otimes v, [h' \otimes v, h\otimes u^3]] \notag
=-\frac{\lambda}{4}\Big[h\otimes v,  \sum_{p+q=2}\sum_{\alpha\in \Phi}(h, \alpha)(h', \alpha)
S(X_{\alpha}\otimes u^p, X_{-\alpha} \otimes u^q)\Big] \notag\\
=&-\frac{\lambda}{4} \sum_{p+q=2}\sum_{\alpha\in \Phi}
(h, \alpha)(h', \alpha)S([h\otimes v, X_{\alpha}\otimes u^p], X_{-\alpha} \otimes u^q)] \label{eq: [hv, hu^3]}\\
&-\frac{\lambda}{4} \sum_{\alpha\in \Phi}\sum_{p+q=2}(h, \alpha)(h', \alpha)
S(X_{\alpha}\otimes u^p, [h\otimes v, X_{-\alpha} \otimes u^q])]
\notag
\end{align}
On the other hand, we have:
\begin{align}
&\Big[S([h, X_{\alpha}]\otimes u^p, [X_{-\alpha}, h]\otimes u^q), \,\ h' \otimes v\Big]
=\Big[(h, \alpha)^2 S(X_{\alpha}\otimes u^p, X_{-\alpha}\otimes u^q), \,\ h' \otimes v\Big] \notag\\
\Omit{
=&(h, \alpha)^2 S([X_{\alpha}\otimes u^p, h' \otimes v], X_{-\alpha}\otimes u^q)+
(h, \alpha)^2 S(X_{\alpha}\otimes u^p, [X_{-\alpha}\otimes u^q, h' \otimes v])\\}
=&-(h, \alpha)^2 S([h' \otimes v, X_{\alpha}\otimes u^p], X_{-\alpha}\otimes u^q)
-(h, \alpha)^2 S(X_{\alpha}\otimes u^p, [h' \otimes v, X_{-\alpha}\otimes u^q])
\label{[S, h'v]}
\end{align}
Plugging the computations of \eqref{eq: [hv, hu^3]} and \eqref{[S, h'v]} into the definition of $\widetilde{\IE}$ \eqref{tilde E} and arranging the summands, we have:
\begin{align}
(h, h)[\widetilde{\IE}, & \,\ h'\otimes v]=
\left[[h\otimes v, h\otimes u^3]-\frac{\lambda}{4}\sum_{p+q=2} 
\sum_{\alpha\in \Phi}S([h, X_{\alpha}]\otimes u^p, [X_{-\alpha}, h]\otimes u^q), \,\ h'\otimes v\right] \notag\\
=&
-\frac{\lambda}{4} \sum_{p+q=2}\sum_{\alpha\in \Phi}
(h, \alpha)S\Big( \Big[\big((h', \alpha)h-(h, \alpha)h'\big)\otimes v, X_{\alpha}\otimes u^p\Big], X_{-\alpha} \otimes u^q\Big) \notag\\
&-\frac{\lambda}{4} \sum_{\alpha\in \Phi}\sum_{p+q=2}(h, \alpha)
S\Big(X_{\alpha}\otimes u^p, \Big[\big((h', \alpha)h-(h, \alpha)h' \big)\otimes v, X_{-\alpha} \otimes u^q\Big]\Big) \notag\\
=&-\frac{\lambda}{2} \sum_{p+q=2,}\sum_{\alpha\in \Phi}
(h, \alpha)S\Big( \Big[\big((h', \alpha)h-(h, \alpha)h'\big)\otimes v, X_{\alpha}\otimes u^p\Big], X_{-\alpha} \otimes u^q\Big) \notag\\
=&-\frac{\lambda}{2} \sum_{p+q=2, p>0}\sum_{\alpha\in \Phi}(h, \alpha)
\frac{\lambda}{4}\sum_{\beta\in \Phi}\sum_{s+t=p-1} 
S\left(S\Big([((h', \alpha)h-(h, \alpha)h'), X_{\beta}]\otimes u^s, [X_{-\beta}, X_{\alpha}]\otimes u^t\Big),
X_{-\alpha} \otimes u^q\right) \notag\\
=&-\frac{\lambda^2}{8} \sum_{s+t+q=1}\sum_{\alpha, \beta\in \Phi}\Big((h, \alpha)(h', \alpha)(h, \beta)-(h, \alpha)^2(h', \beta)\Big)
 S\Big(S\Big(X_{\beta} \otimes u^s, [X_{-\beta}, X_{\alpha}]\otimes u^t\Big),
X_{-\alpha} \otimes u^q\Big)
\label{eq:[tildeE, h']}
\end{align}
We use the trick in \cite[\S 6.2]{GY} to simplify equation \eqref{eq:[tildeE, h']}. 
By the formula of \cite[\S 6.2, Page 1357]{GY}, we have the following identity
\begin{align}
[\widetilde{\IE}, h'\otimes v]
=\frac{\lambda^2}{4(h, h)}\sum_{\alpha, \beta\in \Phi}
\Big((h, \alpha)(h', \alpha)(h, \beta)-(h, \alpha)^2(h', \beta)\Big)
\Big([X_{\beta},  X_{-\alpha}] \mid [X_{-\beta}, X_{\alpha}]\Big)\,\ \beta \otimes u,
\label{equ:[tildeE, h' otimes v]}
\end{align}
where $(h, h')=0$. 

It remains to simply the right hand side of equation \eqref{equ:[tildeE, h' otimes v]} under the assumption $(h, h')=0$. For convenience of the notation, we set
\[
F(h, h'):=\sum_{\alpha, \beta\in \Phi}
\Big((h, \alpha)(h', \alpha)(h, \beta)-(h, \alpha)^2(h', \beta)\Big)
\Big([X_{\beta},  X_{-\alpha}] \mid [X_{-\beta}, X_{\alpha}]\Big) \,\ \beta.
\] 
Note that the element $F(h, h')$ is well-defined for any $h, h'\in \h$. In general, $F(h, h')$ is some element in $\h$.
For $(h, h')=0$, equation \eqref{equ:[tildeE, h' otimes v]} is the same as $[\widetilde{\IE}, h'\otimes v]=\frac{\lambda^2}{4(h, h)}F(h, h') \otimes u$.
By Lemma \ref{lem:E indep of h}, $\widetilde{\IE}$ is independent of the choice of $h\in \h$. As a consequence, suppose $(h, h')=0$, then the element $\frac{F(h, h')}{(h, h)}$ is also independent of the choice of $h\in \h$.

We now list some properties of the element $F(h, h')$. 
\begin{lemma}
\label{lem:element F}
The following holds.
\begin{romenum}
\item The element $F(h, h')$ is linear in $h'$. 
\item If $h$ and $h'$ are parallel to each other, we have $F(h, h')=0$. 
\item For any $h, h'\in \h$, write $h'=h'_{\parallel}+h'_{\perp}$, where $h'_{\parallel}$ is parallel to $h$, and $h'_{\perp}$ is perpendicular to $h$, then 
we have $ F(h, h')= F(h, h'_{\perp})$. 
\item For any $h, h'\in \h$, the inner product $(F(h, h'), h)$ is zero. 
\item $F(h, h')=(h, h) \widetilde{C} h'$, for some constant $\widetilde{C}\in \C$.
\end{romenum}
\end{lemma}
\begin{proof}
(i) and (ii) are clear from the definition of $F(h, h')$. (iii) follows easily from (i), (ii).

For (iv), we have
\begin{align*}
(F(h, h'), h)=&\sum_{\alpha, \beta\in \Phi}
\Big((h, \alpha)(h', \alpha)(h, \beta)^2-(h, \alpha)^2(h', \beta)(h, \beta)\Big)
\Big([X_{\beta},  X_{-\alpha}] \mid [X_{-\beta}, X_{\alpha}]\Big)\\
=&\sum_{\alpha, \beta\in \Phi}
\Big((h, \beta)(h', \beta)(h, \alpha)^2-(h, \beta)^2(h', \alpha)(h, \alpha)\Big)
\Big([X_{\alpha},  X_{-\beta}] \mid [X_{-\alpha}, X_{\beta}]\Big)
=-(F(h, h'), h).
\end{align*} This concludes (iv).

For (v), we fix $h'\in \h$ and let $h$ vary in $P_{h'}$, where $P_{h'}$ is the plane perpendicular to $h'$. 
We have the fact that if $(h, h')=0$, the element $\frac{F(h, h')}{(h, h)}$ is  independent of the choice of $h$. By (iv), $\frac{F(h, h')}{(h, h)}$ is perpendicular to $P_{h'}$. Therefore, there exists some $\widetilde{C}$, such that $F(h, h')=(h, h) \widetilde{C} h'$. This completes the proof. 
\end{proof}
By Lemma \ref{lem:element F} (v), we have 
\[
[\widetilde{\IE}, h'\otimes v]=
\frac{\lambda^2}{4} \widetilde{C} h' \otimes u, \,\ \text{therefore, the constant $C$ is the same as $C=\frac{\lambda^2}{4}\widetilde{C}$.
}
\] 
For the rest of this section, 
we determine the constant $\widetilde{C}$, and therefore give an explicit formula of the constant $C$.
Let $\{h_{a}\}$ be a basis of $\h$, and $\{h^{a}\}$ be the dual basis.
On one hand, we have:
\begin{align}
 \sum_{a}(F(h, h_{a}), h^a)
 =&\sum_{\alpha, \beta\in \Phi}\sum_{a}
\Big((h, \alpha)(h_a, \alpha)(h, \beta)(\beta, h^a)-(h, \alpha)^2(h_a, \beta)(h^a, \beta)\Big)
\Big([X_{\beta},  X_{-\alpha}] \mid [X_{-\beta}, X_{\alpha}]\Big) \notag\\
=&\sum_{\alpha, \beta\in \Phi}
\Big((h, \alpha)(\alpha, \beta)(h, \beta)-(h, \alpha)^2(\beta, \beta)\Big)
\Big([X_{\beta},  X_{-\alpha}] \mid [X_{-\beta}, X_{\alpha}]\Big).\label{F1}
\end{align}
On the other hand, by Lemma  \ref{lem:element F} (iii) and (v), we have
$F(h, h_a)=(h, h)\widetilde{C} \left(h_a-\frac{(h_a, h)}{(h, h)}h\right)$. 
Therefore, 
\begin{align}
\sum_{a}(F(h, h_{a}), h^a)
=&\sum_{a} (h, h)\widetilde{C} \left((h_a, h^a)-\frac{(h_a, h)}{(h, h)}(h, h^a)\right)
=(h, h) \widetilde{C} (\dim\h-1).\label{F2}
\end{align}
Combining \eqref{F1} and \eqref{F2}, we have:
\[
\sum_{\alpha, \beta\in \Phi}
\Big((h, \alpha)(\alpha, \beta)(h, \beta)-(h, \alpha)^2(\beta, \beta)\Big)
\Big([X_{\beta},  X_{-\alpha}] \mid [X_{-\beta}, X_{\alpha}]\Big)
=(h, h) \widetilde{C} (\dim\h-1).
\]
Choosing $h=\gamma \in \Phi^+$, and taking the sum of $\gamma$ over $\Phi^+$,
we have
\[
\sum_{\alpha, \beta\in \Phi, \gamma\in \Phi^+}
\Big((\gamma, \alpha)(\alpha, \beta)(\gamma, \beta)-(\gamma, \alpha)^2(\beta, \beta)\Big)
\Big([X_{\beta},  X_{-\alpha}] \mid [X_{-\beta}, X_{\alpha}]\Big)
=\sum_{\gamma\in \Phi^+}(\gamma, \gamma) \widetilde{C} (\dim\h-1).
\]
Using the identity $(\cdot|\cdot)=\frac{1}{h^\vee}\sum_{\gamma\in \Phi^+}\langle \cdot, \gamma\rangle\langle \cdot, \gamma\rangle$ in Lemma \ref{lem:phi and h}, we simplify the above equality as
\begin{align*}
&h^\vee \sum_{\alpha, \beta\in \Phi}
\Big((\beta, \alpha)(\alpha, \beta)-(\alpha, \alpha)(\beta, \beta)\Big)
\Big([X_{\beta},  X_{-\alpha}] \mid [X_{-\beta}, X_{\alpha}]\Big)
=\sum_{\gamma\in \Phi^+}(\gamma, \gamma) \widetilde{C} (\dim\h-1). 
\Omit{
=&\sum_{\gamma\in \Phi^+}\sum_{a}(\gamma, h_a)(h^a, \gamma) \widetilde{C} (\dim\h-1)
=h^\vee\sum_a (h_a, h^a)\widetilde{C} (\dim\h-1)}
\end{align*}
Write $(\gamma, \gamma)=\sum_{a}(h_a, \gamma)(h^a, \gamma)$. We have
\[
\sum_{\gamma\in \Phi^+}(\gamma, \gamma)
=\sum_{\gamma\in \Phi^+}\sum_{a}(h_a, \gamma)(h^a, \gamma)
=\sum_{a}(h_a, h^a) h^{\vee}
= h^{\vee}\dim\h. 
\]
Therefore, $h^\vee \sum_{\alpha, \beta\in \Phi}
\Big((\beta, \alpha)(\alpha, \beta)-(\alpha, \alpha)(\beta, \beta)\Big)
\Big([X_{\beta},  X_{-\alpha}] \mid [X_{-\beta}, X_{\alpha}]\Big)=h^\vee\dim\h  (\dim\h-1)\widetilde{C}$.

This gives an explicit formula of $\widetilde{C}$: 
\begin{equation}\label{eq:tilde C}
\widetilde{C}=\frac{\sum_{\alpha, \beta\in \Phi}
\Big((\alpha, \beta)^2-(\alpha, \alpha)(\beta, \beta)\Big)
\Big([X_{\beta},  X_{-\alpha}] \mid [X_{-\beta}, X_{\alpha}]\Big) }{\dim\h  (\dim\h-1)}, 
\end{equation}
and the constant $C$ is given by $C=\frac{\lambda^2}{4}\widetilde{C}$. We compute the constant $\widetilde{C}$ in Appendix \ref{appendix} using \eqref{eq:tilde C}. 

\section{The extension of the elliptic Casimir connection}
\label{sec:extension}
In this section, we extend the derivation action of $\dd$ on the Lie algebra  $\Aell$ in Section \ref{sec:derivation} to an action on the deformed double current algebra $\DDg$. We show furthermore that the action of $\dd$ on $\DDg$ is inner.
 \subsection{}
Set
\begin{equation}\label{eqn:E(n)}
\widetilde{\IE}(n):=\frac{1}{(h, h)}
\left([h\otimes v, h\otimes u^n]-\frac{\lambda}{4}\sum_{p+q=n-1} 
\sum_{\alpha\in \Phi}S([h, X_{\alpha}]\otimes u^p, [X_{-\alpha}, h]\otimes u^q)\right).
\end{equation}
We have the following facts about the element $\widetilde{\IE}(n)$:
\begin{prop}\label{prop:about tilde E}
\leavevmode
\begin{enumerate}
\item We have the following equality.
\begin{align*}
\widetilde{\IE}
(n)=\widetilde{\IE}(\alpha, n):=[X_{-\alpha} \otimes v, X_{\alpha} \otimes u^n]&
+P_n(H_{\alpha})-\lambda\frac{(\alpha, \alpha)}{4} \sum_{p+q=n-1}S(X_{-\alpha} \otimes u^p, X_{\alpha} \otimes u^q)\\
&-\frac{\lambda}{4}\sum_{p+q=n-1}\sum_{\beta\in \Phi}S([X_{-\alpha}, X_{\beta}] \otimes u^p, [X_{-\beta}, X_{\alpha}] \otimes u^q).
\end{align*}
\item For any two roots $\alpha, \beta \in \Phi$, we have $\widetilde{\IE}(\alpha, n)=\widetilde{\IE}(\beta, n)$. 
\item $\widetilde{\IE}(n)$ commutes with the subalgebra $\g[u]$ of $\DDg$.
\item If $n=1$, $\widetilde{\IE}(n)$ is a central element of $\DDg$. 
\item For any $n\geq 2$, we have:
\[
[\widetilde{\IE}(n), z\otimes v]=\frac{C}{3} {n \choose 2} z\otimes u^{n-2}, 
\] 
where the constant $C\in \Q$ is given as in \eqref{table:constant C}.
\end{enumerate}
\end{prop}
\begin{proof}
The Proposition is a consequence of the higher degree relation in Proposition \ref{Ps}.
(1) and (2) follow from the same proof as \cite[Proposition 6.2]{GY}. (3) follows from the same proof as \cite[Lemma 6.2]{GY}. 
(4) is proved in \cite[Theorem 4.1]{GY}. 
(5) is a similar computation as \eqref{equ:[tildeE, h' otimes v]}.
Indeed, by \cite[\S 6.2]{GY},  if $(h, h')=0$, we have
\begin{align}
&(h, h)[\widetilde{\IE}(n),  \,\ h'\otimes v] \notag\\
=&-\frac{\lambda^2}{8} \sum_{s+t+q=n-2}\sum_{\alpha, \beta\in \Phi}\Big((h, \alpha)(h', \alpha)(h, \beta)-(h, \alpha)^2(h', \beta)\Big)
 S\Big(S\Big(X_{\beta} \otimes u^s, [X_{-\beta}, X_{\alpha}]\otimes u^t\Big),
X_{-\alpha} \otimes u^q\Big)\notag\\
=&\frac{\lambda^2}{4}\frac{1}{3}\binom{n}{2}\sum_{\alpha, \beta\in \Phi}
\Big((h, \alpha)(h', \alpha)(h, \beta)-(h, \alpha)^2(h', \beta)\Big)
\big([X_{\beta},  X_{-\alpha}] , [X_{-\beta}, X_{\alpha}]\big) \beta \otimes u^{n-2}). \label{eq:higher E}
\end{align}
In \S\ref{subsec:the proof of sl2(1)}, the formula \eqref{equ:[tildeE, h' otimes v]} implies $[\widetilde{\IE},  h'\otimes v]=C h'\otimes u$, for $C$ given in \eqref{table:constant C}. Comparing the formula \eqref{eq:higher E} with \eqref{equ:[tildeE, h' otimes v]}, 
we conclude that $[\widetilde{\IE}(n),  h'\otimes v]= \frac{C}{3}\binom{n}{2} h'\otimes u^{n-2}$, for the same constant $C$. 
This completes the proof of (5). 
\end{proof}
\Omit{
Let $h\otimes u:=Q(h)\in \g[u]\subset \dd_{\lambda}(\g)$,
and $h\otimes v:=K(h)\in \g[v] \subset \dd_{\lambda}(\g)$, then for any root $\alpha\in \Phi$, we have:
\begin{align*}
&[h\otimes v, H_{\alpha}\otimes u^n]
=[h\otimes v, [X_{\alpha}\otimes u^n, X_{-\alpha}]]\\
=&[[h\otimes v, X_{\alpha}\otimes u^n], X_{-\alpha}]+[X_{\alpha}\otimes u^n, [h\otimes v, X_{-\alpha}]]\\
=&\left[ (h, \alpha)P(n)(X_{\alpha})+\frac{\lambda}{4}\sum_{\beta\in \Phi}\sum_{p+q=n-1} S([h, X_{\beta}]\otimes u^p, [X_{-\beta}, X_{\alpha}]\otimes u^q), X_{-\alpha}\right]
+(h, -\alpha)[X_{\alpha}\otimes u^n, X_{-\alpha}\otimes v]\\
=&(h, \alpha)P(n)(H_{\alpha})
+\frac{\lambda}{4}(h, -\alpha)\sum_{\beta\in \Phi}\sum_{p+q=n-1} S([X_{-\alpha}, X_{\beta}]\otimes u^p, [X_{-\beta}, X_{\alpha}]\otimes u^q)\\
&+\frac{\lambda}{4}\sum_{\beta\in \Phi}\sum_{p+q=n-1} S([h, X_{\beta}]\otimes u^p, [X_{-\beta}, H_{\alpha}]\otimes u^q)\\
&+\lambda\frac{(\alpha, \alpha)}{4}(h, -\alpha)  \sum_{p+q=n-1} S(X_{-\alpha}\otimes u^p, X_{\alpha}\otimes u^q)
+(h, \alpha)[X_{-\alpha}\otimes v, X_{\alpha}\otimes u^n]\\
=&(h, \alpha) \widetilde{\IE}(n)+\frac{\lambda}{4}\sum_{\beta\in \Phi}\sum_{p+q=n-1} S([h, X_{\beta}]\otimes u^p, [X_{-\beta}, H_{\alpha}]\otimes u^q)
\end{align*}
The above formula holds for any root $\alpha\in \Phi$. 
}
 \subsection{}
Let $\{h_i\}_{1\leq i\leq n}$ be the basis of $\h$, and  $\{h^i\}_{1\leq i\leq n}$ be the corresponding dual basis and write $h\otimes u:=Q(h)\in \g[u]$.
Set
\begin{align}
\delta_{2m}
=\frac{\lambda}{2}\sum_{p=0}^{2m}(-1)^p{2m \choose p} &\sum_i(h_i\otimes u^p)(h^i\otimes u^{2m-p})
+\frac{3\lambda}{ C}\Bigg(\sum_{p=0}^{m-2}
(-1)^p\frac{2(2m)!}{(p+2)! (2m-p)!}\widetilde{\IE}(p+2)\widetilde{\IE}(2m-p) \notag\\
&+(-1)^{m-1} \frac{(2m)!}{(m+1)!(m+1)!} \widetilde{\IE}(m+1)^2-\frac{2}{2m+1}\widetilde{\IE}(2m+1)\widetilde{\IE}(1)\Bigg), 
\label{eq:delta}
\end{align}
where the constant $C$ is given in \eqref{table:constant C}. 

\begin{prop}\label{prop:derivation action}
The $\sl{2}$-triple $\{\IH, \IE, \IF\}$ and $\delta_{2m}$ satisfy the relations of the derivation algebra $\dd$.
\end{prop}
\begin{proof}
We need to check that
\[
[\delta_{2m}, \IE]=0, \,\ [\IH, \delta_{2m}]=2m\delta_{2m}, \,\ \text{and $(\ad\IF)^{2m+1}(\delta_{2m})=0$.}
\]
Recall that we have, for any $z \in \g$,
\begin{enumerate}
  \item $[\IH,  z]=0$, $[\IH, K(z)] =-K(z)$ and $[\IH, Q(z)] =Q(z)$.
  \item $[\IE,  z ]=0$, $[\IE, K(z)] =Q(z)$,  $[\IE, Q(z)]=0$.
  \item $[\IF,  z] =0$, $[\IF, K(z)] =0$,  $[\IF, Q(z)]=K(z)$.
\end{enumerate}

The fact that $\IE$ commutes with $\g[u]$ and $[\IE, K(z)] =Q(z)$ implies 
$[\IE, \widetilde{\IE}(n)] =0$, for any $n\in \N$.
Thus,
$[\delta_{2m}, \IE]=0$.

By induction, we have $[\IH, z\otimes v^n]=-n z\otimes v^n$, and $[\IH, z\otimes u^n]=n z\otimes u^n$.
Thus, $[\IH, \widetilde{\IE}(n)]=(n-1)\widetilde{\IE}(n)$.
A direct calculation shows that $[\IH, \delta_{2m}]=2m\delta_{2m}$.

By induction, $(\ad\IF)^n (z\otimes u^k)=0$, for $k< n$, and $(\ad\IF)^n(z\otimes u^n)=n! z\otimes v^n$.
Thus, $(\ad\IF)^n (\widetilde{\IE}(k))=0,$ for $k\leq n$, which implies 
$
(\ad\IF)^{2m+1}(\delta_{2m})=0.
$
This completes the proof. 
\end{proof}
 \subsection{}
We have a subalgebra $\dd$ of the deformed double current algebra $\DDg$. For any element $X\in \dd$, taking
$[X, \cdot]$ gives a derivation action of $\dd$ on $\DDg$.
 \begin{theorem}\label{thm: derivation on D(g)}
 The action of $\dd$ on $\DDg$ is extended from the action of $\dd$ on $\Aell$. In other words, 
 the following diagram commutes.
 \begin{equation*}
\xymatrix{ 
\dd \times \Aell  \ar[r]\ar[d] & \Aell \ar[d]\\
\dd \times \DDg \ar[r]         & \DDg\\
}
\end{equation*}
 \end{theorem}
\begin{proof}
 It is clear that the action of the $\sl{2}$ triple $\{\IE, \IF, \IH\}$ makes the diagram commute.
 
 It remains to show that the action of $\delta_{2m}\in \dd$ makes the diagram commute. That is, we need to show that:
\begin{enumerate}
\item $ [\delta_{2m}, S(X_{\alpha}^+, X_{\alpha}^-)]=\frac{\lambda}{2}[S(X_{\alpha}^+, X_{\alpha}^-), 
(\ad\frac{Q(\alpha^\vee)}{2})^{2m}S(X_{\alpha}^+, X_{\alpha}^-)]$.\label{item1_delta}
 \item $ [\delta_{2m},  Q(h)]=0 $, for any $h\in \h$. \label{item2_delta}
 \item $ [\delta_{2m},  K(h)]=\frac{\lambda^2}{8}\sum_{\alpha\in \Phi^+}\alpha(h)\sum_{p+q=2m-1}[(\ad\frac{Q(\alpha^\vee)}{2})^{p}
   (\kappa_\alpha), (\ad-\frac{Q(\alpha^\vee)}{2})^{q}(\kappa_\alpha)]$, for any $h\in \h$.\label{item3_delta}
\end{enumerate}
For any $X_{\alpha} \in \g_{\alpha}$, as $[ \widetilde{\IE}(n), X_\alpha]=0$, we have
\begin{align*}
[\delta_{2m}, X_{\alpha}]
=&\frac{\lambda}{2}\left[\sum_{p=0}^{2m}(-1)^p{2m \choose p}\sum_i(h_i\otimes u^p)(h^i\otimes u^{2m-p}), X_{\alpha}\right]\\
=&\frac{\lambda}{2}\sum_{p=0}^{2m}(-1)^p{2m \choose p}
S(X_{\alpha}\otimes u^p, H_{\alpha}\otimes u^{2m-p})\\
=&\frac{\lambda}{2}\sum_{p=0}^{2m}(-1)^p{2m \choose p}\Big[X_{\alpha},
S(X_{\alpha}\otimes u^p, X_{\alpha}^{-}\otimes u^{2m-p})\Big]\\
=&\frac{\lambda}{2}\Big[X_{\alpha},
(\ad\frac{Q(\alpha^\vee)}{2})^{2m}S(X_{\alpha}, X_{\alpha}^{-})\Big].
\end{align*}
Thus, 
\begin{align*}
\delta_{2m}S(X_{\alpha}^+, X_{\alpha}^-))
=&S(\delta_{2m}( X_{\alpha}^+), X_{\alpha}^-))+S(X_{\alpha}^+, \delta_{2m} (X_{\alpha}^-)))\\
=&\frac{\lambda}{2}S(\left[X_{\alpha}^+,
(\ad\frac{Q(\alpha^\vee)}{2})^{2m}S(X_{\alpha}^{+}, X_{\alpha}^{-})\right], X_{\alpha}^-)+
\frac{\lambda}{2}S(X_{\alpha}^+, \left[X_{\alpha}^-,
(\ad\frac{Q(\alpha^\vee)}{2})^{2m}S(X_{\alpha}^-, X_{\alpha}^{+})\right])\\
=&\frac{\lambda}{2}[S(X_{\alpha}^+, X_{\alpha}^-), 
(\ad\frac{Q(\alpha^\vee)}{2})^{2m}S(X_{\alpha}^+, X_{\alpha}^-)].
\end{align*}   The assertion \eqref{item1_delta} follows.
 
 The assertion \eqref{item2_delta} is clear. 
 Indeed, $\widetilde{\IE}(n)$ commutes with $\g[u]$, and $[h'\otimes u^n, Q(h)]=0$, for any $h, h'\in \h$, and $n\in \N$. This implies 
  $[\delta_{2m}, Q(h)]=0$.
 
The proof of the  assertion \eqref{item3_delta} is in the next subsection. 
 \end{proof}
 \subsection{}
 In this subsection, we complete the proof of Theorem \ref{thm: derivation on D(g)} by checking assertion  \eqref{item3_delta} .
 
For notational reason, we write 
\[
\widetilde{\delta}_{2m}(K(h)):=\frac{\lambda^2}{8}\sum_{\alpha\in \Phi^+}\alpha(h)\sum_{p+q=2m-1}[(\ad\frac{Q(\alpha^\vee)}{2})^{p}
   (\kappa_\alpha), (\ad-\frac{Q(\alpha^\vee)}{2})^{q}(\kappa_\alpha)].
\]
The assertion  \eqref{item3_delta}  can be rewritten as 
\[
[\delta_{2m}, K(h)]=\widetilde{\delta}_{2m}(K(h)).
\]
We now compute $\widetilde{\delta}_{2m}(K(h))$.
 
\begin{lemma}\label{delta_Lemma1}
For any $h\in \h$, we have
\begin{align*}
\widetilde{\delta}_{2m}(K(h))
&=-\frac{\lambda^2}{4}\sum_{\chi+t+k=2m-1}(-1)^{\chi} {2m \choose \chi}\sum_{\beta\in \Phi^+}(\beta, h)
S\Big( H_{\beta}\otimes u^{\chi}, S (X_{\beta}^-\otimes u^t , X_{\beta}^+\otimes u^k)\Big).
\end{align*}
\end{lemma}
\begin{proof} We have
\begin{align}
&\widetilde{\delta}_{2m}(K(h))\notag\\
=&\frac{\lambda^2}{8}\sum_{\beta \in \Phi^+}\beta(h)\sum_{p+q=2m-1}(-1)^q[(\ad\frac{Q(\beta^\vee)}{2})^{p}
   (S(X_{\beta}^+, X_{\beta}^-)), (\ad\frac{Q(\beta^\vee)}{2})^{q}(S(X_{\beta}^+, X_{\beta}^-))] \notag\\
=&\frac{\lambda^2}{8}\sum_{\beta \in \Phi^+}\beta(h)\sum_{p+q=2m-1}(-1)^q[\sum_{s+t=p}(-1)^t{p\choose s}S(X_{\beta}^+\otimes u^s, X_{\beta}^-\otimes u^t)
, \sum_{k+j=q}(-1)^j{q \choose k}(S(X_{\beta}^+\otimes u^k, X_{\beta}^-\otimes u^j))] \notag\\
=&\frac{\lambda^2}{8}\sum_{\beta \in \Phi^+}\beta(h)\sum_{p+q=2m-1}(-1)^q\sum_{s+t=p}\sum_{k+j=q}(-1)^{t+j}{p\choose s}{q \choose k}\,\
\Big[S(X_{\beta}^+\otimes u^s, X_{\beta}^-\otimes u^t)
,S(X_{\beta}^+\otimes u^k, X_{\beta}^-\otimes u^j)\Big] \notag\\
=&\frac{\lambda^2}{8}\sum_{\beta \in \Phi^+}\beta(h)\sum_{p+q=2m-1}(-1)^q\sum_{s+t=p}\sum_{k+j=q}(-1)^{t+j}{p\choose s}{q \choose k} \label{eqn:term}\\
&\phantom{123456789}\Big(S(S(X_{\beta}^+\otimes u^k, H_{\beta}\otimes u^{s+j}), X_{\beta}^-\otimes u^t)
-S(X_{\beta}^+\otimes u^s, S(H_{\beta}\otimes u^{t+k}, X_{\beta}^-\otimes u^j))\Big)\notag
\end{align}

We compute the summand in \eqref{eqn:term} as follows. For a fixed positive root $\beta\in \Phi^+$, we have
\begin{align*}
&S(S(X_{\beta}^+\otimes u^k, H_{\beta}\otimes u^{s+j}), X_{\beta}^-\otimes u^t)-S(X_{\beta}^+\otimes u^s, S(H_{\beta}\otimes u^{t+k}, X_{\beta}^-\otimes u^j))\\
=
&X_{\beta}^-\otimes u^t \Big( 2(X_{\beta}^+\otimes u^k)( H_{\beta}\otimes u^{s+j})+(\beta, \beta) X_{\beta}^+\otimes u^{k+s+j} \Big)\\
&+\Big(2(H_{\beta}\otimes u^{s+j})(X_{\beta}^+\otimes u^k)-(\beta, \beta)   X_{\beta}^+\otimes u^{k+s+j}\Big)X_{\beta}^-\otimes u^t\\
&-X_{\beta}^+\otimes u^s \Big( 2(X_{\beta}^-\otimes u^j)H_{\beta}\otimes u^{t+k}-(\beta, \beta)X_{\beta}^-\otimes u^{j+t+k} \Big)\\
&-\Big( 2(H_{\beta}\otimes u^{t+k})( X_{\beta}^-\otimes u^j) +(\beta, \beta) X_{\beta}^-\otimes u^{j+t+k}\Big)X_{\beta}^+\otimes u^s\\
=&2 (X_{\beta}^-\otimes u^t )(X_{\beta}^+\otimes u^k)H_{\beta}\otimes u^{s+j} +2 H_{\beta}\otimes u^{s+j}(X_{\beta}^+\otimes u^k)(X_{\beta}^-\otimes u^t)  \\
&-2(X_{\beta}^+\otimes u^s) (X_{\beta}^-\otimes u^j)
H_{\beta}\otimes u^{t+k}-2H_{\beta}\otimes u^{t+k}(X_{\beta}^-\otimes u^j) (X_{\beta}^+\otimes u^s) .
\end{align*}
For any $m\in \N$, we have the equality
\[
\sum_{p+q=2m-1}(-1)^q\sum_{s+t=p}\sum_{k+j=q}(-1)^{t+j}{p\choose s}{q \choose k}
=\sum_{s+t+k+j=2m-1}(-1)^{k+t}{s+t\choose s}{k+j \choose k}.
\]
If we switch the pair $(s, j)$ with the pair $(t, k)$, then the coefficient $(-1)^{k+t}{s+t\choose s}{k+j \choose k}$ in  \eqref{eqn:term} is changed by a negative sign.
Thus, we have two summands of the following form in  \eqref{eqn:term}:
\begin{align}
&2 (X_{\beta}^-\otimes u^t )(X_{\beta}^+\otimes u^k)H_{\beta}\otimes u^{s+j} 
+2 H_{\beta}\otimes u^{s+j}(X_{\beta}^+\otimes u^k)(X_{\beta}^-\otimes u^t)\notag\\
&-2(X_{\beta}^+\otimes u^s) (X_{\beta}^-\otimes u^j)
H_{\beta}\otimes u^{t+k}-2H_{\beta}\otimes u^{t+k}(X_{\beta}^-\otimes u^j) (X_{\beta}^+\otimes u^s),
\label{summand1}
\end{align} and
\begin{align}
&-2 (X_{\beta}^-\otimes u^s )(X_{\beta}^+\otimes u^j)H_{\beta}\otimes u^{k+t} -2 H_{\beta}\otimes u^{k+t}(X_{\beta}^+\otimes u^j)(X_{\beta}^-\otimes u^s)\notag\\
&+2(X_{\beta}^+\otimes u^t) (X_{\beta}^-\otimes u^k)
H_{\beta}\otimes u^{s+j}+2H_{\beta}\otimes u^{s+j}(X_{\beta}^-\otimes u^k) (X_{\beta}^+\otimes u^t).
\label{summand2}
\end{align}
Combining the above two summands \eqref{summand1} \eqref{summand2} with $(s, t, k, j)$ and $(t, s, j, k)$ (note the coefficient is $(-1)^{k+t}{s+t\choose s}{k+j \choose k}$), we get
\begin{align*}
&\Big(2 (X_{\beta}^-\otimes u^t )(X_{\beta}^+\otimes u^k)+2 (X_{\beta}^+\otimes u^t )(X_{\beta}^-\otimes u^k) \Big)H_{\beta}\otimes u^{s+j}\\
&+H_{\beta}\otimes u^{s+j}\Big(2 (X_{\beta}^+\otimes u^k)(X_{\beta}^-\otimes u^t )+2 (X_{\beta}^-\otimes u^k)(X_{\beta}^+\otimes u^t ) \Big)\\
&-\Big(2 (X_{\beta}^+\otimes u^s )(X_{\beta}^-\otimes u^j)+2(X_{\beta}^-\otimes u^s )(X_{\beta}^+\otimes u^j) \Big)H_{\beta}\otimes u^{t+k}\\
&-H_{\beta}\otimes u^{t+k}\Big( 2 (X_{\beta}^-\otimes u^j )(X_{\beta}^+\otimes u^s)+2 (X_{\beta}^+\otimes u^j )(X_{\beta}^-\otimes u^s)\Big)\\
=&\Big(S (X_{\beta}^-\otimes u^t , X_{\beta}^+\otimes u^k)+ S(X_{\beta}^+\otimes u^t , X_{\beta}^-\otimes u^k) \Big)H_{\beta}\otimes u^{s+j}\\
&+H_{\beta}\otimes u^{s+j}\Big(S (X_{\beta}^+\otimes u^k, X_{\beta}^-\otimes u^t )+S (X_{\beta}^-\otimes u^k, X_{\beta}^+\otimes u^t ) \Big)\\
&-\Big(S (X_{\beta}^+\otimes u^s , X_{\beta}^-\otimes u^j)+S(X_{\beta}^-\otimes u^s, X_{\beta}^+\otimes u^j) \Big)H_{\beta}\otimes u^{t+k}\\&-H_{\beta}\otimes u^{t+k}\Big( S (X_{\beta}^-\otimes u^j, X_{\beta}^+\otimes u^s)+S(X_{\beta}^+\otimes u^j , X_{\beta}^-\otimes u^s)\Big)\\
=&S\Big(S (X_{\beta}^-\otimes u^t , X_{\beta}^+\otimes u^k), H_{\beta}\otimes u^{s+j}\Big) +S\Big( S(X_{\beta}^+\otimes u^t , X_{\beta}^-\otimes u^k),   H_{\beta}\otimes u^{s+j}\Big)\\
&-S\Big( S (X_{\beta}^+\otimes u^s , X_{\beta}^-\otimes u^j), H_{\beta}\otimes u^{t+k}\Big)
-S\Big(  S(X_{\beta}^+\otimes u^j , X_{\beta}^-\otimes u^s), H_{\beta}\otimes u^{t+k}\Big).
\end{align*}
Thus, for a fixed root $\beta\in \Phi^+$, the summation \eqref{eqn:term} becomes:
\[
\sum_{s+t+k+j=2m-1}(-1)^{k+t}{s+t\choose s}{k+j \choose k}
\left(S\Big(S (X_{\beta}^-\otimes u^t , X_{\beta}^+\otimes u^k), H_{\beta}\otimes u^{s+j}\Big) +S\Big( S(X_{\beta}^+\otimes u^t , X_{\beta}^-\otimes u^k),   H_{\beta}\otimes u^{s+j}\Big)\right).
\]

For any integers $j, k$, and $n$ satisfying $0 \leq j \leq k \leq n$, we have an identity of the binomial coefficients:
\[
\sum_{m=0}^n \binom {m} {j} \binom {n-m}{k-j}= \binom {n+1}{k+1}. 
\]
Now fix the indices $k, t$, and fix the sum $s+j=\chi$, and we vary the index $p=0, \dots, 2m-1$, thus, the indices $s, j, q$ are varying according to $p$.
The coefficient of the term
$
S\Big(S (X_{\beta}^-\otimes u^t , X_{\beta}^+\otimes u^k), H_{\beta}\otimes u^{s+j}\Big)
$
becomes
\[ (-1)^{k+t}{s+t\choose s}{k+j \choose k}
+(-1)^{k+t}{k+j \choose j}{s+t\choose t}
  = 2(-1)^{k+t} \sum_{p=0}^{2m-1}{p\choose t}{2m-1-p \choose 2m-1-\chi-t}=-2(-1)^{\chi} {2m \choose \chi}.\]
We simplify \eqref{eqn:term} using the above observation, and we get
\begin{align*}
\widetilde{\delta}_{2m}(K(h))=-\frac{\lambda^2}{4}\sum_{\chi=0}^{2m-1}(-1)^{\chi} {2m \choose \chi}\sum_{\beta\in \Phi^+}(\beta, h)
S\Big( H_{\beta}\otimes u^{\chi}, \sum_{t+k=2m-1-\chi}S (X_{\beta}^-\otimes u^t , X_{\beta}^+\otimes u^k)\Big).
\end{align*}
This completes the proof. 
\end{proof}

\begin{lemma} \label{lem: action of delta 2m on K(h)}
We have \begin{align*}
&\widetilde{\delta}_{2m}(K(h))
   =\frac{\lambda}{2}\Big[\sum_{i}\sum_{p+q=2m}(-1)^q {2m\choose p} (h_i\otimes u^p)(h^i\otimes u^q) , K(h)\Big]+\lambda\sum_{p+q=2m}(-1)^q {2m\choose p}\Big( (h\otimes u^p)  \widetilde{\IE}(q)\Big).
\end{align*}
\end{lemma}
\begin{proof}
We compute the commutator $\frac{\lambda}{2}\Big[\sum_{i}\sum_{p+q=2m}(-1)^q {2m\choose p} (h_i\otimes u^p)(h^i\otimes u^q) , K(h)\Big]$ and compare it with $\widetilde{\delta}_{2m}(K(h))$ in Lemma \ref{delta_Lemma1}. We have
\begin{align}
&\frac{\lambda}{2}[h\otimes v, \sum_{i}\sum_{p+q=2m}(-1)^q {2m\choose p} (h_i\otimes u^p)(h^i\otimes u^q)] \notag\\
=&\frac{\lambda}{2}\sum_{i}\sum_{p+q=2m}(-1)^q {2m\choose p}
\Big(
h_i\otimes u^p[h\otimes v, h^i\otimes u^q]+[h\otimes v, h_i\otimes u^p]h^i\otimes u^q \Big) \notag\\
=&\frac{\lambda}{2}\sum_{i}\sum_{p+q=2m}(-1)^q {2m\choose p}
\Big(h_i\otimes u^p (h, h^i) \widetilde{\IE}(q)+(h, h_i) \widetilde{\IE}(p) h^i\otimes u^q \Big)\notag\\
&+\frac{\lambda^2}{8}\sum_{p+q=2m}(-1)^q {2m\choose p}
\sum_{\alpha\in \Phi}(\alpha, h)\Big(H_{\alpha}\otimes u^p \sum_{s+t=q-1}S(X_{\alpha}\otimes u^s, X_{-\alpha}\otimes u^t)\notag\\
&+\sum_{s+t=p-1}S(X_{\alpha}\otimes u^s,  X_{-\alpha}\otimes u^t )H_{\alpha}\otimes u^q \label{eq:about K}
\Big),\end{align}
where the last equality follows from the definition of $\widetilde{\IE}(n)$ in \eqref{eqn:E(n)}.  
Using the simple fact that $h=\sum_i(h_i, h)h^i$, we compute \eqref{eq:about K} as follows. 
\begin{align*}
\eqref{eq:about K}=&\frac{\lambda}{2}\sum_{p+q=2m}(-1)^q {2m\choose p}
\Big( (h\otimes u^p)  \widetilde{\IE}(q)+ (h\otimes u^q) \widetilde{\IE}(p)\Big)\\
&+\frac{\lambda^2}{8}\sum_{\alpha\in \Phi}(\alpha, h)\sum_{p+q=2m, q\geq 1} (-1)^q{2m\choose p}\Big(H_{\alpha}\otimes u^p \sum_{s+t=q-1}S(X_{\alpha}\otimes u^s, X_{-\alpha}\otimes u^t)\\
&\phantom{12345678901234567890}+\sum_{s+t=q-1}S(X_{\alpha}\otimes u^s,  X_{-\alpha}\otimes u^t )H_{\alpha}\otimes u^p
\Big)\\
=&\lambda \sum_{p+q=2m}(-1)^q {2m\choose p}
\Big( (h\otimes u^p)  \widetilde{\IE}(q)\Big)\\
&\phantom{12345678901234567890}+\frac{\lambda^2}{4}\sum_{\alpha\in \Phi^+}(\alpha, h)\sum_{p+s+t=2m-1}(-1)^p {2m\choose p}
S(H_{\alpha}\otimes u^p,  S(X_{\alpha}\otimes u^s, X_{-\alpha}\otimes u^t))\\
=&\lambda \sum_{p+q=2m}(-1)^q {2m\choose p}\Big( (h\otimes u^p)  \widetilde{\IE}(q)\Big)
-\widetilde{\delta}_{2m}(K(h)),
\end{align*}
where the last equality follows from Lemma \ref{delta_Lemma1}. 
This completes the proof by rearranging the above identity.
\end{proof}

\begin{lemma}\label{lem:compute delta}
For any $z\in \g$, we have
\begin{align*}
\lambda\sum_{p+q=2m}(-1)^q {2m\choose p}\Big( (z\otimes u^p)  \widetilde{\IE}(q)\Big)
=&\frac{3\lambda }{C}\Big[\sum_{p=0}^{m-2}
(-1)^p\frac{2(2m)!}{(p+2)! (2m-p)!}\widetilde{\IE}(p+2)\widetilde{\IE}(2m-p)\\
&+(-1)^{m-1} \frac{(2m)!}{(m+1)!(m+1)!} \widetilde{\IE}(m+1)^2
-\frac{2}{2m+1}\widetilde{\IE}(2m+1)\widetilde{\IE}(1), \,\
K(z)\Big].
\end{align*}
\end{lemma}
\begin{proof}
The statement follows from the following calculations.
When $q=0$, we have $ \widetilde{\IE}(0)=0$. So we could assume $q\neq 0$. 
We separate the cases when $(p, q)=(m-1, m+1)$, and $(p, q)=(2m-1, 1)$ from the set $\{(p, q)\mid p+q=2m\}$. 
We have
\begin{align}
&\sum_{p+q=2m}(-1)^q {2m\choose p} (z\otimes u^p)  \widetilde{\IE}(q) \notag\\
=&\sum_{p=0}^{m-2}(-1)^p {2m\choose p} (z\otimes u^p)  \widetilde{\IE}(2m-p)
+\sum_{q=2}^{m}(-1)^q {2m\choose q} (z\otimes u^{2m-q})  \widetilde{\IE}(q) \notag\\&
+(-1)^{m+1} {2m\choose m+1} (z\otimes u^{m-1})  \widetilde{\IE}(m+1)
- 2m (z\otimes u^{2m-1})  \widetilde{\IE}(1)
\notag\\
=&\sum_{p=0}^{m-2}
\left(
(-1)^p {2m\choose p} (z\otimes u^p)  \widetilde{\IE}(2m-p) 
+(-1)^p {2m\choose p+2} (z\otimes u^{2m-p-2})  \widetilde{\IE}(p+2) 
\right) \notag\\
&+(-1)^{m-1} {2m\choose m-1} (z\otimes u^{m-1})  \widetilde{\IE}(m+1)
-2m  (z\otimes u^{2m-1})  \widetilde{\IE}(1) \label{eq:manyE}
\end{align}
By Proposition \ref{prop:about tilde E},  $\widetilde{\IE}(1)$ is central, and $[\widetilde{\IE}(n), z\otimes v]=\frac{C}{3} {n \choose 2} z\otimes u^{n-2}$, for $n\geq 2$. Furthermore, $\widetilde{\IE}(n)$ commutes with the subalgebra $\g[u]$. 
Therefore, we have
\begin{align*}
\eqref{eq:manyE}
=&\frac{3}{C}\sum_{p=0}^{m-2}
(-1)^p\frac{{2m+2\choose p+2}}{(2m+1)(m+1)} [\widetilde{\IE}(p+2)\widetilde{\IE}(2m-p), \,\ z\otimes v]
\\
&+\frac{3}{C}(-1)^{m-1} \frac{{2m\choose m-1}}{(m+1)m} [\widetilde{\IE}(m+1)^2, \,\ z\otimes v]
-\frac{6}{C(2m+1)}[\widetilde{\IE}(2m+1)\widetilde{\IE}(1),\,\  z\otimes v], 
\end{align*} where the constant $C$ is given by \eqref{table:constant C}.
This completes the proof. 
\end{proof}

It is clear that the assertion \eqref{item3_delta} follows from Lemma \ref{lem: action of delta 2m on K(h)} and Lemma 
\ref{lem:compute delta}, together with the formula  \eqref{eq:delta} of $\delta_{2m}$. 
This completes the proof of Theorem \ref{thm: derivation on D(g)}.

\appendix
\section{}
\label{appendix}
In the appendix, we show the constant $C=\frac{\lambda^2}{4}\widetilde{C}$ is given by the formula \eqref{table:constant C}. We furthermore compute \eqref{table:constant C} explicitly. 
Recall the following formula of $\widetilde{C}$ in \eqref{eq:tilde C}:
\begin{equation*}
\widetilde{C}=\frac{\sum_{\alpha, \beta\in \Phi}
\Big((\alpha, \beta)^2-(\alpha, \alpha)(\beta, \beta)\Big)
\Big([X_{\beta},  X_{-\alpha}] \mid [X_{-\beta}, X_{\alpha}]\Big) }{\dim\h  (\dim\h-1)}.
\end{equation*}

Let $\g$ be a finite dimensional simple Lie algebra. There is a Chevalley basis of $\g$ (see \cite[\S 25]{H}), which we now recall here. 
Fix a pair of roots $\alpha, \beta$, consider the $\alpha$-string through $\beta$:
\[
\beta-r\alpha, \cdots, \beta, \cdots, \beta+q\alpha.
\]
\begin{prop}\cite[Proposition 25.1]{H}
We have 
\begin{enumerate}
\item $\langle \beta, \alpha \rangle:=\frac{2(\beta, \alpha)}{(\alpha, \alpha)}=r-q$.
\item At most two root lengths occur in this string.
\item If $\alpha+\beta\in \Phi$, then $r+1=\frac{q(\alpha+\beta, \alpha+\beta)}{(\beta, \beta)}$. 
\end{enumerate}
\end{prop}
As a consequence, for any two roots $\alpha, \beta$, if $\alpha+\beta\in \Phi$, then we have
\[
q=\frac{(\beta, \beta)}{(\alpha, \alpha)}, \,\ r=\frac{(\beta, \beta)+2(\beta, \alpha)}{(\alpha, \alpha)}.
\]
\begin{prop}\cite[Proposition 25.2]{H}
It is possible to choose root vectors $x_{\alpha}$, $\alpha\in \Phi$ satisfying
\begin{enumerate}
\item $[x_{\alpha}, x_{-\alpha}]=h_{\alpha}$. 
\item If $\alpha, \beta, \alpha+\beta\in \Phi$, $[x_{\alpha}, x_{\beta}]=c_{\alpha\beta}x_{\alpha+\beta}$, then $c_{\alpha\beta}=-c_{-\alpha, -\beta}$. For any
such choice of root vectors, the scalar $c_{\alpha, \beta}(\alpha, \beta, \alpha+\beta\in \Phi)$ automatically satisfy:
$c_{\alpha, \beta}^2=q(r+1)\frac{(\alpha+\beta, \alpha+\beta)}{(\beta, \beta)}$, where $\beta-r\alpha, \cdots, \beta, \cdots, \beta+q\alpha$ is the $\alpha$-string
through $\beta$. 
\end{enumerate}
\end{prop}
Therefore, for any two roots $\alpha, \beta$, if $\alpha+\beta\in \Phi$, then we have
\[
c_{\alpha, \beta}^2=q(r+1)\frac{(\alpha+\beta, \alpha+\beta)}{(\beta, \beta)}
= \frac{(\beta, \beta)^2}{(\alpha, \alpha)^2}\frac{(\alpha+\beta, \alpha+\beta)^2}{(\beta, \beta)^2}
=\frac{(\alpha+\beta, \alpha+\beta)^2}{(\alpha, \alpha)^2}.
\]
The killing form of $\{x_{\alpha}\mid \alpha\in \Phi\}$ is given by \cite[Page 147]{H}:
\[
(x_{\alpha}|x_{-\alpha})=\frac{2}{(\alpha, \alpha)}. 
\]
Choose $X_{\alpha}:=\sqrt{ \frac{(\alpha, \alpha)}{2}} x_{\alpha}$, for all $\alpha\in \Phi$. Thus, we have $(X_{\alpha}|X_{-\alpha})=1$.
We compute
\begin{align*}
&\sum_{\alpha, \beta\in \Phi}
\Big((\alpha, \beta)^2-(\alpha, \alpha)(\beta, \beta)\Big)
\Big([X_{\beta},  X_{-\alpha}] \mid [X_{-\beta}, X_{\alpha}]\Big)\\
=&\sum_{\alpha, \beta\in \Phi}
\Big((\alpha, \beta)^2-(\alpha, \alpha)(\beta, \beta)\Big)(\frac{(\alpha, \alpha)}{2}\frac{(\beta, \beta)}{2} )
\Big([x_{\alpha},  x_{\beta}] \mid [x_{-\alpha}, x_{-\beta}]\Big)\\
=&-\sum_{\{\alpha, \beta\in \Phi \mid \alpha+\beta\in \Phi\}}
\Big((\alpha, \beta)^2-(\alpha, \alpha)(\beta, \beta)\Big)(\frac{(\alpha, \alpha)}{2}\frac{(\beta, \beta)}{2} )
c_{\alpha\beta}^2\frac{2}{(\alpha+\beta, \alpha+\beta)}\\
=&\sum_{\{\alpha, \beta\in \Phi \mid \alpha+\beta\in \Phi\}}
\Big((\alpha, \alpha)(\beta, \beta)-(\alpha, \beta)^2\Big)\frac{(\beta, \beta)}{2} 
\frac{(\alpha+\beta, \alpha+\beta)}{(\alpha, \alpha)}\\
=&\frac{1}{4}\sum_{\{\alpha, \beta\in \Phi \mid \alpha+\beta\in \Phi\}}
\Big((\alpha, \alpha)(\beta, \beta)-(\alpha, \beta)^2\Big)(\frac{(\beta, \beta)}{(\alpha, \alpha)}+\frac{(\alpha, \alpha)}{(\beta, \beta)}
(\alpha+\beta, \alpha+\beta)\\
=&\frac{1}{4}\sum_{\{\alpha, \beta\in \Phi \mid \alpha+\beta\in \Phi\}}
\Big(1-\frac{(\alpha, \beta)^2}{ (\alpha, \alpha)(\beta, \beta)}\Big)\Big((\beta, \beta)^2+(\alpha, \alpha)^2\Big)
(\alpha+\beta, \alpha+\beta)
\end{align*}
The above computation gives the following formula of $\widetilde{C}$:
\begin{equation}
\label{eq:tildeC explicit}
\widetilde{C}=\frac{\sum_{\{\alpha, \beta\in \Phi \mid \alpha+\beta\in \Phi\}}
\Big(1-\frac{(\alpha, \beta)^2}{ (\alpha, \alpha)(\beta, \beta)}\Big)\Big((\beta, \beta)^2+(\alpha, \alpha)^2\Big)
(\alpha+\beta, \alpha+\beta)}{4\dim\h  (\dim\h-1)}.
\end{equation}
Therefore, $C=\frac{\lambda^2}{4}\widetilde{C}$ is given by the formula \eqref{table:constant C}. 

\subsection{The simply laced case}
When $\g$ is simply-laced, we have $\alpha+ \beta\in \Phi$, if and only if $(\alpha, \beta)=-1$. Thus,
\[
\widetilde{C}=\frac{3 | \{\alpha, \beta\in \Phi \mid \alpha+\beta\in \Phi\} |}{   \dim\h(\dim\h-1)  }, \,\ \text{where $\g$ is of type ADE.}
\]

For example, in type $A_{n-1}$, let $\g=\sl{n}$, we have
$\widetilde{C}
=\frac{3n(n-1) 2(n-2)}{ (n-1)(n-2) }=6n.$
\Omit{
\textcolor{blue}{Recompute the following: I might accidentally use $2h^{\vee}=|\Phi^+|$, which is not true. }
In type $D_n$, we have
\begin{align*}
\widetilde{C}
=&\frac{3h^\vee | \{\alpha, \beta\in \Phi \mid \alpha+\beta\in \Phi\} |}{ 2  (\dim\h-1) |\Phi^+| }
=\frac{3(2n-2)  \times 48 {n \choose 3}}{ 2  (n-1)  2 {n\choose 2} }=24(n-2).
\end{align*}

In type $E_{7}$, we have
\begin{align*}
\widetilde{C}=&\frac{3h^\vee | \{\alpha, \beta\in \Phi \mid \alpha+\beta\in \Phi\} |}{ 2  (\dim\h-1) |\Phi^+| }
  =\frac{3 \times 18 \left(2{8\choose 2}\times (12+{6\choose 3})+{8\choose 4} (16+16)\right)}{ 2  (7-1) ({8\choose 2}+\frac{1}{2}{8\choose 4}) }\\
  =&\frac{3 \times 18 \left(1792+2240 \right)}{ 12 \times 63 }
  =288.
\end{align*}

In type $E_{6}$, we have
\begin{align*}
\widetilde{C}=&\frac{3h^\vee | \{\alpha, \beta\in \Phi \mid \alpha+\beta\in \Phi\} |}{ 2  (\dim\h-1) |\Phi^+| }
  =\frac{3\times 12 \Big(2{6\choose 2}(8+12)+2 {6\choose 3}20+2\cdot 20\Big)}{ 2  (6-1) 36 }
  =144.
  \end{align*}

In type $E_{8}$, we have
\begin{align*}
\widetilde{C}=&\frac{3h^\vee | \{\alpha, \beta\in \Phi \mid \alpha+\beta\in \Phi\} |}{ 2  (\dim\h-1) |\Phi^+| }
  =\frac{90 \cdot  240\cdot 56}{ 2  (8-1) 120 }
  =720.
  \end{align*}
  }
  
\subsection{Non-simply laced case}

Denote by $\Phi_{l}$ the set of the long roots in $\Phi$, and $\Phi_{s}$ the set of the short roots.
We have the decomposition $\Phi=\Phi_l\sqcup \Phi_s$.
\subsubsection{Type $B_n$}
When $\g$ is of type $B_n$, we have
\begin{itemize}
\item
For any roots $\alpha, \beta\in \Phi_{l}$ such that $\alpha+\beta$ is a root, then
$\alpha+\beta\in \Phi_l$, and $(\alpha, \beta)=-1$. 
\item
For any roots $\alpha, \beta\in \Phi_{s}$ such that $\alpha+\beta$ is a root, then
$\alpha+\beta\in \Phi_l$, and $(\alpha, \beta)=0$. 
\item
For any roots $\alpha\in \Phi_l, \beta\in \Phi_{s}$ such that $\alpha+\beta$ is a root, then
$\alpha+\beta\in \Phi_s$, and $(\alpha, \beta)=-1$. 
\end{itemize}
Using the above observations, we can simplify \eqref{eq:tildeC explicit}. We have
\begin{align*}
&\sum_{\{\alpha, \beta\in \Phi \mid \alpha+\beta\in \Phi\}}
\Big(1-\frac{(\alpha, \beta)^2}{ (\alpha, \alpha)(\beta, \beta)}\Big)\Big((\beta, \beta)^2+(\alpha, \alpha)^2\Big)
(\alpha+\beta, \alpha+\beta)\\
=&\sum_{\{\alpha, \beta\in \Phi_l \mid \alpha+\beta\in \Phi_l\}}12+\sum_{\{\alpha, \beta\in \Phi_s \mid \alpha+\beta\in \Phi_l\}}
4
+\sum_{\{\alpha\in \Phi_l, \beta\in \Phi_s \mid \alpha+\beta\in \Phi_s\}}\frac{5}{2}+
\sum_{\{\alpha\in \Phi_s, \beta\in \Phi_l \mid \alpha+\beta\in \Phi_s\}}
\frac{5}{2}\\
=&12|\{\alpha, \beta\in \Phi_l \mid \alpha+\beta\in \Phi_l\}|+4
|\{\alpha, \beta\in \Phi_s \mid \alpha+\beta\in \Phi_l|+|\frac{5}{2}\{\alpha\in \Phi_l, \beta\in \Phi_s \mid \alpha+\beta\in \Phi_s\}|
\end{align*}
Therefore, 
\[
\widetilde{C}=
\frac{12\#_{\{\alpha, \beta\in \Phi_l \mid \alpha+\beta\in \Phi_l\}}+4\#_{\{\alpha, \beta\in \Phi_s \mid \alpha+\beta\in \Phi_l\}}
+\frac{5}{2}\#_{\{\alpha\in \Phi_l, \beta\in \Phi_s \mid \alpha+\beta\in \Phi_s\}}}{4n  (n-1)}.
\]

\Omit{Thus, in the case of $B_n$, we have:
\begin{align*}
\widetilde{C}=&\frac{ 3h^\vee\cdot 8n(n-1)(n-2)+h^\vee\cdot 4n(n-1)+2h^\vee\cdot 4n(n-1) }{h^\vee n(n-1)}
=12(2n-3).
\end{align*}}

\subsubsection{Type $C_n$}

The root system of type $C_n$ is $\Phi=\{\frac{1}{\sqrt{2}} (\pm \epsilon_i\pm\epsilon_j), \sqrt{2}\epsilon_i \mid 1\leq i\neq j \leq n\}$ (see for example \cite[Section 12.1]{H}), where $\{\epsilon_i\}_{i=1}^n$ is the standard orthonormal basis of $\mathbb{R}^n$. 
\begin{itemize}
\item
For any roots $\alpha, \beta\in \Phi_{l}$, the sum $\alpha+\beta$ can never be a root.
\item
For any roots $\alpha\in \Phi_l, \beta\in \Phi_{s}$ such that $\alpha+\beta$ is a root, then
$\alpha+\beta\in \Phi_s$, and $(\alpha, \beta)=-1$. 
\item
For any roots $\alpha, \beta\in \Phi_{s}$ such that $\alpha+\beta$ is a root, then we have two possibilities. 
If the sum $\alpha+\beta$ is a short root, then we have $(\alpha, \beta)=-\frac{1}{2}$. If the sum $\alpha+\beta$ is a long root, then $(\alpha, \beta)=0$.\end{itemize}
Using the above observation, we have
\begin{align*}
&\sum_{\{\alpha, \beta\in \Phi \mid \alpha+\beta\in \Phi\}}
\Big(1-\frac{(\alpha, \beta)^2}{ (\alpha, \alpha)(\beta, \beta)}\Big)\Big((\beta, \beta)^2+(\alpha, \alpha)^2\Big)
(\alpha+\beta, \alpha+\beta)\\
=&\sum_{\{\alpha\in \Phi_l, \beta\in \Phi_s \mid \alpha+\beta\in \Phi_s\}} \frac{5}{2}
+\sum_{\{\alpha\in \Phi_s, \beta\in \Phi_s \mid \alpha+\beta\in \Phi_s\}} \frac{3}{2}
+\sum_{\{\alpha\in \Phi_s, \beta\in \Phi_s \mid \alpha+\beta\in \Phi_l\}} 4
\end{align*}
Therefore, 
\[
\widetilde{C}
=
\frac{\frac{5}{2}\#_{\{\alpha\in \Phi_l, \beta\in \Phi_s \mid \alpha+\beta\in \Phi_s\}}
+\frac{3}{2}\#_{\{\alpha\in \Phi_s, \beta\in \Phi_s \mid \alpha+\beta\in \Phi_s\}}
+4\#_{\{\alpha\in \Phi_s, \beta\in \Phi_s \mid \alpha+\beta\in \Phi_l\}}}
{4n  (n-1)}.
\]

\subsubsection{Type $F_4$}
The root system is 
$\Phi=\{\pm \epsilon_i, \pm \epsilon_i\pm \epsilon_j, 
\frac{1}{2}(\pm \epsilon_1\pm \epsilon_2\pm \epsilon_3\pm \epsilon_4)\mid 1\leq i\neq j 
\leq 4\}$ (see for example \cite[Section 12.1]{H}), where $\{\epsilon_i\}_{i=1}^4$ is the standard orthonormal basis of $\mathbb{R}^4$. 

\begin{itemize}
\item
For any roots $\alpha, \beta\in \Phi_{l}$, we have $\alpha+\beta\in \Phi_l$, and $(\alpha, \beta)=-1$. 
\item
For any roots $\alpha\in \Phi_l, \beta\in \Phi_{s}$ such that $\alpha+\beta$ is a root, then 
$\alpha+\beta\in \Phi_s$ and $(\alpha, \beta)=-1$. 
\item
For any roots $\alpha, \beta\in \Phi_{s}$ such that $\alpha+\beta$ is a root. We have two possibilities. 
If the sum  $\alpha+\beta$ is a short root, then $(\alpha, \beta)=-\frac{1}{2}$. If the sum $\alpha+\beta$ is a long root, then $(\alpha, \beta)=0$. 
\end{itemize}
Using the above observation, we have
\begin{align*}
&\sum_{\{\alpha, \beta\in \Phi \mid \alpha+\beta\in \Phi\}}
\Big(1-\frac{(\alpha, \beta)^2}{ (\alpha, \alpha)(\beta, \beta)}\Big)\Big((\beta, \beta)^2+(\alpha, \alpha)^2\Big)
(\alpha+\beta, \alpha+\beta)\\
=&\sum_{\{\alpha, \beta\in \Phi_l \mid \alpha+\beta\in \Phi_l\}}12
+\sum_{\{\alpha\in \Phi_l, \beta_s\in \Phi \mid \alpha+\beta\in \Phi_s\}}\frac{5}{2}
+\sum_{\{\alpha, \beta\in \Phi_s \mid \alpha+\beta\in \Phi_s\}}\frac{3}{2}
+ \sum_{\{\alpha, \beta\in \Phi_s \mid \alpha+\beta\in \Phi_l\}}4
\end{align*}
Therefore, 
\[
\widetilde{C}=\frac{12\#_{\{\alpha, \beta\in \Phi_l \mid \alpha+\beta\in \Phi_l\}}
+\frac{5}{2}\#_{\{\alpha\in \Phi_l, \beta_s\in \Phi \mid \alpha+\beta\in \Phi_s\}}
+\frac{3}{2}\#_{\{\alpha, \beta\in \Phi_s \mid \alpha+\beta\in \Phi_s\}}
+4\#_{\{\alpha, \beta\in \Phi_s \mid \alpha+\beta\in \Phi_l\}}}{48}.
\]
\subsection{Type $G_2$}
Consider the root system of $G_2$. It is given as in the following picture. (see for example \cite[Section 12.1]{H}), where $\{\epsilon_1, \epsilon_2, \epsilon_3 \}$ is the standard orthonormal basis of $\mathbb{R}^3$. 
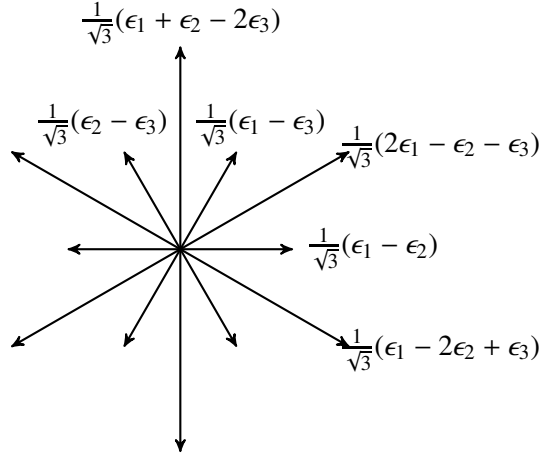
\[
\begin{tikzpicture}[scale=1.5]
 \coordinate (y) at (0,0);
 \coordinate (x) at (0,0);
 \draw[thick][->] (0,0)--(1,0);
  \draw[thick][->] (0,0)--(-1,0);
 \draw[thick][->] (0,0)--(0,1.8);
    \draw[thick][->] (0,0)--(0,-1.8);
 \draw[thick][->] (0, 0)--(0.5, sin 60);
  \draw[thick][->] (0, 0)--(-0.5, sin 60);
   \draw[thick][->] (0, 0)--(0.5, -sin 60);
  \draw[thick][->] (0, 0)--(-0.5, -sin 60);
 \draw[thick][->] (0, 0)--(1.5, sin 60);
 \draw[thick][->] (0, 0)--(-1.5, -sin 60);
  \draw[thick][->] (0, 0)--(-1.5, sin 60);
 \draw[thick][->] (0, 0)--(1.5, -sin 60);
 \draw (1.7, 0) node {$\frac{1}{\sqrt 3}(\epsilon_1-\epsilon_2)$};
 \draw (2.3, 0.9) node {$\frac{1}{\sqrt 3}(2\epsilon_1-\epsilon_2-\epsilon_3)$};
 \draw (0, 2) node {$\frac{1}{\sqrt 3}(\epsilon_1+\epsilon_2-2\epsilon_3)$};
 \draw (0.7, 1.1)node {$\frac{1}{\sqrt 3}(\epsilon_1-\epsilon_3)$};
 \draw (-0.7, 1.1)node {$\frac{1}{\sqrt 3}(\epsilon_2-\epsilon_3)$};
 \draw (2.3, -0.9) node {$\frac{1}{\sqrt 3}(\epsilon_1-2\epsilon_2+\epsilon_3)$};
 \end{tikzpicture} 
 \]
 \begin{itemize}
\item
If $\alpha, \beta\in \Phi_l$, and $\alpha+\beta\in \Phi_l$, then $(\alpha, \beta)=-1$.
\item
If $\alpha\in \Phi_l$, $\beta\in \Phi_s$, and $\alpha+\beta\in \Phi_s$, then $(\alpha, \beta)=-1$.
\item
If both $\alpha$, $\beta$ are short roots, and $\alpha+\beta\in \Phi_l$, then we have $(\alpha, \beta)=\frac{1}{3}$.
\item
If both $\alpha$, $\beta$ are short roots, and $\alpha+\beta\in \Phi_s$, then we have $(\alpha, \beta)=\frac{-1}{3}$.
\end{itemize}
Using the above observations, we can simplify the following.
\begin{align*}
&\sum_{\{\alpha, \beta\in \Phi \mid \alpha+\beta\in \Phi\}}
\Big(1-\frac{(\alpha, \beta)^2}{ (\alpha, \alpha)(\beta, \beta)}\Big)\Big((\beta, \beta)^2+(\alpha, \alpha)^2\Big)
(\alpha+\beta, \alpha+\beta)\\
=&\sum_{\{\alpha, \beta\in \Phi_l \mid \alpha+\beta\in \Phi_l\}}12
+\sum_{\{\alpha\in \Phi_l \beta\in \Phi_s \mid \alpha+\beta\in \Phi_s\}}\frac{5}{2}
+\sum_{\{\alpha, \beta\in \Phi_s \mid \alpha+\beta\in \Phi_l\}}\frac{32}{9}
+\sum_{\{\alpha, \beta\in \Phi_s \mid \alpha+\beta\in \Phi_s\}}\frac{16}{9}.
\end{align*}
Therefore, 
\[
\widetilde{C}=\frac{
12\#_{\{\alpha, \beta\in \Phi_l \mid \alpha+\beta\in \Phi_l\}}
+\frac{5}{2}\#_{\{\alpha\in \Phi_l \beta\in \Phi_s \mid \alpha+\beta\in \Phi_s\}}
+\frac{32}{9}\#_{\{\alpha, \beta\in \Phi_s \mid \alpha+\beta\in \Phi_l\}}
+\frac{16}{9}\#_{\{\alpha, \beta\in \Phi_s \mid \alpha+\beta\in \Phi_s\}}
}{8}.
\]

\newcommand{\arxiv}[1]
{\textsf{\href{http://arxiv.org/abs/#1}{arXiv:#1}}}
\newcommand{\doi}[1]
{\textsf{\href{http://dx.doi.org/#1}{doi:#1}}}
\renewcommand{\MR}[1]{}

\end{document}